\documentclass[11pt]{article}
\usepackage[english]{babel}
\usepackage{amsmath,amssymb,amsfonts,amsthm,mathrsfs,bm}
\usepackage{mathtools}
\usepackage{graphicx,xcolor}
\usepackage{enumerate,geometry,marginnote,afterpage}
\usepackage{fancyhdr,courier,type1cm,dsfont,float,tikz}
\usepackage{hyperref}
\usepackage[nameinlink,noabbrev,capitalise]{cleveref}
 \crefname{theorem}{Theorem}{Theorems}
 \crefname{prop}{Proposition}{Propositions}
 \crefname{lemma}{Lemma}{Lemmas}
 \crefname{remark}{Remark}{Remarks}
 \crefname{example}{Example}{Examples}
 \crefname{corollary}{Corollary}{Corollaries}

\AddToHook{env/corollary/begin}{\crefalias{theorem}{corollary}}
\AddToHook{env/lemma/begin}{\crefalias{theorem}{lemma}}
\AddToHook{env/prop/begin}{\crefalias{theorem}{prop}}
\AddToHook{env/definition/begin}{\crefalias{theorem}{definition}}
\AddToHook{env/assumption/begin}{\crefalias{theorem}{assumption}}
\AddToHook{env/remark/begin}{\crefalias{theorem}{remark}}
\AddToHook{env/example/begin}{\crefalias{theorem}{example}}
\AddToHook{env/conjecture/begin}{\crefalias{theorem}{conjecture}}
\AddToHook{cmd/appendix/before}{\crefalias{section}{appendix}}

\crefname{definition}{Definition}{Definitions}
\Crefname{definition}{Definition}{Definitions}
\crefname{assumption}{Assumption}{Assumptions} 
\Crefname{assumption}{Assumption}{Assumptions}
\crefname{conjecture}{Conjecture}{Conjectures}
\Crefname{conjecture}{Conjecture}{Conjectures}
\crefname{appendix}{Appendix}{Appendices}

\usetikzlibrary{arrows, positioning}

\usepackage{enumitem}
\usepackage{graphicx} 
\usepackage{tikz}
\usepackage{longtable}
\usepackage{microtype}
\usepackage{ulem}
\usepackage[comma,sort, numbers]{natbib}

\definecolor{vry}{RGB}{253, 231, 37}

\definecolor{vrg}{RGB}{94,201,98}

\definecolor{vrdg}{RGB}{33, 145, 140}

\definecolor{vrb}{RGB}{59,82,139}

\definecolor{vrp}{RGB}{68,1,84}
 
\definecolor{vro}{RGB}{249,142,9}

\definecolor{vrr}{RGB}{188,55,84}

\definecolor{vrnb}{RGB}{13,8,135}

\definecolor{vrw}{RGB}{240,240,240}

\definecolor{vrbg}{RGB}{252, 244, 182}

\definecolor{vrdr}{RGB}{159, 42, 99}

\hypersetup{colorlinks=true, citecolor=vrr, linkcolor=vrb, urlcolor=vrb}

\renewcommand{\le}{\leqslant}
\renewcommand{\ge}{\geqslant}
\renewcommand{\leq}{\leqslant}
\renewcommand{\geq}{\geqslant}

\newcommand{\E}[1]{\mathbb{E}\left[#1\right]}
\newcommand{\Var}[1]{\mathrm{Var}\left[#1\right]}
\newcommand{\Cov}[2]{\mathrm{Cov}\left[#1,#2\right]}
\newcommand{\R}{\mathbb{R}}
\newcommand{\N}{\mathbb{N}}
\newcommand{\Gn}{G_n}

\newcommand{\sign}{\sigma_n}
\newcommand{\Wass}{\mathrm{Wass}}
\newcommand{\Ber}{\mathrm{Ber}}

\newcommand{\NH}[1]{N({#1},\Gn)}
\newcommand{\NHK}[1]{N({#1},\mathcal G_K(\mathcal Z_n))}

\usepackage{textcase} 
\usepackage{filemod} 
\usepackage{titlesec} 
\usepackage{etoolbox} 
\usepackage{booktabs} 
\usepackage[linewidth=1pt]{mdframed} 
\usepackage{comment} 

\newtheorem{theorem}{Theorem}[section]
\newtheorem{lemma}[theorem]{Lemma}
\newtheorem{corollary}[theorem]{Corollary}
\newtheorem{prop}[theorem]{Proposition}
\newtheorem{remark}[theorem]{Remark}

\newtheorem{definition}[theorem]{Definition}
\newtheorem{example}[theorem]{Example}
\numberwithin{equation}{section}
\theoremstyle{definition}

\tikzset{
    motif/.style={baseline=-0.5ex, scale=0.8},
    node/.style={draw=black!55, circle, fill=black!55, inner sep=0pt, minimum size=2.2pt},
    edge/.style={line width=0.55pt, draw=black!55}
}

\DeclareRobustCommand{\Ktwo}{%
\tikz[motif]{
    \node[node] (1) at (0,0) {};
    \node[node] (2) at (0.4,0) {};
    \draw[edge] (1) -- (2);
}}

\DeclareRobustCommand{\Pthree}{%
\tikz[motif]{
    \node[node] (1) at (0,0.2) {};
    \node[node] (2) at (0.2,-0.1) {};
    \node[node] (3) at (0.4,0.2) {};
    \draw[edge] (1) -- (2) -- (3);
}}

\DeclareRobustCommand{\Konethree}{%
\tikz[motif]{
    \node[node] (center) at (0,0) {};
    \node[node] (top) at (0,0.25) {};
    \node[node] (left) at (-0.2,-0.15) {};
    \node[node] (right) at (0.2,-0.15) {};
    \draw[edge] (center) -- (top);
    \draw[edge] (center) -- (left);
    \draw[edge] (center) -- (right);
}}

\DeclareRobustCommand{\Kthree}{%
\tikz[motif]{
    \node[node] (1) at (-0.2,-0.15) {};
    \node[node] (2) at (0.2,-0.15) {};
    \node[node] (3) at (0,0.2) {};
    \draw[edge] (1) -- (2) -- (3) -- (1);
}}

\DeclareRobustCommand{\Pfour}{%
\tikz[motif]{
    \node[node] (1) at (-0.3,0.15) {};
    \node[node] (2) at (-0.1,-0.15) {};
    \node[node] (3) at (0.1,0.15) {};
    \node[node] (4) at (0.3,-0.15) {};
    \draw[edge] (1) -- (2) -- (3) -- (4);
}}

\DeclareRobustCommand{\Pfive}{%
\tikz[motif]{
    \node[node] (1) at (-0.4,0.15) {};
    \node[node] (2) at (-0.2,-0.15) {};
    \node[node] (3) at (0,0.15) {};
    \node[node] (4) at (0.2,-0.15) {};
    \node[node] (5) at (0.4,0.15) {};
    \draw[edge] (1) -- (2) -- (3) -- (4) -- (5);
}}

\DeclareRobustCommand{\Konefour}{%
\tikz[motif]{
    \node[node] (center) at (0,0) {};
    \node[node] (tl) at (-0.2,0.2) {};
    \node[node] (tr) at (0.2,0.2) {};
    \node[node] (bl) at (-0.2,-0.2) {};
    \node[node] (br) at (0.2,-0.2) {};
    \draw[edge] (center) -- (tl);
    \draw[edge] (center) -- (tr);
    \draw[edge] (center) -- (bl);
    \draw[edge] (center) -- (br);
}}

\DeclareRobustCommand{\Cfour}{%
\tikz[motif]{
    \node[node] (1) at (-0.15,-0.15) {};
    \node[node] (2) at (0.15,-0.15) {};
    \node[node] (3) at (0.15,0.15) {};
    \node[node] (4) at (-0.15,0.15) {};
    \draw[edge] (1) -- (2) -- (3) -- (4) -- (1);
}}

\DeclareRobustCommand{\Paw}{%
\tikz[motif]{
    \node[node] (1) at (-0.2,0.15) {};
    \node[node] (2) at (-0.2,-0.15) {};
    \node[node] (3) at (0,0) {};
    \node[node] (4) at (0.25,0) {};
    \draw[edge] (1) -- (2) -- (3) -- (1);
    \draw[edge] (3) -- (4);
}}

\DeclareRobustCommand{\Tthreetwo}{%
\tikz[motif]{
    \node[node] (tl) at (-0.2,0.2) {};
    \node[node] (bl) at (-0.2,-0.2) {};
    \node[node] (center) at (0,0) {};
    \node[node] (mid) at (0.25,0) {};
    \node[node] (right) at (0.5,0) {};
    \draw[edge] (tl) -- (center) -- (bl);
    \draw[edge] (center) -- (mid) -- (right);
}}

\title{  Third and Fourth Moment Phenomena for Monochromatic Edges and Applications to Graph-Based Two-Sample Tests }

\author{Siva Athreya \and Bhaswar B.~Bhattacharya \and Yogesh Dahiya \and Karambir Das \and
Spandan Ghoshal \and Sourish Maniyar \and Elizabeth Sara Roy}

\date{}

\begin{document} 
\maketitle

\begin{abstract}  
Given a graph sequence $\{\Gn\}_{n\in\N}$, denote by $T_n$ the number of monochromatic edges when the vertices of $G_n$ are independently colored red or blue with probabilities $p$ and $1-p$, respectively. We characterize the asymptotic normality of $T_n$, after centering and scaling to unit variance, in terms of subgraph (motif) counts and standardized moments. One of the highlights of our results is a {\it third moment phenomenon}: for every fixed $p\in(0,1)\backslash\{\frac12\}$, asymptotic normality holds if and only if the standardized third moment converges to zero. The more familiar {\it fourth moment condition}, which requires the standardized fourth moment to converge to $3$, gives an equivalent characterization except at the two exceptional probabilities $p=\frac{1\pm1/\sqrt3}{2}$. At these values, the standardized fourth moment converges to $3$ for every graph sequence with a diverging number of edges, even when Gaussian convergence fails. In contrast, in the balanced case $p=\frac12$, the fourth moment condition characterizes asymptotic normality, while the third moment condition is insufficient. These results identify three distinct regimes for moment characterizations of Gaussian convergence, determined by the coloring probability. We further show that the third moment condition remains sufficient in the sparse regime, where  $p=p_n \to 0$. As an application, we derive the limiting null distribution of two-sample statistics based on undirected $K$-nearest-neighbor graphs in the Poissonized setting, identifying sharp growth conditions on $K$ for asymptotic normality.  
\end{abstract}

\small  

  

\section{Introduction}

Let $\{\Gn\}_{n\in\N}$ be a sequence of simple, undirected graphs with
vertex set $[n] := \{1,\ldots,n\}$ and edge set $E(\Gn)$. Suppose the
vertices of $\Gn$ are independently colored red with probability $p
\in (0,1)$ and blue with probability $1-p$. Let $T_n$ be the number of
monochromatic edges in $E(\Gn)$ and $U_n$ be its standardized version.
Despite being elementary the statistic $U_n$ exhibits a rich range of limiting
behaviors and arises in a wide range of problems across combinatorial probability and statistical inference.  In the balanced (uniform) case $p=\frac12$, it can be expressed as a second-order Rademacher chaos, that is, a quadratic form in independent Rademacher variables (see \eqref{eq:Tnquadratic}).  This is a prototypical example of a degenerate $U$-statistic of order two
\cite{dependenceindependence,jansongaussian,ustatistics}, which arises in many contexts, such as the Hamiltonian of Ising models \cite{bmpotts,chatterjee}, tests of independence in autoregressive models \cite{beran_dependence}, and Ramsey theory \cite{anitconcentrationramsey,localclt}. The unbalanced (biased) case $p\ne\frac12$ also arises naturally in nonparametric two-sample
testing based on geometric graphs (see \cref{sec:knn} for details)
and in motif estimation problems \cite{kw2018,subgraph}.

Since the statistic $U_n$ is a sum of locally dependent indicators
(see \eqref{eq:Un}), sufficient conditions for its asymptotic
normality can be obtained using standard versions of Stein's method or
other normal approximation techniques. Identifying necessary and
sufficient conditions, however, is considerably more delicate. For
graph sequences with $|E(G_n)|\to\infty$, asymptotic normality in the
balanced case $p=\frac12$ is characterized by a {\it fourth moment
  phenomenon} \cite{BDM2017,rademacher,quadratic}:
$$U_n \stackrel{D}{\rightarrow} N(0,1) \quad\text{if and only if}\quad
\E{U_n^4} \rightarrow 3.$$
The unbalanced case, despite its practical relevance, has received considerably less attention.

In this paper, we develop a general framework for characterizing the
asymptotic normality of $U_n$ for every fixed $p\in(0,1)$ and
providing sufficient conditions in the regime when $p=p_n \rightarrow
0$. Our results reveal a new {\it third moment phenomenon}: for every
fixed $p\ne\frac12$, asymptotic normality is equivalent to
$\E{U_n^3}\to0$. A particularly striking feature is the emergence of
the two exceptional probabilities $p=\frac{1\pm1/\sqrt3}{2}$. At these
values, the fourth moment converges to its Gaussian value even when
asymptotic normality fails (see \cref{example:fourth}), whereas the
third moment continues to characterize Gaussian convergence. Outside
these exceptional values, the fourth moment condition
also characterizes asymptotic normality (see \cref{thm:Un}). At $p=\frac{1}{2}$, (see \cref{example:third}), 
we provide an example to show that the third moment condition is
insufficient. In the sparse regime, (see
\cref{ppn:thirdfourthsparse}, \cref{ppn:clthalf_sparse},
\cref{example:thirdsparse} and \cref{example:sparsenormal}), where $p=p_n\to0$, as $n\to\infty$ we
show that the third moment condition is sufficient for asymptotic
normality of $U_n$. Interestingly, the third moment condition is
strictly weaker than the fourth moment condition in this regime,
although neither is necessary for asymptotic normality. In addition for the fixed $p$ regime we provide
necessary and sufficient conditions in terms of motif counts for the
third and fourth moment phenomena to hold (see \cref{thm:Un}). We
establish a similar result for the sparse regime, prescribing a sufficient
decay rate of $p_n$ in terms of motif counts
(see \cref{lem:fourthmomentsparse}) .

Next, using the above results we derive the asymptotic null
distribution of two-sample tests based on undirected
$K$-nearest-neighbor ($K$-NN) graphs in the Poissonized
setting. Allowing $K$ to grow with the sample size has been often
recommended to improve power for graph-based nonparametric tests
\cite{ChenFriedman2017,generalgraph}. Our results provide sharp
general growth conditions that justify normal calibration in this
regime. Specifically, we show that under the null hypothesis, the test
statistic is asymptotically normal for arbitrary $K$ when the limiting
sample proportions are unequal, and for $K=o(n)$ when the sample sizes
are equal (see \cref{thm:cltH0} for precise statement). The same framework also gives conditions for asymptotic
normality of two-sample edge-count statistics based on general
similarity graphs. 

The rest of the introduction is organized as follows. In \cref{sec:Un}
we define the model and statistics precisely and state our main
results on the third and fourth moment phenomena for monochromatic
edges. We then show asymptotic properties of graph based two-sample
tests in \cref{sec:knn}. We conclude in \cref{sec:proofoutline} by
discussing the novelty and significance of our main results in
relation to the existing literature and outlining the proofs of these
results.

\subsection{Third and Fourth Moment Phenomena for Monochromatic Edges}

\label{sec:Un}   

Let $\{\Gn\}_{n\in\N}$ be a sequence of simple, undirected graphs with vertex set $[n] := \{1,\ldots,n\}$ and edge set $E(\Gn)$. Let $ d_{G_n}(i)$ will be the degree of the $i$-th vertex of $G_n$, for $1 \leq i \leq n$. Suppose the vertices of $\Gn$ are independently colored red with probability $p \in (0,1)$ and blue with probability $1-p$. For $1\le i\le n$, denote by $X_i$ the color of the vertex $i$, with $X_i =1$ for red and $X_i=0$ for blue. Our object of interest in this paper is the number of monochromatic edges in $\Gn$, that is, 
\begin{equation}\label{eq:Tn}
    T_n := \sum_{\{i,j\}\in E(\Gn)} \bm{1} \{X_i = X_j \}. 
\end{equation}
Denote the centered and rescaled version of $T_n$ by: 
\begin{equation} \label{eq:Un}
    U_n := \frac{T_n - \E{T_n}}{\sqrt{\Var{T_n}}}  .  
\end{equation}

Throughout, for a finite graph $H$ (a {\it motif}), we will denote by $\NH{H}$ the number of copies of $H$ in $\Gn$. We will use pictograms for the motifs, so for instance $\NH{\Ktwo}$ will denote the number of edges in $G_n$ and $\NH{\Cfour}$ will denote the number of $4$-cycles on $G_n$. Recall that $X_i$ is the color of the vertex $i$, for $1\le i\le n$. Hence, 
\begin{equation*}
   \E{\bm{1}\{X_i=X_j\}} = p^2+(1-p)^2 = \tfrac{1}{2}(1+u)  ,  
\end{equation*}
where $u := (2p-1)^2 \in [0,1)$. Then, from \eqref{eq:Tn}, 
$$\E{T_n} = \tfrac{1}{2}(1+u) \NH{\Ktwo} .$$
Further, since the indicators of a pair of edges being monochromatic are independent when the edges are disjoint, a direct computation gives,  
\begin{equation}\label{eq:var_term}
    \Var{T_n} = \frac{1-u^2}{4} \NH{\Ktwo} + \frac{u-u^2}{2} \NH{\Pthree}  . 
\end{equation}
Note that the second term in \eqref{eq:var_term} vanishes when $u=0$, that is, $p = \frac{1}{2}$. Consequently, the leading order of the variance changes depending on whether or not $p = \frac{1}{2} $.  This dichotomy is reflected throughout our results, with the balanced case often exhibiting different asymptotic behavior.

For two positive sequences $\{a_n\}_{n\in\N}$ and $\{b_n\}_{n\in\N}$ we write $a_n=o(b_n)$ if $a_n/b_n\to0$, $a_n=O(b_n)$ if $a_n/b_n$ is bounded, and $a_n\asymp b_n$ if both $a_n=O(b_n)$ and $b_n=O(a_n)$. All asymptotics are taken as $n\to\infty$.  We separate our results on the asymptotic normality of $U_n$ in two regimes: the fixed $p$ regime, in which $p\in(0,1)$ is a fixed constant, and the sparse regime, in which $p=p_n\to0$, as $n\to\infty$.
 
\paragraph{The Fixed $p$ Regime:} The following theorem characterizes the asymptotic normality of $U_n$ in terms of its third and fourth moments and the growth rates of certain motif counts, when the coloring probability $p\in(0,1)$ is fixed.. 

\begin{theorem}\label{thm:Un}
Let $\{\Gn\}_{n\in\N}$ be a sequence of graphs such that $\NH{\Ktwo} \to \infty$ and $U_n$ be as defined in \eqref{eq:Un}. Suppose $ p \in (0, 1)$ is fixed. Then the following hold:   
\begin{enumerate}
    
\item[$(i)$] If $p \in (0,1) \backslash \{\frac{1}{2}, \frac{1\pm 1/\sqrt{3}}{2} \}$, then the following are equivalent:
\begin{itemize}
\item[$(a)$] $U_n \stackrel{D} \rightarrow N(0, 1)$. 
\item[$(b)$] $\NH{\Konefour} = o((\NH{\Ktwo}+\NH{\Pthree})^2)$.
\item[$(c)$] $\E{U_n^4} \to 3$.
\item[$(d)$] $\E{U_n^3} \to 0$.  
\end{itemize}

\item[$(ii)$] If $p = \frac{1\pm 1/\sqrt{3}}{2}$, then the following are equivalent:
\begin{itemize}
\item[$(a)$] $U_n \stackrel{D} \rightarrow N(0, 1)$. 
\item[$(b)$] $\NH{\Konefour} = o((\NH{\Ktwo}+\NH{\Pthree})^2)$.  
\item[$(c)$] $\E{U_n^3} \to 0$. 
\end{itemize}

\item[$(iii)$] If $p = \frac{1}{2}$, then the following are equivalent:
\begin{itemize}
\item[$(a)$] $U_n \stackrel{D} \rightarrow N(0, 1)$. 
\item[$(b)$] $\NH{\Cfour} = o(\NH{\Ktwo}^2)$.
\item[$(c)$] $\E{U_n^4} \to 3$. 
\end{itemize}

\end{enumerate}

\end{theorem}

 The theorem identifies three distinct regimes for the asymptotic
 normality of $U_n$, distinguished by the exceptional probabilities
 $p=\frac{1\pm1/\sqrt{3}}{2}$ and the balanced case
 $p=\frac{1}{2}$. Detailed remarks on the three
 regimes, on the novelty of  the third and fourth moment phenomena and an
 outline of the proofs are given in \cref{sec:proofoutline}. As remarked
 earlier for the two exceptional probabilities
 $p=\frac{1\pm1/\sqrt3}{2}$ in \cref{example:fourth}, we construct a
 graph sequence, ($n$-star $K_{1,n}$), for which the fourth moment
 converges to its Gaussian value even when asymptotic normality fails
 and for $p=\frac{1}{2}$, in \cref{example:third} we consider a graph
 sequence, (the complete bipartite $K_{2,n}$), for which the third moment
 condition is insufficient. The proof of \cref{thm:Un} is given in
 \cref{sec:Unpf}.

\paragraph{The Sparse $p$ Regime:} Next, we consider the sparse regime $p_n\to0$. In this regime the third moment condition turns out to be sufficient for the asymptotic normality of $U_n$.

\begin{theorem}\label{ppn:clthalf_sparse}
  Let $\{\Gn\}_{n\in\N}$ be a sequence of graphs such that $\NH{\Ktwo} \to \infty$ and $U_n$ be as defined in \eqref{eq:Un}. Suppose $ p=p_n \to 0$, as $ n \rightarrow \infty$. Then 
$$\E{U_n^3} \to 0  \implies U_n \stackrel{D} \rightarrow N(0, 1).$$
\end{theorem}

The theorem shows that in the sparse regime, the third moment condition is sufficient for asymptotic normality of $U_n$.  Our next result states that  the fourth moment condition always implies the third moment condition in this regime.
\begin{prop}\label{ppn:thirdfourthsparse}
Let $\{\Gn\}_{n\in\N}$ be a sequence of graphs such that $\NH{\Ktwo}\to\infty$, and let $U_n$ be as defined in \eqref{eq:Un}. Suppose $p=p_n\to0$ as $n\to\infty$.
Then
$$
\E{U_n^4}\to3 \implies \E{U_n^3}\to0.
$$
\end{prop} 
We conclude by stating a  lemma that  characterizes the fourth moment condition in the sparse regime in terms of motif counts.

\begin{lemma}\label{lem:fourthmomentsparse}
Let $\{\Gn\}_{n\in\N}$ be a sequence of graphs such that $\NH{\Ktwo} \to \infty$, and let $U_n$ be as defined in \eqref{eq:Un}. Suppose $p=p_n\to0$ as $n\to\infty$.
Then the following are equivalent:
\begin{enumerate}
\item[$(a)$] $\E{U_n^4} \to 3$.
\item[$(b)$] $p_n(\NH{\Ktwo}+\NH{\Pthree})\to\infty$ and $\NH{\Konefour}=o(p_n(\NH{\Ktwo}+\NH{\Pthree})^2)$.
\end{enumerate}
\end{lemma}

As remarked earlier, we show in \cref{example:thirdsparse} that there are graph sequences, (consisting of an $n^2$-star and $n^5$ disjoint edges, $p_n = \frac{1}{n^3}$), for which $U_n$ is asymptotically normal, even though its third moment does not tend to zero. In \cref{example:sparsenormal} we provide a graph sequence, (an $n$-star and $n^3$ disjoint edges, $p_n = \frac{1}{n^2}$),  for which the third moment condition holds, and hence $U_n$ is asymptotically normal, while the fourth moment condition fails. Hence the implication in \cref{ppn:thirdfourthsparse} is  strict and that the third moment condition is strictly weaker in this regime.  The proof of \cref{ppn:clthalf_sparse} is given in \cref{sec:Unpf} while \cref{ppn:thirdfourthsparse} and \cref{lem:fourthmomentsparse} are shown in \cref{sec:thftmph}.

\subsection{Asymptotic Properties of Graph Based Two-Sample Tests} 
\label{sec:knn}

In this section, we will derive sharp conditions for  the asymptotic normality and consistency of two-sample tests based on nearest-neighbor graphs. We begin by formally describing the two-sample problem: Given two probability measures $P$ and $Q$ in $\R^d$, the two-sample problem is to test the hypotheses: 
$$H_0 : P = Q \qquad \text{against} \qquad H_1 : P \neq Q,$$
based on independent samples from the respective distributions. Throughout, we will assume $P$ and $Q$ have densities $f$ and $g$ with respect to the Lebesgue measure on $\R^d$.

\begin{definition}\label{def:knn}
Fix $K \geq 1$. Given a finite set $S \subset \R^{d}$ with pairwise distinct distances, the undirected $K$-nearest neighbor graph ($K$-NN) is the graph with vertex set $S$ and an edge
$\{a,b\}$ between distinct points $a,b\in S$ if $b$ is one of the
$K$ nearest neighbors of $a$ in $S\backslash\{a\}$ or $a$ is one of
the $K$ nearest neighbors of $b$ in $S\backslash\{b\}$.
If $|S|\le K$, we adopt the convention that the graph is complete.
Denote the undirected $K$-NN graph of $S$ by $\mathcal G_K(S)$.
\end{definition}

Here, we will consider the two-sample problem in the Poissonized setting, where instead of taking a fixed number $n_{1}$ and $n_2$ samples from the two distributions, we have $\mathrm{Poi}(n_{1})$
from $f$ and $\mathrm{Poi}(n_{2})$ samples from $g$. Specifically, suppose
$\mathcal X=\{X_{1}, X_{2}, \ldots\}$ and $\mathcal Y=\{Y_{1}, Y_{2}, \ldots\}$
are independent collections of i.i.d. samples from $f$ and $g$, respectively. Define 
\begin{align}
\mathcal X_{n_{1}}=\{X_{1}, X_{2}, \ldots, X_{L_{n_{1}}}\}\quad \text{and}\quad \mathcal Y _{n_{2}}= \{Y_{1}, Y_{2}, \ldots, Y_{L_{n_{2}}}\},
\label{2poisson}
\end{align}
where $L_{n_{1}}\sim\mathrm{Poi}(n_{1})$ and $L_{n_{2}} \sim\mathrm{Poi}
(n_{2})$ are independent of each other, and of $\mathcal X$ and $\mathcal Y$.
Poissonization is a common assumption in this literature \cite{Bhattacharya2020,Kanekar2025,distributioncomparison}, which
simplifies calculations, due to the spatial independence of the Poisson
process, and yields cleaner formulas for the asymptotic variances.  Then the Poissonized two-sample statistic based on the $K$-NN graph is defined as follows:\footnote{Since $\mathcal X_{n_{1}}$ and $\mathcal Y_{n_{2}}$ are independent observations from distributions that are absolutely continuous with respect to Lebesgue measure, all the interpoint distances of the set $\mathcal X_{n_{1}}\cup \mathcal Y_{n_{2}}$ are distinct almost surely.} 
\begin{align}
\label{graph2testp} 
T^{\times}(\mathcal G_K(\mathcal X_{n_{1}}\cup
\mathcal Y_{n_{2}})):=\sum
_{i=1}^{L_{n_{1}}} \sum_{j=1}^{L_{n_{2}}}
\pmb1\{ \{X_{i}, Y_{j}\} \in E(\mathcal G_K(
\mathcal X_{n_{1}}\cup \mathcal Y_{n_{2}})
)\}.
\end{align}
The distribution of this statistic in \eqref{graph2testp} can be equivalently described as follows: Let
$\phi_n(x):=\frac{n_{1}}{n} f(x)+ \frac{n_{2}}{n} g(x)$ and
$Z_{1}, Z_{2}, \ldots,$ be independent random variables with common
density $\phi_n(\cdot)$ and $L_{n}$ be an independent Poisson variable with mean $n = n_{1}+n_{2}$. Then 
\begin{align}\label{eq:Zn}
\mathcal Z_{n}  =  \{  Z_{1}, Z_{2}, \ldots, Z_{L_{n}} \}  ,  
\end{align}  
is a nonhomogeneous Poisson process in $\R^{d}$
with rate function $n_{1} f+n_{2} g$. Label each point  
$z\in\mathcal Z_{n}$ independently with
\begin{align*}  
c_{z}= 
\begin{cases} 1 &
\text{with probability } \frac{n_{1} f(z)}{n_{1} f(z)+n_{2}
g(z)},
\\
2 & \text{with probability } \frac{n_{2} g(z)}{n_{1} f(z)+n_{2} g(z)}. \end{cases}
\end{align*} 
Then the sets of points assigned labels 1 and 2 have the same
distribution as $\mathcal X_{n_{1}}$ and $\mathcal Y_{n_{2}}$ (as in \eqref{2poisson}), respectively. Therefore, \eqref{graph2testp} is equal in distribution to
\begin{align}\label{eq:Tpoissonc}
T^{\times}(\mathcal G_K(\mathcal Z_{n})) & \stackrel{D} = \sum_{\{x, y\} \in\mathcal Z_{n}} \bm 1\{c_{x} \ne c_{y}\} \bm 1\{\{x, y\}\in E(\mathcal G_K(\mathcal Z_{n}))\}  \\  
& = |E(\mathcal G_K(\mathcal Z_{n}))|  -  T^{\circ}(\mathcal G_K(\mathcal Z_{n}))  ,       \nonumber  
\end{align}
where the sum is over unordered pairs $\{x, y\}$, with $ x, y \in\mathcal Z_{n}$, and $|E(\mathcal G_K(\mathcal Z_{n}))|=\NHK{\Ktwo}$ is the number of edges in $\mathcal G_K(\mathcal Z_{n})$ and 
\begin{align}
\label{eq:Tpoissonm} 
T^{\circ}(\mathcal G_K(\mathcal Z_{n})) = \sum_{ \{ x, y \} \in\mathcal Z_{n}} \bm 1\{c_{x} = c_{y}\} \bm 1\{\{x, y\}\in E(\mathcal G_K(\mathcal Z_{n}))\}  , 
\end{align}
is the number of monochromatic edges in $\mathcal G_K(\mathcal Z_{n})$. 
Let $\mathscr F=\sigma(\mathcal Z_n)$ be the sigma-algebra generated by $L_n$ and the pooled sample in \eqref{eq:Zn}. The edge count $|E(\mathcal G_K(\mathcal Z_n))|$ is $\mathscr F$-measurable, since it is determined by the pooled sample without reference to the markings. Consequently, conditional on $\mathscr F$, tests based on the statistics $T^{\times}(\mathcal G_K(\mathcal Z_{n}))$ and $T^{\circ}(\mathcal G_K(\mathcal Z_{n}))$ are equivalent: rejecting for small values of $T^{\times}(\mathcal G_K(\mathcal Z_n))$ corresponds to rejecting for large values of $T^{\circ}(\mathcal G_K(\mathcal Z_n))$.

Under the null hypothesis $H_0$, let $f$ denote the common density. Then $\mathcal Z_n=\{Z_1,\ldots,Z_{L_n}\}$ is a Poisson process on $\R^d$ with intensity function $nf$. Independently of the point locations, each point receives label $1$ with probability $\lambda_n:= \frac{n_1}{n}$ and label $2$ with probability $1-\lambda_n= \frac{n_2}{n}$. By independent thinning, $\mathcal X_{n_1}$ and $\mathcal Y_{n_2}$ are independent Poisson processes with intensity functions $n_1f$ and $n_2f$, respectively. Thus, conditional on $\mathcal Z_n$, $\mathcal G_K(\mathcal Z_n)$ is a fixed graph with an i.i.d.\ $\Ber(\lambda_n)$ vertex coloring, reducing the problem to the setting considered in the previous section. Specifically, conditional on $\mathscr F$, the expectation and variance of $T^{\times}(\mathcal G_K(\mathcal Z_{n}))$ under $H_0$ are given by:  
$$\mu_{\mathscr F}^\times := \E{T^{\times}(\mathcal G_K(\mathcal Z_{n}))| \mathscr F} = 2 \lambda_n(1-\lambda_n) \NHK{\Ktwo} $$
and, recalling \eqref{eq:var_term}, 
\begin{align}\label{eq:varianceT}
\sigma_{\mathscr F}^2 := \Var{T^{\times}(\mathcal G_K(\mathcal Z_{n}))  | \mathscr {F} } & = \Var{T^{\circ}(\mathcal G_K(\mathcal Z_{n}))  | \mathscr {F} }  \nonumber  \\   
& = \frac{1-\nu_n^2}{4} N(\Ktwo,\mathcal G_K(\mathcal Z_n)) + \frac{\nu_n-\nu_n^2}{2} N(\Pthree,\mathcal G_K(\mathcal Z_n))  ,  
\end{align}
where $\nu_n := (2 \lambda_n-1)^2$. Define 
\begin{equation}\label{eq:Un_pois}
U_n^\times(\mathcal G_K(\mathcal Z_{n}))  := \frac{T^{\times}(\mathcal G_K(\mathcal Z_{n})) - \mu_{\mathscr F}^\times }{\sigma_{\mathscr F} },
\end{equation}
where we put $U_n^\times(\mathcal G_K(\mathcal Z_{n}))  := 0$ on the event $\{\sigma_{\mathscr F}^2 = 0\} = \{L_n\le 1\}$, whose probability tends to zero. The following result establishes the asymptotic normality of $U_n^\times(\mathcal G_K(\mathcal Z_{n}))$ under $H_0$, with convergence in the Wasserstein distance.

\begin{theorem}\label{thm:cltH0}
Let $U_n^\times(\mathcal G_K(\mathcal Z_{n})) $ be as defined above in \eqref{eq:Un_pois}. Then the following hold under $H_0$: 
\begin{itemize} 

\item[$(i)$]  Suppose $\lambda_n= \frac{n_1}{n}  \rightarrow \lambda  \in  (0, 1) \backslash \{ \frac{1}{2}\}$. Then $\Wass(U_n^\times(\mathcal G_K(\mathcal Z_{n})) , N(0,1) ) \rightarrow 0$, 
for any sequence $K_n \ge 1$. 

\item[$(ii)$]  Suppose $n_1 = n_2$. Then 
$\Wass(U_n^\times(\mathcal G_K(\mathcal Z_{n})) , N(0,1) ) \rightarrow 0$, for any $K \geq 1$ fixed or any sequence $K_n \rightarrow \infty$ with $K_n = o(n)$.     

\end{itemize} 
\end{theorem}

 The result distinguishes two regimes.  For unbalanced sample sizes,
 in $(i)$, with $\frac{n_1}{n} \to\lambda\in(0,1)\backslash\{\frac12\}$,
 the standardized null test statistic converges to $N(0, 1)$ in
 Wasserstein distance for every sequence $K_n\ge1$, including fixed
 $K$ and sequences diverging at any rate (see \eqref{eq:Wp} for an
 explicit convergence rate). For equal sample sizes, in (ii), the same
 conclusion holds for every fixed $K$ and every growing sequence
 $K_n=o(n)$. Together, these conclusions give sharp growth conditions
 as general guarantees of asymptotic normality. In particular, the
 sublinear restriction in the balanced case cannot generally be
 removed. For instance, in the balanced coloring of the complete graph the centered statistic is a degenerate $U$-statistic of order two and has a non-Gaussian limit
 given by a scaled, centered chi-square variable (see
 \cite[Example~7.1]{BDM2017}). The proof of \cref{thm:cltH0} is given in \cref{sec:proofs-knn}.

Although we have so far focused on $K$-NN graphs for illustration, the preceding approach extends almost verbatim to two-sample edge-count tests based on general similarity graphs. Specifically, suppose $G_n$ is any simple graph constructed on the pooled sample $\mathcal Z_n$ without using the sample labels. Define $T^\times(G_n)$ and $U_n^\times(G_n)$ as in \eqref{eq:Tpoissonc} and \eqref{eq:Un_pois}, respectively, with $\mathcal G_K(\mathcal Z_n)$ replaced by $G_n$. The formulas for the conditional mean and variance remain unchanged under this replacement. Under $H_0$, conditional on $\mathscr F$, the vertex labels are independent $\Ber(\lambda_n)$ variables. Consequently, the quantitative bounds provided by \cref{thm:clt} and \cref{ppn:clthalf} in \cref{sec:wass} give the following general conditions for asymptotic normality for any similarity graph $G_n$ constructed as above with $\NH{\Ktwo}\stackrel{P}{\rightarrow}\infty$: 
\begin{itemize}

\item If $\lambda_n\to\lambda\in(0,1)\backslash\{\frac12\}$, then the $\frac{\NH{\Konefour}}{(\NH{\Ktwo}+\NH{\Pthree})^2}\stackrel{P}{\rightarrow}0$ implies $\Wass(U_n^\times(G_n),N(0,1))\to0$.

\item If $n_1=n_2$, the same conclusion holds when $\frac{\NH{\Cfour}}{\NH{\Ktwo}^2}\stackrel{P}{\rightarrow}0$. 
\end{itemize}

Examples of commonly used similarity graphs beyond $K$-NNs, include $K$-minimum spanning trees ($K$-MSTs) \cite{friedman1979,ChenFriedman2017} and $K$-minimum-distance matchings ($K$-matchings) \cite{ZhouChen2023}. A $K$-MST is the union of $K$ successively constructed minimum spanning trees, with previously selected edges excluded at each step. Similarly, on an even number of observations, a $K$-matching is the union of $K$ successively constructed minimum-weight perfect matchings, each minimizing the total dissimilarity between paired observations subject to the exclusion of previously selected edges. The case $K=1$ gives the graph underlying Rosenbaum's cross-match test \cite{rosenbaum2005}. For the Poissonized sample, if $L_n$ is odd, we omit one observation using a rule that does not depend on the sample labels. If $K$ exceeds the number of retained observations minus one, so that $K$ edge-disjoint perfect matchings do not exist, we adopt the convention that $G_n$ is the complete graph on the retained observations, as in \cref{def:knn}. Each retained vertex then has degree $\min\{K, |V(G_n)|-1\}$, hence, the relevant graph conditions for asymptotic normality can be easily verified. This leads to the following corollary, which follows by the same argument as in the proof of \cref{thm:cltH0}:

\begin{corollary}\label{cor:matching}
Let $G_n$ be the $K$-matching graph described above and define $U_n^\times(G_n)$ as in \eqref{eq:Un_pois}. Then the following hold under $H_0$:
\begin{itemize}

\item[$(1)$] Suppose $\lambda_n=\frac{n_1}{n}\to\lambda\in(0,1)\backslash\{\frac12\}$. Then
$\Wass(U_n^\times(G_n),N(0,1))\to0$,  for any sequence $K_n\ge1$.

\item[$(2)$] Suppose $n_1=n_2$. Then
$\Wass(U_n^\times(G_n),N(0,1))\to0$,
for any fixed $K\ge1$ or any sequence $K_n\to\infty$ with $K_n=o(n)$.

\end{itemize}
\end{corollary}

The feature that makes two-sample tests based on geometric graphs
particularly attractive is their universal consistency, that is, their
power converges to one against every fixed alternative for which $f$
and $g$ differ on a set of positive Lebesgue measure. This is
classically known for the $K$-NN-based test when $K$ is fixed
\cite{Schilling1986} (see also \cite{HenzePenrose1999}). Although the
benefits of allowing $K$ to grow with the sample size have often been
advocated \cite{ChenFriedman2017,generalgraph}, consistency in the
diverging-$K$ regime was established only recently by
\citet[Proposition 2.1]{Kanekar2025} for the directed $K$-NN test when
$K=o(n)$. In the appendix of
this paper, we prove universal consistency for the undirected $K$-NN
test under the same condition (see \cref{ppn:KNNconsistency}). We also illustrate the finite-sample performance of the test through
simulations in \cref{sec:experiments}.

\subsection{Remarks and Overview of Proofs}
\label{sec:proofoutline}

The fourth moment phenomenon was first discovered by \citet{NuPe05},
who showed that the convergence of the first, second, and fourth
moments to $0, 1$, and $3$, respectively, guarantees asymptotic
normality for a sequence of multiple stochastic Wiener-It\^o integrals
of fixed order.  Later, \citet{NoPe09,moment4_np} provided error
bounds for the fourth moment theorem. This phenomenon has since
emerged as a general principle governing Gaussian limits of nonlinear
functionals of random fields (see the book \cite{book_np} for an
introduction to the topic and the website
\url{https://sites.google.com/site/malliavinstein/home} for a
comprehensive list of the recent results). Related results for
degenerate $U$-statistics of fixed order were obtained earlier by de
Jong \cite{dejong87,de90}. In this setting, the fourth moment
condition generally requires an additional assumption controlling the
maximum influence of the underlying independent variables (see
\cite[Theorem~2.1]{dejong87}, \cite[Theorem~1]{de90}, and
\cite[Corollary~1.2]{DoeblerKrokowski2019}). For other classical
sufficient conditions for asymptotic normality and convergence rates
for quadratic forms, see
\cite{chatterjee_normal_approximation,gotze,gotze_applications,hall,rotar}
and the references therein. In the context of random graph coloring,
the fourth moment phenomenon was first observed for monochromatic edge
counts in the balanced case by \citet{BDM2017}, and later extended to
counts of more general monochromatic subgraphs in
\cite{color,fourthmomentH,fourthmomentinfluence} (see also
\cite{Doebler2024ProductSpaces,universalclt} for related results).

An instance of the third moment phenomenon was established by
\citet[Theorem~3.2]{NeufcourtViens2016} for normalized quadratic
variations of stationary Gaussian sequences under suitable covariance
assumptions. For these statistics, which belong to the second Wiener
chaos, both the third and fourth moment conditions characterize
asymptotic normality (see also \cite{KimPark2019} for related
results).  Our results, on the other hand, establish a third moment
characterization for monochromatic-edge counts on arbitrary graph
sequences under biased Bernoulli coloring.

\paragraph{The fixed $p$ regime:} We shall make remarks on each of the three cases of \cref{thm:Un}.

\begin{itemize}
\item For $p \in (0,1) \backslash \{\frac{1}{2}, \frac{1\pm
  1/\sqrt{3}}{2}\}$, the fourth moment characterization in
  \cref{thm:Un}$(i)$ is an instance of the broader phenomenon
  described above, but is new for biased graph colorings. The
  equivalent third moment characterization is a new phenomenon in its
  own right : convergence of an even lower-order moment suffices to
  determine Gaussian convergence.

\item At the exceptional probabilities $p=\frac{1\pm1/\sqrt3}{2}$,
  the coefficient of the $4$-star term in the fourth cumulant
  $\E{U_n^4}-3$ vanishes, while all remaining terms tend to zero (see
  \cref{lem:fm_pneqhalf}). The fourth moment therefore no longer
  detects the $4$-star configurations responsible for non-Gaussian
  behavior, which is why it can converge to $3$ while asymptotic
  normality fails, whereas the third moment, whose $4$-star
  coefficient does not vanish, continues to characterize Gaussian
  convergence.

\item The exceptional probabilities $p=\frac{1\pm1/\sqrt{3}}{2}$ have appeared before in \cite[Corollary~1.4 and Example~1.5]{DoeblerKrokowski2019}
in connection with fourth moment criteria for linear combinations
of independent biased Rademacher variables. Their significance
comes from the identity
$$\E{ \left( \frac{X-p}{\sqrt{p(1-p)}} \right)^4}-3 =\frac{1-6p(1-p)}{p(1-p)},$$
where $X\sim \mathrm{Ber}(p)$. Thus, precisely at these probabilities, the standardized Bernoulli variable has fourth moment $3$ despite being non-Gaussian. \citet{DoeblerKrokowski2019} establish a fourth moment criterion for
normalized linear combinations away from these values and show
that the exclusion is necessary. The same cancellation explains the exceptional probabilities in \cref{thm:Un}, at which the $4$-star contribution to the fourth
cumulant of $U_n$ vanishes. Our result identifies the exact graph condition for asymptotic normality of monochromatic edge counts, whose expansion in centered
variables contains both linear and quadratic terms under biased
colorings. At the exceptional probabilities, this condition is
equivalent to vanishing of the third moment, while the fourth moment condition holds automatically and therefore cannot characterize asymptotic normality. To the best of our knowledge, this provides the first example of a sequence of quadratic multilinear polynomials in independent biased Rademacher variables for which the third moment characterizes Gaussian convergence, even though the fourth moment criterion fails.   

\item Finally, in the balanced case $p=\frac{1}{2}$, the edge
  indicators $\{\bm{1}\{X_i=X_j\}\}_{\{i,j\}\in E(G_n)}$ are pairwise
  independent, and the centered statistic in \eqref{eq:Tn} is a
  Rademacher chaos of order two, that is, a quadratic form in
  i.i.d.\ standard Rademacher variables. Indeed, setting
  $R_i=2X_i-1\in\{\pm1\}$ for $1\leq i\leq n$, we obtain
\begin{align}\label{eq:Tnquadratic}
T_n-\E{T_n} =\frac{1}{4}\sum_{1\leq i,j\leq n}a_{ij}R_iR_j,
\end{align}
 where $a_{ij}=\bm{1}\{\{i,j\}\in E(G_n)\}$ and the variables $R_1,\ldots,R_n$ are i.i.d.\ standard Rademacher variables. As mentioned earlier, asymptotic normality of $U_n$ in this setting has been extensively studied, with \cref{thm:Un}~(iii) established in \cite[Theorem~1.4]{BDM2017} (see also \cite{rademacher,quadratic}). Here, the fourth moment condition, which requires the number of $4$-cycles to be negligible relative to its extremal order, is necessary and sufficient for asymptotic normality of $U_n$.  As discussed earlier, the third moment condition, however, is no longer sufficient. The third moment can converge to zero, even when asymptotic normality fails (see \cref{example:third}).    

\end{itemize}

We now provide a brief outline of the proofs in the fixed $p$ regime.  In \cref{sec:thftmph} we prove equivalence of the third and fourth moment
phenomena to their respective motif count conditions in
\cref{lem:thirdm_pneqhalf} and \cref{lem:fm_pneqhalf}. These lemma
follow by expanding the third and fourth centered moments of $T_n$ in
terms of subgraph counts and identifying the leading contributions in
each regime.  In  \cref{sec:wass}, we show convergence of $U_n$
to Normal distribution in the Wasserstein metric under the respective
motif count conditions proven in \cref{lem:thirdm_pneqhalf} and
\cref{lem:fm_pneqhalf}. For $p=\frac{1}{2}$ we reprove the well known
result in \cref{ppn:clthalf} and for unbalanced case in
\cref{thm:clt}$(i)$. The proof of \cref{thm:clt} is an application of
Stein's method and an overview is given in that section.  A uniform
integrability argument (see \cref{lem:momentUn}) ensures that
asymptotic normality of $U_n$ implies convergence of its third and
fourth moments to their Gaussian values.

We also note that, since we establish fourth or third-moment
phenomena, weak convergence is equivalent to convergence in the
Wasserstein or Kolmogorov distance. Finding the precise rate of
convergence in any metric remains an open question.

\paragraph{The sparse $p$ regime:}The proof strategy parallels the
fixed $p$ regime. In \cref{lem:thirdmomentsparse} we show that the
third moment condition is equivalent to a motif/$p_n$ condition which
is weaker than (that is, implied by) the motif/$p_n$ condition characterizing the fourth
moment condition in \cref{lem:fourthmomentsparse}; both lemmas are
proved in \cref{sec:thftmph}, and \cref{ppn:thirdfourthsparse} follows
immediately from them. In \cref{thm:clt}$(ii)$ we show convergence of
$U_n$ to the normal distribution in the Wasserstein metric under the
motif/$p_n$ condition of \cref{lem:thirdmomentsparse}, which yields
\cref{ppn:clthalf_sparse}. Although the overall strategy is the same
as for fixed $p$, it is considerably harder to execute, because the
moment bounds must now be controlled uniformly in the varying
$p_n$. This is explained in detail in \cref{sec:wass}.

\paragraph{Graph based two-sample tests:} 

Note that \cref{thm:cltH0} treats two cases: unbalanced sample sizes
with $\lambda_n=\frac{n_1}{n} \rightarrow \lambda\in(0,1) \backslash
\{\frac{1}{2}\}$, and equal sample sizes with $n_1=n_2$. It does not
cover the asymptotically balanced case, where $n_1\ne n_2$, but
$\lambda_n\to\frac12$. The conclusion of \cref{thm:cltH0} $(ii)$ remains
valid in this regime, but establishing this extension requires a more
careful argument, which we omit.

Recently, \citet{Kanekar2025} studied a related two-sample statistic
based on directed nearest-neighbor graphs in the Poissonized
setting. In particular, \cite[Theorem~3.1]{Kanekar2025} gives a central limit theorem
under $H_0$ for the corresponding test statistic in the diverging $K$
regime, under $K_n=o(n)$. For the undirected statistic
\eqref{eq:Un_pois} considered here, \cref{thm:cltH0} establishes
asymptotic normality for arbitrary $K_n$ in the unbalanced case
$\lambda\ne\frac12$, and for $K_n=o(n)$ in the balanced case
$n_1=n_2$. For related results on the asymptotic normality of
two-sample statistics based on general similarity graphs under the
fixed-sample-size permutation null, see \citet{ZhuChen2024}.

\subsection{Outline of the Paper}
\label{sec:outline}

In \cref{sec:thftmph} we express the third and fourth moments of $U_n$
in terms of motif counts and prove the equivalences between the
moment conditions and the motif conditions, in both the fixed $p$ and
the sparse regimes. In \cref{sec:wass} we prove convergence of $U_n$
to the normal distribution in the Wasserstein metric under these
motif conditions. \cref{sec:Unpf} combines these results to prove
\cref{thm:Un}, \cref{ppn:clthalf_sparse} and
\cref{ppn:thirdfourthsparse}, and \cref{sec:proofs-knn} contains the
proof of \cref{thm:cltH0}. The examples showing that the moment
conditions cannot be improved are collected in \cref{sec:examples},
and \cref{sec:experiments} reports simulations assessing the
finite-sample performance of the $K$-NN test. The appendix contains
the motif tables used in the proofs (\cref{app:tables}) and the proof
of universal consistency of the undirected $K$-NN test
(\cref{ppn:KNNconsistency}).

\paragraph{Acknowledgments:} SA was supported through a Knowledge Exchange Grant, under project
no. RTI4001 of the Department of Atomic Energy, Government of
India. BBB was supported by NSF CAREER grant DMS 2046393, a Sloan
Research Fellowship, and the National University of Singapore's PESS
and Provost's Chair funds. The authors thank Frank den Hollander and
Rajat Subhra Hazra who along with SA organized the ICTS-Networks Workshop in
April 06-17th, 2026 at the ICTS-TIFR in India, where the research in this paper
started. The authors also thank ICTS-TIFR for support and hospitality.

\paragraph{AI disclosure} 
AI tools (Claude Fable 5.1, GPT 5.6 Sol and GPT 6 Astra) were used to compute the motif tables, verify some of the mathematical arguments, and refine grammar and typography. The authors assume responsibility for all content.

\section{Convergence Criteria for Third and Fourth Moment of $U_n$}
\label{sec:thftmph}

In this section we obtain a characterization of the  third moment condition and fourth moment condition in all regimes in terms of motif counts. Let $(p_n)_{n\in\N}$ be a sequence such that $p_n\in (0,1)$, for all $n\in \N$. In the fixed $p$ regime $p_n = p \in (0, 1)$, for all $n\in \N$, and in the sparse regime $p_n \rightarrow 0$, as $n \rightarrow \infty$.
For each $e = \{i,j\} \in E(\Gn)$, denote  
\begin{equation}\label{eq:Ye}
    Y_e := \bm{1} \{X_i = X_j \} - q_n,
\end{equation}
where $q_n = \E{\bm{1} \{X_i = X_j \}} = p_n^2 + (1-p_n)^2$. Recall the definition of $T_n$  from \eqref{eq:Tn}. Denote 
\begin{equation}\label{eq:Sn}
    S_n := T_n - \E{T_n} = \sum_{e\in E(\Gn)} Y_e .
\end{equation}

\subsection{Criteria for Convergence of Third Moment of $U_n$}

First we compute the third moment of $S_n$. For this, expanding the sum in $S_n^3$ gives, 
\begin{align*}
    \E{S_n^3} = \E{\sum_{e_1,e_2,e_3 \in E(\Gn)}Y_{e_1}Y_{e_2}Y_{e_3}}.
\end{align*}
Since $\E{Y_e} = 0$ for all $e\in E(\Gn)$, and $Y_e$ is independent of $Y_{e'}$ whenever $e\cap e'=\varnothing$, the only non-vanishing summands are those for which each connected component of $e_1\cup e_2\cup e_3$ contains at least two edges, counting multiplicities. Grouping the summands according to the isomorphism type of
$e_1\cup e_2\cup e_3$, as in the proof of  \cref{lemma:Varlemma}, we can write
\begin{equation}\label{eq:H_moment_expansion}
    \E{S_n^3} = \sum_{H\in \mathcal H_3} c_H^{(3)}(u_n) \NH{H},
\end{equation}
where $u_n  = (2p_n - 1)^2$ and 
$\mathcal{H}_3:= \{\Ktwo,\Pthree,\Kthree, \Pfour,\Konethree\}$. 
 \cref{tab:third} in \cref{app:tables} gives the exact values of the coefficients $c_H^{(3)}(u_n)$ and the asymptotic order of the subgraph count associated with each motif.  Using this we   can characterize the third moment condition  $\E{U_n^3} \to 0$ in terms of the motif counts.

\paragraph{The Fixed $p$ Regime:}  We obtain equivalent condition for balanced case $p=\frac{1}{2}$ involving  the number of triangles and  number of edges, whereas in the unbalanced case we need to control the number of 3-stars w.r.t the number of edges and 2-stars.

\begin{lemma}\label{lem:thirdm_pneqhalf}
    Let $U_n$ be as defined in \eqref{eq:Un} and $\NH{\Ktwo} \rightarrow \infty$, as $n\to \infty$. Then the following hold. 
     \begin{enumerate}[label = $(\alph*)$]
        \item If $p \in (0, 1) \backslash \{\frac 12\}$, then
        \begin{equation} \label{eq:thirdmc_cond1}
            \E{U_n^3} \to 0  \quad \Longleftrightarrow \quad
            \NH{\Konethree} = o\left((\NH{\Ktwo}+\NH{\Pthree})^\frac{3}{2}\right).
        \end{equation}
            \item If $p = \frac{1}{2}$, then
        \begin{equation} \label{eq:thirdmc_phalf}
            \E{U_n^3} \to 0 \quad \Longleftrightarrow \quad
            \NH{\Kthree} = o\left(\NH{\Ktwo}^\frac{3}{2}\right).
        \end{equation}
    \end{enumerate}
\end{lemma}

\begin{proof}  
First, suppose $p \in (0,1) \backslash \{\frac{1}{2}\}$. In this case, recall \eqref{eq:varianceUn}, 
$$\sign^2\asymp \NH{\Ktwo}+\NH{\Pthree}.$$
Referring to the order of the motif counts in \cref{tab:third} it follows that, for all $H \in \mathcal H_3 \backslash \{\Pfour, \Konethree\}$, 
\begin{align}\label{eq:NHGnH}
N(H, G_n) = o\left((\NH{\Ktwo}+\NH{\Pthree})^{\frac{3}{2}}\right)  .  
\end{align}
For $H = \Pfour$, we use the following estimate: For any $t > 0$, 
\begin{align*}
N(H, G_n) \leq \sum_{\{i,j\}\in E(G_n)} d_{G_n}(i)  d_{G_n}(j) \le t\sum_{\{i,j\}\in E(G_n)}( d_{G_n}(i)+ d_{G_n}(j))
  +\frac12\left(\sum_{i:  d_{G_n}(i)>t} d_{G_n}(i)\right)^2 . 
\end{align*}
Denote $\delta_2:=\sum_{i=1}^n  d_{G_n}(i)^2$.  Note that $\sum_{\{i,j\}\in E(G_n)}( d_{G_n}(i)+ d_{G_n}(j))=\delta_2$ and $$\sum_{i:  d_{G_n}(i)>t} d_{G_n}(i) \le \frac1t\sum_{i:  d_{G_n}(i)>t} d_{G_n}(i)^2 \le \frac{\delta_2}{t}.$$
Then choosing $t=\delta_2^{\frac{1}{3}}$ gives, 
\begin{align}\label{eq:NHGnpath} 
\NH{\Pfour} \leq \frac{3}{2} \delta_2^{\frac{4}{3}} =o(\delta_2^{\frac{3}{2}}) = o\left((\NH{\Ktwo}+\NH{\Pthree})^\frac{3}{2}\right ) . 
\end{align}  
Combining \eqref{eq:NHGnH} and \eqref{eq:NHGnpath} with \eqref{eq:H_moment_expansion} and referring to  \cref{tab:third} gives, 
$$\E{U_n^3} = - \frac{3u^2(1-u)}{2}  \frac{\NH{\Konethree}}{\sigma_n^3}  + o(1),  $$
where $u = (2p-1)^2$. This proves the result in \eqref{eq:thirdmc_cond1}.

Next, suppose $p =\frac{1}{2}$, that is $u=0$. In this case, referring to  \cref{tab:third} gives $c_H^{(3)}(u)  = 0$, for all $H \in \mathcal H_3 \backslash \{\Kthree\}$. Hence, 
$$\E{U_n^3} = \frac{3}{4}  \frac{\NH{\Kthree}}{\sigma_n^3} .$$
This proves the result in \eqref{eq:thirdmc_phalf}, since, in this case, $\sign^2\asymp \NH{\Ktwo}$ (recall \eqref{eq:varianceUn}).  
\end{proof}

\paragraph{The Sparse $p$ Regime:} Here we show that the third moment condition is equivalent to a condition requiring that the rate of decay of $p_n$ and the number of $3$-stars compare suitably with the number of edges and $2$-stars.

\begin{lemma}\label{lem:thirdmomentsparse}
Let $\{\Gn\}_{n\in\N}$ be a sequence of graphs such that $\NH{\Ktwo} \to \infty$ and $U_n$ be as defined in \eqref{eq:Un}. Suppose $p= p_n \rightarrow 0$, as $n \rightarrow \infty$. 
Then the following are equivalent: 
\begin{enumerate}
\item[$(a)$] $\E{U_n^3} \to 0$. 
\item[$(b)$] $p_n(\NH{\Ktwo} + \NH{\Pthree}) \to \infty$ and $\NH{\Konethree} = o(p_n^{\frac{1}{2}}(\NH{\Ktwo} + \NH{\Pthree})^{\frac{3}{2}})$.  
\end{enumerate}    
\end{lemma}

\begin{proof} Recall that $u_n=(2p_n-1)^2$, so that $u_n=1+O(p_n)$ and $1-u_n=4p_n+O(p_n^2)$. Referring to \cref{tab:third},
    \begin{align*}
        \E{S_n^3} = -2p_n\NH{\Ktwo}  -6p_n\left(\NH{\Pthree}+\NH{\Konethree}\right) + \mathcal Q_n  ,  
        \end{align*}
        where 
  $$|\mathcal Q_n | = O\left(p_n^2\left(\NH{\Ktwo}+\NH{\Pthree}+\NH{\Konethree}+\NH{\Pfour}+\NH{\Kthree}\right)\right).$$
To control the error term, note that each copy of $\Pfour$ has a unique middle edge. Hence,
    \begin{align*}
        \NH{\Pfour} = \sum_{\{i,j\}\in E(\Gn)}( d_{G_n}(i)-1)( d_{G_n}(j)-1) &\le \sum_{\{i,j\}\in E(\Gn)} d_{G_n}(i)  d_{G_n}(j) \\
        &\le \frac12\sum_{\{i,j\}\in E(\Gn)}( d_{G_n}(i)^2+ d_{G_n}(j)^2) \\
        &= \frac12\sum_{i=1}^n  d_{G_n}(i)^3 \\
        &= \NH{\Ktwo}+3\NH{\Pthree}+3\NH{\Konethree}.
    \end{align*}
Moreover, $3\NH{\Kthree}\le\NH{\Pthree}$. Hence, 
$$| \mathcal Q_n | = O \left(p_n^2 \left(\NH{\Ktwo}+3\NH{\Pthree}+3\NH{\Konethree}\right) \right).$$ 
This implies, 
    \begin{align*}
        \E{S_n^3} =-2p_n\left(\NH{\Ktwo}+3\NH{\Pthree}+3\NH{\Konethree}\right) (1+O(p_n)).
    \end{align*}
    Also, \eqref{eq:var_term} gives $\sign^2=2p_n(\NH{\Ktwo}+\NH{\Pthree})(1+O(p_n)).$
    Therefore,
    \begin{equation*}
  \E{U_n^3} =-\frac{1+O(p_n)}{\sqrt{2}}  \frac{\NH{\Ktwo}+3\NH{\Pthree}+3\NH{\Konethree}}
        {p_n^{\frac12} (\NH{\Ktwo}+\NH{\Pthree})^{\frac32}}.
    \end{equation*}
    This proves the result in \cref{lem:thirdmomentsparse}.  
     \end{proof}

\subsection{Criteria for Convergence of Fourth Moment of $U_n$}

Next, we consider the fourth moment of $S_n$. Expanding as above gives, 
\begin{align}\label{eq:H_moment_expansionfourth}
    \E{S_n^4} = \E{\sum_{e_1,e_2,e_3,e_4\in E(\Gn)}Y_{e_1}Y_{e_2}Y_{e_3}Y_{e_4}}.
\end{align}
As before, grouping the summands according to the isomorphism type of
$e_1\cup e_2\cup e_3\cup e_4$, as in the proof of  \cref{lemma:Varlemma}, we can write
\begin{equation}\label{eq:fourth_moment_expansion}
    \E{S_n^4} = \sum_{H\in \mathcal H_4} c_H^{(4)}(u_n) \NH{H},
\end{equation}
where
$\mathcal{H}_4:= \{\Ktwo,\Pthree,\Ktwo\sqcup\Ktwo,\Kthree, \Pfour,\Konethree,\Ktwo\sqcup\Pthree,\Pfive,\Konefour,\Tthreetwo,\Cfour,\Paw,\Pthree\sqcup\Pthree\}$. 
 \cref{tab:fourth} in \cref{app:tables} gives the exact values of the coefficients $c_H^{(4)}(u_n)$ and the asymptotic order of the subgraph count associated with each motif. Using this we   can characterize the fourth moment condition  $\E{U_n^4} \to 3$ in terms of the motif counts.

\paragraph{The Fixed $p$ Regime:}  We obtain that as long as the edges diverge the fourth moment convergence holds for the exceptional values  $p = \frac{1\pm 1/\sqrt{3}}{2}$. We obtain equivalent conditions for balanced case $p=\frac{1}{2}$ involving  the number of squares and  number of edges, whereas in the unbalanced case outside of the exceptional values  we need to control the number of 4-stars w.r.t the number of edges and 2-stars.

\begin{lemma}\label{lem:fm_pneqhalf}
    Let $U_n$ be as defined in \eqref{eq:Un} and $\NH{\Ktwo} \rightarrow \infty$, as $n\to \infty$. Then the following hold. 
    \begin{enumerate}[label = (\alph*)]
        \item If $p \in (0, 1) \backslash \{\frac 12, \frac{1\pm 1/\sqrt{3}}{2}\}$, then
        \begin{equation} \label{eq:fmc_cond1}
            \E{U_n^4} \to 3  \quad \Longleftrightarrow \quad
            \NH{\Konefour} = o\left((\NH{\Ktwo}+\NH{\Pthree})^2\right).
        \end{equation}
        \item If $p = \frac{1\pm 1/\sqrt{3}}{2}$, then, without any additional assumption,
        \begin{equation*}
            \E{U_n^4} \to 3 .
        \end{equation*}
        \item If $p = \frac{1}{2}$, then
        \begin{equation} \label{eq:fmc_phalf}
            \E{U_n^4} \to 3 \quad \Longleftrightarrow \quad
            \NH{\Cfour} = o\left(\NH{\Ktwo}^2\right).
        \end{equation}
    \end{enumerate}
\end{lemma}

\begin{proof}
First, suppose $p \in (0,1) \backslash \{\frac{1}{2}\}$. In this case, recall \eqref{eq:varianceUn}, 
$$\sign^2\asymp \NH{\Ktwo}+\NH{\Pthree}.$$
Referring to the order of the motif counts in \cref{tab:fourth} and by arguments as in  \cref{lemma:Varlemma}, it follows that, for all $H \in \mathcal H_4 \backslash \{\Ktwo\sqcup\Ktwo, \Ktwo\sqcup\Pthree, \Pthree\sqcup\Pthree, \Konefour\}$, 
\begin{align}
\label{eq:smallmotifs}
N(H, G_n) = o\left((\NH{\Ktwo}+\NH{\Pthree})^{2}\right)  .  
\end{align}
Also, observe the following identities: $\binom{\NH{\Ktwo}}{2} = \NH{\Ktwo\sqcup\Ktwo} + \NH{\Pthree}$,  where $\sqcup$ denotes the disjoint union of two graphs,  
$$\NH{\Ktwo}\NH{\Pthree} = \NH{\Ktwo\sqcup\Pthree} + 2\NH{\Pthree} + 3\NH{\Konethree} + 2\NH{\Pfour} + 3\NH{\Kthree},$$
    and
$$\binom{\NH{\Pthree}}{2} = \NH{\Pthree\sqcup\Pthree} + 3\NH{\Konefour} + o\left((\NH{\Ktwo}+\NH{\Pthree})^2\right),$$
    where the $o(\cdot)$ collects the motifs on at most four edges spanned by two overlapping wedges other
    than $\Konefour$ (namely $\Konethree$, $\Pfour$, $\Kthree$, $\Cfour$, $\Paw$, $\Tthreetwo$, $\Pfive$),
    all of which are $o((\NH{\Ktwo}+\NH{\Pthree})^2)$ by \eqref{eq:smallmotifs}. Combining the above with \eqref{eq:smallmotifs} and by arguments as in  \cref{lemma:Varlemma}, we obtain from \eqref{eq:fourth_moment_expansion},  
    \begin{align}\label{eq:momentSn}
        \E{S_n^4} &= 3\left(\frac{1-u^2}{4}\NH{\Ktwo} + \frac{u(1-u)}{2}\NH{\Pthree}\right)^2 \nonumber \\
        &\quad + \left(\frac{3u^2(1-u)(3u+1)}{2} - 3\cdot\frac{3u^2(1-u)^2}{2}\right)\NH{\Konefour}
        + o(\sign^4)  \nonumber  \\  
        &= 3\sign^4 + 3u^2(1-u)(3u-1)\NH{\Konefour} + o(\sign^4).
    \end{align}
    Dividing by $\sign^4$ and using $\sign^4 \asymp (\NH{\Ktwo}+\NH{\Pthree})^2$ (recall \eqref{eq:varianceUn}), 
$$\frac{\E{S_n^4}}{\sign^4} = 3 + \frac{3u^2(1-u)(3u-1)\NH{\Konefour}}{\sign^4} + o(1).$$   Thus, provided $u \neq \frac{1}{3}$, that is, $p \ne \frac{1\pm 1/\sqrt{3}}{2}$, the necessary and sufficient condition for $\E{U_n^4} \to 3$ is $\NH{\Konefour} = o\left((\NH{\Ktwo}+\NH{\Pthree})^2\right)$. This proves (a).

    Next, suppose $p = \frac{1\pm 1/\sqrt{3}}{2}$. Then $u = \frac{1}{3}$, and the coefficient of $\NH{\Konefour}$ in \eqref{eq:momentSn} vanishes and $\E{U_n^4} \to 3$ with no additional condition. This proves (b).

    Finally, suppose $p = \frac{1}{2}$, that is, $u=0$. In this case, the result follows from \cite[Proposition 1.11]{quadratic}. We will give a short proof here, for the sake of completeness. In particular, referring to \cref{tab:fourth} it follows that 
    $c_H^{(4)}(u) = 0$, for all $H \in \mathcal H_4 \backslash \{ \Ktwo, \Pthree, \Ktwo\sqcup\Ktwo, \Cfour\}$. Also, recall from \eqref{eq:varianceUn} that $\sign^2 = \frac 14 \NH{\Ktwo}$. Hence, by \cref{tab:fourth}, 
       \begin{align*}
        \E{S_n^4} &= \frac{1}{16}\NH{\Ktwo} + \frac 38 \NH{\Pthree}
        + \frac 38 \left(\binom{\NH{\Ktwo}}{2} - \NH{\Pthree}\right) + \frac 32 \NH{\Cfour}\\
        &= \frac{3\NH{\Ktwo}^2 - 2\NH{\Ktwo}}{16} + \frac 32 \NH{\Cfour} \nonumber \\ 
        & = 3\sign^4(1+o(1)) + \frac32 \NH{\Cfour}.
    \end{align*}
   This implies that $\E{U_n^4}\to3$ if and only if $\NH{\Cfour} = o(\NH{\Ktwo}^2)$. This proves (c).  
\end{proof}

\paragraph{The Sparse $p$ Regime:} We have already stated the result for the sparse regime in \cref{sec:Un}. We  prove \cref{lem:fourthmomentsparse}.

\begin{proof}[Proof of \cref{lem:fourthmomentsparse}:]
Recall that $u_n=(1-2p_n)^2$ and $1-u_n=4p_n+O(p_n^2)$. We first consider the disconnected motifs in \cref{tab:fourth}. Counting pairs of edges, pairs consisting of an edge and a wedge, and pairs of distinct wedges according to their unions gives (as in the proof of \cref{lem:fm_pneqhalf}),  $\binom{\NH{\Ktwo}}{2} =\NH{\Ktwo\sqcup\Ktwo}+\NH{\Pthree}$, where $\sqcup$ denotes the disjoint union of two graphs, 
$$\NH{\Ktwo}\NH{\Pthree} =\NH{\Ktwo\sqcup\Pthree}+2\NH{\Pthree}+3\NH{\Konethree}
+2\NH{\Pfour}+3\NH{\Kthree},$$
and
\begin{multline*}
\binom{\NH{\Pthree}}{2}
=\NH{\Pthree\sqcup\Pthree}+3\NH{\Konethree}
+\NH{\Pfour}+3\NH{\Kthree}\\
+3\NH{\Konefour}+\NH{\Pfive}+\NH{\Tthreetwo}
+2\NH{\Cfour}+2\NH{\Paw}.
\end{multline*}
Substituting these identities into \eqref{eq:fourth_moment_expansion} and using \eqref{eq:var_term}, the terms involving products of motif counts cancel exactly upon subtracting $3\sign^4$. Expanding the remaining coefficients around $u_n=1$ gives
\begin{multline*}
\E{S_n^4}-3\sign^4
=p_n(1+O(p_n))
\left(2\NH{\Ktwo}+14\NH{\Pthree}
+36\NH{\Konethree}+24\NH{\Konefour}\right)\\
+o\left(p_n^2(\NH{\Ktwo}+\NH{\Pthree})^2\right).
\end{multline*}
This follows by observing that the errors in the coefficients of the four displayed motifs are $O(p_n^2)$, which are absorbed into the factor $1+O(p_n)$. Each remaining connected motif has coefficient $O(p_n^2)$ and count $o((\NH{\Ktwo}+\NH{\Pthree})^2)$, by the combinatorial estimates used in the proof of \cref{lem:fm_pneqhalf}.

Also, \eqref{eq:var_term} gives, 
$$
\sign^2
=2p_n(\NH{\Ktwo}+\NH{\Pthree})(1+O(p_n)).
$$
Dividing the preceding fourth-moment expansion by $\sign^4$ and using the non-negativity of the motif counts, we obtain
\begin{equation}\label{eq:c}
\E{U_n^4}\to3
\quad\Longleftrightarrow\quad
\frac{\NH{\Ktwo}+\NH{\Pthree}+\NH{\Konethree}+\NH{\Konefour}}
{p_n(\NH{\Ktwo}+\NH{\Pthree})^2}\to0.
\end{equation}
The condition on the right immediately implies $(b)$. Conversely, separating vertices of degree $3$ from those of degree at least $4$ gives, 
$$
\NH{\Konethree}
\le \NH{\Pthree}+4\NH{\Konefour}  .  
$$
Hence,
\begin{align*}
& \frac{\NH{\Ktwo}+\NH{\Pthree}+\NH{\Konethree}+\NH{\Konefour}}
{p_n(\NH{\Ktwo}+\NH{\Pthree})^2}  \nonumber \\ 
& \leq \frac{2}{p_n(\NH{\Ktwo}+\NH{\Pthree})}
+\frac{5\NH{\Konefour}}
{p_n(\NH{\Ktwo}+\NH{\Pthree})^2}.
\end{align*} 
Under $(b)$, both terms on the right tend to zero. The equivalence in \eqref{eq:c} therefore proves the result.
\end{proof}

\section{Wasserstein Convergence of $U_n$}
\label{sec:wass}
In this section we establish general conditions for the asymptotic
normality of $U_n$ in Wasserstein metric. The results are key
ingredients in the proofs of \cref{thm:Un} and
\cref{ppn:clthalf_sparse}. For this, we begin by recalling the
definition of the Wasserstein distance.

\begin{definition} 
Given two real-valued random variables $X$ and $Z$ with finite first moments, the Wasserstein distance between $X$ and $Z$ is defined as
$$\mathrm{Wass}(X,  Z):=\sup_{g \in\mathcal{L}} \left|\E{g(X) }  - \E{g(Z)} \right|,$$
where $\mathcal L$ is the collection of 1-Lipschitz functions from $\R \rightarrow \R$. 
\end{definition} 

In the balanced case $p=\frac{1}{2}$, the edge indicators $\{\bm{1}\{X_i=X_j\}\}_{\{i,j\}\in E(G_n)}$ are pairwise independent, so the variance is determined solely by the number
of edges (recall \eqref{eq:var_term}). In this case, using a quantitative  martingale central limit theorem (CLT), \citet{color} established asymptotic normality of $U_n$ with convergence rates in the Kolmogorov distance. The following result, which shows convergence in the Wasserstein distance, is an easy consequence of the result.  The proof  also gives an explicit rate of convergence.

\begin{prop}\label{ppn:clthalf}
  Let $\{\Gn\}_{n\in\N}$ be a sequence of graphs with $\NH{\Ktwo} \to \infty$ and $U_n$ be as defined in \eqref{eq:Un}. Suppose $p = \frac{1}{2}$ and
     $$\NH{\Cfour} = o\big(\NH{\Ktwo}^2\big).$$ Then $ \Wass\left(U_n, N(0,1)\right) \to 0.$
\end{prop}

\begin{proof}
Define $$\mathcal R_n:= \NH{\Ktwo}^{-\frac{1}{2}}+\frac{\NH{\Cfour}}{\NH{\Ktwo}^2}.$$
By \cite[Theorem 1.4]{color}, applied with two colours,
\begin{align}\label{eq:N}
\sup_{x\in\mathbb R}| \mathbb P(U_n \leq x ) -\Phi(x)| \le C \mathcal R_n^{\frac{1}{5}},    
\end{align}
where $C > 0$ is a constant and $\Phi$ is the standard normal distribution function. We will now invoke the well known bound of the Wasserstein distance in terms of the Kolmogorov distance (See{\cite[Proposition~A.1.2]{BCG2016}}). That is, for two random variables $X$ and $Y$ with second moments bounded by $B^2$, 
\begin{equation} \label{eq:wkbd} \mathrm{Wass}(X, Y) \leq 4B \sqrt{\sup_{t \in \R} |\mathbb P(X \leq t) - \mathbb P(Y \leq t) | }. \end{equation}

Since $\mathbb E[U_n^2]=\mathbb E[Z^2]=1$, where $Z\sim N(0,1)$, \eqref{eq:N} and   \eqref{eq:wkbd} gives, 
\begin{align}\label{eq:WUnZhalf}
\operatorname{Wass}(U_n,Z)\le 4\sqrt{C} \mathcal R_n^{\frac{1}{10}}.  
\end{align}
Under the hypothesis, $\mathcal R_n \rightarrow 0$ and the result follows. 
\end{proof}

The following result gives sufficient conditions for convergence of $U_n$ to the standard normal distribution in the Wasserstein distance, both when $p\ne\frac{1}{2}$ is fixed and when $p=p_n\to0$.

\begin{theorem}\label{thm:clt}
    Let $\{\Gn\}_{n\in\N}$ be a sequence of graphs with $\NH{\Ktwo} \to \infty$ and $U_n$ be as defined in \eqref{eq:Un}. 
  \begin{itemize}
   \item [(i)] Suppose $p \in (0, 1)\backslash \{\frac{1}{2}\}$ is fixed and
    \begin{align}\label{eq:cltfixed}
    \NH{\Konefour} = o((\NH{\Ktwo} + \NH{\Pthree})^2).
    \end{align}  
    Then $\Wass(U_n,N(0,1))\to 0.$
      \item[(ii)]Suppose $p = p_n \rightarrow 0$ such that  $p_n(\NH{\Ktwo} + \NH{\Pthree}) \to \infty$ and 
\begin{align}\label{eq:cltsparse}
\NH{\Konethree} = o(p_n^{\frac{1}{2}}(\NH{\Ktwo} + \NH{\Pthree})^\frac{3}{2}). 
\end{align}
    Then $\Wass(U_n,N(0,1))\to 0.$ 
\end{itemize}

    \end{theorem} 

The proof of \cref{thm:clt} uses Stein's method with local dependence neighborhoods determined by edges sharing a vertex. The resulting Wasserstein bound consists of two terms: a cubic remainder and a fluctuation term involving local quadratic products. We present these preliminary preparation in \cref{sec:prepw}. Controlling these terms requires third and fourth-order moment estimates for the summands of $T_n$, which translate into counts of graph motifs with up to four edges. For fixed $p\ne\frac{1}{2}$, combinatorial comparisons between these counts show that the stated $4$-star condition makes both error terms vanish (see \cref{sec:thmclthf}). The proof also provides an explicit bound for the Wasserstein distance in terms of motif counts (see \cref{remark:clt}). The sparse regime is more delicate, as the moment estimates and variance normalization must retain their precise dependence on $p_n$. Here, we approximate this term by a degree-weighted sum of independent color indicators and control its mean absolute deviation by separating vertices of small and large degree. This reduces the required estimates to the stated $3$-star condition, which, together with divergence of the variance, ensures that both Stein error terms vanish (see \cref{sec:thmcltsp}).

We now make a remark on the equivalence of \eqref{eq:cltfixed} and
\eqref{eq:cltsparse} when $p \ne \frac{1}{2}$ is fixed.  Note that
when $p_n\equiv p\in(0,1)$ is fixed, the condition
$p_n(\NH{\Ktwo}+\NH{\Pthree})\to\infty$ holds automatically from our
standing assumption $\NH{\Ktwo}\to\infty$. The remaining condition in
\eqref{eq:cltsparse} then reduces to
$$\NH{\Konethree} =o ((\NH{\Ktwo}+\NH{\Pthree})^{\frac{3}{2}}).$$
For fixed $p\ne\frac12$, this is equivalent to the condition in \eqref{eq:cltfixed}, by \cref{lem:4star3starbound} and \eqref{eq:starcount}. Hence, the $3$-star and $4$-star criteria when $p\ne\frac12$ is fixed, are equivalent, showing that the two cases of \cref{thm:clt} are consistent.

\subsection{Preliminaries for Proof of \cref{thm:clt}} 
\label{sec:prepw}

Let $(p_n)_{n\in\N}$ be a sequence such that $p_n\in (0,1)$, for all $n\in \N$. (In the fixed $p$ regime $p_n = p \in (0, 1) \backslash \{\frac{1}{2}\}$, for all $n\in \N$, and in the sparse regime $p_n \rightarrow 0$, as $n \rightarrow \infty$.) For each $e = \{i,j\} \in E(\Gn)$, recall $Y_e$ from \eqref{eq:Ye}. Then the normalised statistic $U_n$ (recall \eqref{eq:Un}) can be rewritten as
$$U_n = \frac{1}{\sign} \sum_{e\in E(\Gn)} Y_e , $$
where $\sigma_n^2 := \Var{T_n}$.  Note that $\E{Y_e} = 0$ and $\Var{Y_e} = q_n(1-q_n) = \frac{1}{4}(1-u_n^2)$. Further, $Y_e$ and $Y_{e'}$ are independent whenever $e\cap e'=\varnothing$, and for edges $e, e'$ sharing exactly one vertex $\Cov{Y_e}{Y_{e'}} = p_n^3+(1-p_n)^3 - q_n^2 = \frac{1}{4}(u_n-u_n^2)$. Hence,  
\begin{align}\label{eq:varianceUn}
    \sign^2 &= \sum_{e_1,e_2\in E(\Gn)} \Cov{Y_{e_1}}{Y_{e_2}}
    = \frac{1-u_n^2}{4}\NH{\Ktwo} + \frac{u_n-u_n^2}{2}\NH{\Pthree}, 
    \end{align}
where $u_n = (2p_n-1)^2$. (This is the same as the expression in \eqref{eq:var_term} with $u$ replaced by $u_n$, to highlight the dependence on $n$ in the sparse regime). 
We now use the Wasserstein distance bound obtained from Stein's operator (see, for example, \citet[Theorem 3.1]{Ross2011}): 
\begin{equation}\label{eq:steinsbd}
    \Wass(U_n,N(0,1)) \leq \sup_{f\in \mathcal{F}} \big|\E{U_n f(U_n)} - \E{f'(U_n)}\big|,
\end{equation}
where $\mathcal{F} = \{f : \R \rightarrow \R \text{ such that } \|f\|_\infty \le 1, \|f'\|_\infty \le \sqrt{\frac2\pi}, \|f''\|_\infty \le 2 \}$. Fix $f\in \mathcal{F}$ and for each edge $e =\{i,j\}\in E(\Gn)$, set
\begin{equation*}  
   U_e = \frac{1}{\sign} \sum_{\substack{e'\in E(\Gn)\\ e'\cap e=\emptyset}} Y_{e'}.
\end{equation*}
 Then
\begin{align}
    \E{U_nf(U_n)} - \E{f'(U_n)}
    &= \frac{1}{\sign}\sum_{e\in E(\Gn)} \E{Y_e f(U_n)} - \E{f'(U_n)} \notag\\
    &= \frac{1}{\sign}\sum_{e\in E(\Gn)} \E{Y_e(f(U_n)-f(U_e))} - \E{f'(U_n)}, \label{eq:a1_taylor}
\end{align}
where \eqref{eq:a1_taylor} follows from the fact $$\E{Y_e f(U_e)} = \E{Y_e}\E{f(U_e)} = 0,$$ as $Y_e$ is
independent of $U_e$. Adding and subtracting
$\frac1{\sign}\sum_{e\in E(\Gn)} \E{Y_e(U_n - U_e)f'(U_n)}$ on the right-hand side of \eqref{eq:a1_taylor},
we have
\begin{align}
      \E{U_nf(U_n)} - \E{f'(U_n)} &= \frac{1}{\sign}\sum_{e\in E(\Gn)}
        \E{Y_e(f(U_n)-f(U_e)-(U_n-U_e) f'(U_n))} \nonumber \\
        &\quad +\frac{1}{\sign}\sum_{e\in E(\Gn)} \E{Y_e (U_n-U_e) f'(U_n)} - \E{f'(U_n)}. \label{eq:instep}
\end{align}
Define 
\begin{align}\label{eq:Vn}
V_n := \sum_{\substack{e_1,e_2 \in E(\Gn)\\ e_1\cap e_2 \neq \varnothing}} Y_{e_1} Y_{e_2} = \sigma_n \sum_{e\in E(\Gn)} Y_{e} (U_n-U_e).  
\end{align}
Using the fact that $\E{V_n} = \sign^2$, we may rewrite \eqref{eq:instep} as 
\begin{align}
     & \E{U_nf(U_n)} - \E{f'(U_n)} \nonumber \\ 
     &= \frac{1}{\sign} \sum_{e\in E(\Gn)}  \E{Y_e(f(U_n)-f(U_e)-(U_n-U_e) f'(U_n))} + \E{f'(U_n)\left(\frac{V_n - \E{V_n}}{\sign^2}\right)} \nonumber \\
    &=: A_1 + A_2  ,  
    \label{eq:A1_A2}
\end{align}
where
$$A_1 := \frac{1}{\sign}  \sum_{e\in E(\Gn)} \E{Y_e(f(U_n)-f(U_e)-(U_n-U_e) f'(U_n))} $$
and 
$$A_2 :=  \E{f'(U_n)\left(\frac{V_n - \E{V_n}}{\sign^2}\right)}.$$
Hence, by \eqref{eq:steinsbd}, to prove \cref{thm:clt} it suffices to show that both $A_1$ and $A_2$ are $o(1)$, uniformly in $f\in\mathcal F$, in the respective settings.

\subsection{Proof of \cref{thm:clt} $(i)$}
\label{sec:thmclthf}
We will need two lemmas. One which bounds the number of $3$-stars in a graph in terms of the numbers of $2$-stars and $4$-stars, and a second which computes the variance of $V_n$ as defined in \eqref{eq:Vn}.

\begin{lemma}\label{lem:4star3starbound}
For every finite simple graph $G_n$,
\begin{equation}\label{eq:4star3starbound}      
\NH{\Konethree}\le \NH{\Pthree} + \sqrt{\frac 83 \NH{\Pthree} \NH{\Konefour}}.
\end{equation}  
\end{lemma} 

\begin{lemma}\label{lemma:Varlemma}
  Let $\{\Gn\}_{n\in\N}$ be a sequence of graphs with $\NH{\Ktwo} \to \infty$. Let $V_n$ be as defined in \eqref{eq:Vn} and $p\in (0,1)\backslash\{ \frac{1}{2}\}$. 
\begin{equation} \label{eq:VarVbd}
                \Var{V_n} = c_{K_{1,4}} \NH{\Konefour} + o\left((\NH{\Ktwo}+\NH{\Pthree})^2\right)  , 
            \end{equation}
            where $c_{K_{1,4}}= 6u^3(1-u)$ and $u=(2p-1)^2$.  
\end{lemma}

We shall defer the proof of this lemma to the end of this section and prove \cref{thm:clt}$(i)$ assuming the above result.

\begin{proof}[Proof of {\rm \cref{thm:clt}$(i)$:}]
We now proceed with the proof of   \cref{thm:clt} $(i)$. We first estimate $A_2$. By \eqref{eq:A1_A2}, the Cauchy--Schwarz inequality, the fact that $f\in \mathcal{F}$,  and  \cref{lemma:Varlemma}, 
\begin{equation}\label{eq:A2bound}
|A_2| \le \frac{\|f'\|_\infty}{\sign^2}\sqrt{\Var{V_n}}
\le \frac{1}{\sign^2} \sqrt{O(\NH{\Konefour}) + o\left((\NH{\Ktwo}+\NH{\Pthree})^2\right)}.
\end{equation}
In this case, recalling \eqref{eq:varianceUn} 
\begin{align}\label{eq:variancepfixed}
\sign^2 \asymp \NH{\Ktwo}+\NH{\Pthree}.
\end{align}
Using \eqref{eq:variancepfixed} and the assumption $\NH{\Konefour} = o((\NH{\Ktwo}+\NH{\Pthree})^2)$ in \eqref{eq:A2bound} we have $|A_2| = o(1)$.

Next, we consider $A_1$. By a Taylor expansion,
\begin{align}
    |A_1|
    & \leq  \frac{1}{\sign} \cdot \frac{\|f''\|_\infty} {2} \cdot \sum_{e\in E(\Gn)} \E{|Y_e|\cdot(U_n-U_e)^2}
    \leq \frac{1}{\sign}\sum_{e\in E(\Gn)} \E{|Y_e|\cdot(U_n-U_e)^2}. \label{eq:A1_expand}
\end{align}
Recall that $ d_{G_n}(i)$ is the degree of the $i$-th vertex of $G_n$, for $1 \leq i \leq n$. Then the sum in the right-hand side of \eqref{eq:A1_expand} can be expressed as: 
\begin{align*}
\sum_{e\in E(\Gn)} \E{ |Y_e|\cdot(U_n-U_e)^2}
&=  \frac{1}{\sign^2} \sum_{e\in E(\Gn)} \E{|Y_e|\cdot\left(\sum_{\substack{e'\in E(\Gn)\\ e'\cap e \neq \varnothing}}Y_{e'}\right)^2} \\
&\le  \frac{1}{\sign^2}  \sum_{e\in E(\Gn)} \big|\{e'\in E(\Gn) : e'\cap e \neq \varnothing\}\big|^2 \\
&\le  \frac{1}{\sign^2}  \sum_{\{i,j\}\in E(\Gn)} ( d_{G_n}(i) +  d_{G_n}(j))^2 \\
& \le  \frac{2}{\sign^2}  \sum_{\{i,j\}\in E(\Gn)} ( d_{G_n}(i)^2 +  d_{G_n}(j)^2)   = \frac{2}{\sign^2} \sum_{i=1}^n  d_{G_n}(i)^3,
\end{align*}
where the first inequality uses $|Y_e|\le 1$, for all $e$. Substituting this back into
\eqref{eq:A1_expand} gives, 
\begin{align*}
    |A_1| &\le \frac{2}{\sign^3} \sum_{i=1}^n  d_{G_n}(i)^3 \nonumber\\
    &= \frac{2}{\sign^3} \sum_{i=1}^n \left(6\binom{ d_{G_n}(i)}{3} + 6\binom{ d_{G_n}(i)}{2} +  d_{G_n}(i)\right) \nonumber\\
    &= \frac{1}{\sign^3}\left(12\NH{\Konethree} + 12\NH{\Pthree} + 4\NH{\Ktwo}\right)\nonumber\\
    &\leq  \frac{1}{\sign^3}\left(24 \NH{\Pthree} + 12\sqrt{\tfrac83\NH{\Pthree}\NH{\Konefour}}  + 4\NH{\Ktwo} \right), 
\end{align*}  
where the last inequality follows from \cref{lem:4star3starbound}.  Using \eqref{eq:variancepfixed} and the assumption $\NH{\Konefour} = o((\NH{\Ktwo}+\NH{\Pthree})^2)$ gives $|A_1| = o(1)$. This completes the proof of   \cref{thm:clt} $(i)$.  
\end{proof}

\begin{remark}\label{remark:clt}
By tracking the constants in the proof of \cref{thm:clt}$(i)$ and replacing the asymptotic estimate in \cref{lemma:Varlemma} with a quantitative bound on $\Var{V_n}$, one obtains the following rate: For every fixed
$p\in(0,1)\backslash\{\frac12\}$,
\begin{align}\label{eq:WNUndense}
\Wass(U_n, N(0,1)) \le C_p\left( \frac{1}{\sqrt{\NH{\Ktwo}+\NH{\Pthree}}} +  \frac{\sqrt{\NH{\Konefour}}}{\NH{\Ktwo}+\NH{\Pthree}} \right)  ,
\end{align}
for some constant $C_p>0$ depending only on $p$. This bound holds for every $n$ such that $G_n$ has at least one edge. 
\end{remark}

To complete the proof of \cref{thm:clt}$(i)$ we prove \cref{lemma:Varlemma} and \cref{lem:4star3starbound}.

\begin{proof}[Proof of {\rm \cref{lemma:Varlemma}:}] Recall the definition of $V_n$ from \eqref{eq:Vn}.  Expanding the variance
    of $V_n$ gives
    \begin{equation} \label{eq:A2exp}
        \Var{V_n} = \sum_{\substack{e_1,e_2 \in E(\Gn)\\ e_1\cap e_2 \neq \varnothing}}
        \sum_{\substack{e_3,e_4 \in E(\Gn)\\ e_3\cap e_4 \neq \varnothing}} \Cov{Y_{e_1}Y_{e_2}}{Y_{e_3}Y_{e_4}}.
    \end{equation}
    Note that the covariance vanishes unless $(e_1 \cup e_2)\cap (e_3 \cup e_4) \neq \varnothing$, hence the sum in \eqref{eq:A2exp} runs over ordered $4$-tuples of edges $(e_1,e_2,e_3,e_4)$ such that $e_1\cap e_2 \neq \varnothing$,
    $e_3\cap e_4\neq \varnothing$ and $(e_1 \cup e_2)\cap (e_3 \cup e_4) \neq \varnothing$. We split the    summation in \eqref{eq:A2exp} according to the isomorphism type of the subgraph
    $e_1\cup e_2\cup e_3\cup e_4$. The possible subgraphs form the following collection: 
    $$\mathcal H = \{\Ktwo, \Pthree, \Kthree, \Pfour,\Konethree, \Pfive, \Konefour, \Tthreetwo, \Cfour, \Paw \} = \{K_2, P_3, K_3, P_4, K_{1, 3}, P_5, K_{1, 4}, T_{3, 2}, C_4, \mathrm{paw}\} . $$ For any $H \in \mathcal H$, the total contribution of all tuples spanning a given copy of $H$ depends only on $H$ and
    $u_n$. We denote it by $c_H = c_H(u_n)$. Thus, we can write
    \begin{align}        
\Var{V_n} = \sum_{H \in \mathcal H} c_H N(H, G_n). 
        \label{eq:VarVmotif1}
    \end{align}
    To evaluate the coefficients $c_H$, one computes $\Cov{Y_{e_1}Y_{e_2}}{Y_{e_3}Y_{e_4}}$ for all tuples
    of edges $(e_1,e_2,e_3,e_4)$ whose union is a fixed copy of $H$. For instance, when $H = \Ktwo$, all four edges coincide, and
    $c_{K_2} = \Var{Y_e^2}$. The remaining coefficients are computed similarly; the results, together with the asymptotic order of the corresponding subgraph counts (obtained from
    $\sum_{i=1}^n  d_{G_n}(i) = 2\NH{\Ktwo}$, $\sum_{i=1}^n \binom{ d_{G_n}(i)}{2} = \NH{\Pthree}$ and elementary bounds), are displayed in  \cref{tab:varV} in \cref{app:tables}. The following stronger bound will be required in the proof: 
\begin{equation}\label{eq:NHtree}
\NH{\Pfive}+\NH{\Tthreetwo} =O(\NH{\Pthree}^{\frac74}).
\end{equation}
Note that if $\NH{\Pthree}=0$, then both terms on the left of
\eqref{eq:NHtree} vanish. We therefore assume $\NH{\Pthree}\ge1$. Using $\binom{d_{G_n}(i)}{2}\le\NH{\Pthree}$, it follows that $d_{G_n}(i)=O(\NH{\Pthree}^{\frac12})$, uniformly in $1 \leq i \leq n$.
Hence, by the Cauchy--Schwarz inequality,
$$\sum_{j:\{i,j\}\in E(G_n)}(d_{G_n}(j)-1) \le (2d_{G_n}(i)\NH{\Pthree})^{\frac12} =O(\NH{\Pthree}^{\frac34}), $$
where we used $\sum_{j=1}^n d_{G_n}(j)(d_{G_n}(j)-1)=2\NH{\Pthree}$. Counting copies of $\Pfive$ according to their middle vertex gives, 
$$
\NH{\Pfive}
\le \frac12\sum_{i=1}^n
\left(\sum_{j:\{i,j\}\in E(G_n)}(d_{G_n}(j)-1)\right)^2
=O(\NH{\Pthree}^{\frac74}),
$$
since
$\sum_{i=1}^n\sum_{j:\{i,j\}\in E(G_n)}(d_{G_n}(j)-1)
=2\NH{\Pthree}$.
Similarly, counting copies of $\Tthreetwo$ according to their
unique vertex of degree three gives, 
$$
\NH{\Tthreetwo}
\le \sum_{i:d_{G_n}(i)\ge3}
\binom{d_{G_n}(i)-1}{2}
\sum_{j:\{i,j\}\in E(G_n)}(d_{G_n}(j)-1)
=O(\NH{\Pthree}^{\frac74}),
$$
because
$\sum_{i:d_{G_n}(i)\ge3}\binom{d_{G_n}(i)-1}{2}
\le\NH{\Pthree}$.
This proves \eqref{eq:NHtree}.

     We consider the following cases depending on the order of $\NH{\Ktwo}$.

   \begin{itemize}

\item  $\NH{\Ktwo} = o(\NH{\Pthree})$: Then \eqref{eq:VarVbd} follows by plugging the values for $c_H$ and the orders of $\NH{H}$ from  \cref{tab:varV} into \eqref{eq:VarVmotif1}, since every count other than $\NH{\Konefour}$ is $$O(\NH{\Ktwo}\NH{\Pthree}) + o(\NH{\Pthree}^2).$$  

\item $\NH{\Pthree} \le C\NH{\Ktwo}$, for some $C>0$: Note that $$\Delta(G_n):= \max_{ 1 \leq i \leq n}  d_{G_n}(i) = O(\sqrt{\NH{\Ktwo}}),$$ since
    $\binom{\Delta(G_n)}{2} \le \NH{\Pthree} \le C\NH{\Ktwo}$. 
    Using this bound and \eqref{eq:NHtree}, each motif count in  \cref{tab:varV} other than $\NH{\Konefour}$ is seen to be $o(\NH{\Ktwo}^2)$: for instance
    $$\NH{\Cfour}\le \NH{\Pthree} \Delta(G_n) = O(\NH{\Ktwo}^{\frac{3}{2}})$$ and
    $$\NH{\Konethree}\le \NH{\Pthree}\Delta(G_n) = O(\NH{\Ktwo}^{\frac{3}{2}}).$$ Hence, \eqref{eq:VarVbd} also holds in this case.  
\end{itemize}
\end{proof}

\begin{proof}[Proof of {\rm \cref{lem:4star3starbound}:}] 
Since $\NH{\Konethree} = \sum_{i=1}^n \binom{ d_{G_n}(i)}{3}$, we split this sum according to whether
    $ d_{G_n}(i) = 3$ or $ d_{G_n}(i) \ge 4$. For $ d_{G_n}(i) = 3$, use $\binom{ d_{G_n}(i)}{3} =1\le \binom{ d_{G_n}(i)}{2}$, so the sum
    over all vertices of degree $3$ is at most $\NH{\Pthree}$. For $ d_{G_n}(i)\ge 4$, we can write
    \[
        \binom{ d_{G_n}(i)}{3} = \frac{ d_{G_n}(i)-2}{3}\binom{ d_{G_n}(i)}{2}.
    \]
    Summing over all vertices $i$ with $ d_{G_n}(i) \ge 4$ and applying the Cauchy--Schwarz inequality gives
    \begin{align*}
        \sum_{i: d_{G_n}(i)\ge 4} \binom{ d_{G_n}(i)}{3} &\le \sqrt{\sum_{i: d_{G_n}(i)\ge 4}\binom{ d_{G_n}(i)}{2}} \cdot
        \sqrt{\sum_{i: d_{G_n}(i)\ge 4} \binom{ d_{G_n}(i)}{2}\left(\frac{ d_{G_n}(i)-2}{3}\right)^2} \\
        &\le \sqrt{\NH{\Pthree}} \cdot \sqrt{\sum_{i: d_{G_n}(i) \ge 4} \frac 43 \binom{ d_{G_n}(i)}{4} \frac{ d_{G_n}(i) - 2}{ d_{G_n}(i) - 3}}  \nonumber \\  
        & \le \sqrt{\frac 83 \NH{\Pthree}\sum_{i=1}^n \binom{ d_{G_n}(i)}{4}},
    \end{align*}
    where we used $\binom{ d_{G_n}(i)}{2}\frac{( d_{G_n}(i)-2)^2}{9} = \frac43\binom{ d_{G_n}(i)}{4}\frac{ d_{G_n}(i)-2}{ d_{G_n}(i)-3}$ and the fact that $\frac{ d_{G_n}(i) - 2}{ d_{G_n}(i)-3}\le 2$, for $ d_{G_n}(i)\ge 4$. Combining the above bounds the result follows.
    \end{proof}

\subsection{Proof of \cref{thm:clt} $(ii)$}
\label{sec:thmcltsp}

\begin{lemma}\label{lem:threestarbd_sparse}
    Under the hypotheses of  {\rm \cref{thm:clt}  $(ii)$}, 
    \begin{equation*}
       \sum_{i=1}^n  d_{G_n}(i)^3 = o\left(p_n^{\frac{1}{2} }(\NH{\Ktwo} + \NH{\Pthree})^{\frac{3}{2} }\right).
    \end{equation*}
\end{lemma}

\begin{proof}
Using the identity $ d_{G_n}(i)^3 = 6\binom{ d_{G_n}(i)}{3} + 6\binom{ d_{G_n}(i)}{2} +  d_{G_n}(i)$,  we obtain
\begin{align*}
    \sum_{i=1}^n  d_{G_n}(i)^3 & = 6\NH{\Konethree} + 6\NH{\Pthree} + 2\NH{\Ktwo}  \nonumber \\  
    &\le 6\big(\NH{\Konethree} + \NH{\Pthree} + \NH{\Ktwo}\big).
\end{align*}
By the second condition in \cref{thm:clt} $(ii)$,
$$\NH{\Konethree} = o\left(p_n^{\frac12} (\NH{\Ktwo}+\NH{\Pthree})^{\frac32}\right).$$
Combining the above with the first condition in \cref{thm:clt} $(ii)$ gives the required result.  
\end{proof}

\begin{proof}[Proof of {\rm \cref{thm:clt}(ii):}] We first bound $A_1$. Recall from \eqref{eq:var_term} that
\begin{align}
       \sign^2 &= \frac{(1-u_n)(1+u_n)}{4}\NH{\Ktwo} + \frac{u_n (1-u_n)}{2}\NH{\Pthree}   \nonumber \\ 
       & = 2p_n(\NH{\Ktwo}+\NH{\Pthree})(1+o(1))  ,  
       \label{eq:var_sparse}
    \end{align}
     since $1-u_n = 4p_n(1-p_n) = 4p_n(1+o(1))$ and $u_n\to1$, as $p_n \rightarrow 0$. Hence, under the first condition in  \cref{thm:clt} $(ii)$, $\sign^2$ diverges.  Now, from
    \eqref{eq:A1_A2} and \eqref{eq:A1_expand}, we have 
    \begin{align}
        |A_1| &\leq \frac{1}{\sign}\sum_{e\in E(\Gn)} \E{|Y_e|\cdot(U_n-U_e)^2} \nonumber\\
        & \leq \frac{1}{\sign^3}\sum_{e\in E(\Gn)}
        \sum_{{\substack{e_1\in E(\Gn)\\ e_1\cap e \neq \varnothing}}}
        \sum_{{\substack{e_2\in E(\Gn)\\ e_2\cap e \neq \varnothing}}}\E{|Y_e Y_{e_1} Y_{e_2}|}. \label{eq:A1tripprodbd}
    \end{align}
    To estimate the expectation of this triple product, for each $e = \{i,j\}\in E(\Gn)$ define
    $Z_e := \bm{1}\{X_i \neq X_j\}$, so that
    \begin{equation}\label{eq:Zeexp}
        \E{Z_e} = 2p_n(1-p_n) \le 2p_n.
    \end{equation}
    Recall that $q_n = \E{\bm{1} \{X_i = X_j \}} = p_n^2 + (1-p_n)^2$ and, hence, 
    $Y_e = \bm{1}\{X_i = X_j\} - q_n = (1-q_n) - Z_e$ and $1-q_n = 2p_n(1-p_n)\le 2p_n$. Further, $|Y_e|= 1-q_n$ on $\{Z_e=0\}$ and $|Y_e| = q_n\le 1$ on $\{Z_e=1\}$. This gives the bound,  
    \begin{equation*}
        |Y_e| \le 2p_n + Z_e .
    \end{equation*}
    Using this bound on $|Y_e|$, the triple product in \eqref{eq:A1tripprodbd} satisfies
    \begin{align*}
        |Y_e Y_{e_1} Y_{e_2}| &\le (2p_n + Z_e)(2p_n + Z_{e_1})(2p_n + Z_{e_2})\\
        &\le 8p_n^3 + 4 p_n^2(Z_e + Z_{e_1} + Z_{e_2}) + 2 p_n(Z_e Z_{e_1} +Z_e Z_{e_2} + Z_{e_1} Z_{e_2})
        + Z_e Z_{e_1} Z_{e_2}.
    \end{align*}
    Taking expectations and using \eqref{eq:Zeexp} together with
    $\E{Z_a Z_b} \le \E{Z_a} \le 2p_n$ and $\E{Z_a Z_b Z_c} \le \E{Z_a} \le 2p_n$, for all
    $a,b,c \in E(\Gn)$, 
    \begin{align}
        \E{|Y_e Y_{e_1} Y_{e_2}|} &\le 8p_n^3 + 24 p_n^3 + 12 p_n^2 + 2p_n \le C p_n  , \label{eq:productY}
    \end{align}
    for an absolute constant $C$ and all $n$ large enough. Combining \eqref{eq:var_sparse} and \eqref{eq:productY}, together with \eqref{eq:A1tripprodbd},  gives,  
    \begin{align}
        |A_1| &\le \frac{C p_n}{\sign^3}\sum_{e\in E(\Gn)}
        \big|\{e'\in E(\Gn) : e'\cap e \neq \varnothing\}\big|^2
        \le \frac{C'\sum_{i=1}^n  d_{G_n}(i)^3}{p_n^{\frac{1}{2}}(\NH{\Pthree} + \NH{\Ktwo})^{\frac{3}{2}}}  ,  
        \label{eq:A1finalbdsparse}
    \end{align}
    for some constant $C' > 0$. Using  \cref{lem:threestarbd_sparse} in \eqref{eq:A1finalbdsparse}, we get $|A_1| = o(1)$.

We now bound $|A_2|$. By \eqref{eq:A1_A2} and the definition of the function class $\mathcal F$,
    \begin{align}\label{eq:sparseA2}
        |A_2| \le \sqrt{ \frac{2}{\pi} } \frac{1}{\sign^2}\E{|V_n-\E{V_n}|} . 
    \end{align}
    Define $\tilde V_n:=\sum_{i=1}^n  d_{G_n}(i)^2X_i$, where $ d_{G_n}(i)$ is the degree of the $i$-th vertex in $G_n$. We first show that
\begin{equation}\label{eq:degreeVn}
\E{|V_n-\tilde V_n|} \le Cp_n^2\sum_{i=1}^n  d_{G_n}(i)^2 , 
\end{equation}
for some constant $C > 0$. Recall that $Y_e=(1-q_n)-Z_e$ and $\E{Z_e}=1-q_n=2p_n(1-p_n)$. Hence, for any two edges $e_1, e_2 \in E(G_n)$,
    \begin{align*}
        \E{|Y_{e_1} Y_{e_2}-Z_{e_1}Z_{e_2} |} \le (1-q_n)\E{Z_{e_1} + Z_{e_2}}+(1-q_n)^2 =3(1-q_n)^2\le 12p_n^2.
    \end{align*}
Further, for $e_1 =\{i, j\}$ and $e_2 = \{j, k\}$ such that the indices $i, j, k$ are distinct, $$Z_{e_1}^2 =X_i+X_j-2X_iX_j \text{ and } Z_{e_1}Z_{e_2} =X_j-X_iX_j-X_jX_k+X_iX_k.$$ Hence, in the sum
$$\sum_{\substack{e_1, e_2\in E(\Gn)\\e_1 \cap e_2\neq\varnothing}}Z_{e_1}Z_{e_2},$$  
the linear terms add to $\sum_{i=1}^n \left( d_{G_n}(i)+ d_{G_n}(i)( d_{G_n}(i)-1)\right)X_i =\tilde V_n.$
By the triangle inequality,
 \begin{align*}
\E{|V_n-\tilde V_n|} &\le  \E{\left| \sum_{\substack{e_1, e_2\in E(\Gn)\\e_1 \cap e_2\neq\varnothing}}
        (Y_{e_1}Y_{e_2}-Z_{e_1}Z_{e_2})\right|}  + \E{\left| \sum_{\substack{e_1, e_2\in E(\Gn)\\e_1 \cap e_2\neq\varnothing}}
        Z_{e_1}Z_{e_2}-\tilde V_n\right|} \\
        &\le \sum_{\substack{e_1, e_2\in E(\Gn)\\e_1 \cap e_2\neq\varnothing}} \E{| Y_{e_1}Y_{e_2}-Z_{e_1}Z_{e_2}|} +3p_n^2(\NH{\Ktwo}+2\NH{\Pthree}) \\
        &\le 15p_n^2(\NH{\Ktwo}+2\NH{\Pthree}) \\
        &\le 15p_n^2\sum_{i=1}^n  d_{G_n}(i)^2.
    \end{align*}
 This proves \eqref{eq:degreeVn}. Then by triangle inequality, 
\begin{align}\label{eq:momentVn}
\E{|V_n-\E{V_n}|} & \le \E{ \left| \tilde V_n-\E{ \tilde V_n} \right|} +2\E{ \left|V_n - \tilde V_n  \right|}  \nonumber \\ 
& \le \E{\left|\sum_{i=1}^n  d_{G_n}(i)^2(X_i-p_n)\right|}+30p_n^2\sum_{i=1}^n  d_{G_n}(i)^2.
\end{align}
    To bound the first term, set $t_n:= (p_n\sum_{i=1}^n  d_{G_n}(i)^3)^{\frac{1}{3}}$ and first consider the indices where $ d_{G_n}(i) \leq t_n$. By the Cauchy--Schwarz inequality and independence, 
\begin{align*}
\left(\E{\left|\sum_{i:  d_{G_n}(i)\le t_n} d_{G_n}(i)^2(X_i-p_n)\right|}\right)^2 &  \le \E{\left(\sum_{i:  d_{G_n}(i)\le t_n} d_{G_n}(i)^2(X_i-p_n)\right)^2}  \nonumber  \\  
& \le p_n\sum_{i:  d_{G_n}(i)\le t_n} d_{G_n}(i)^4 .
\end{align*}
    For the indices where $ d_{G_n}(i) > t_n$ we use the bound
    \begin{align*}
    \E{\left|\sum_{i:  d_{G_n}(i)>t_n} d_{G_n}(i)^2(X_i-p_n)\right|} \le \sum_{i:  d_{G_n}(i)>t_n} d_{G_n}(i)^2\E{|X_i-p_n|} \le 2p_n\sum_{i:  d_{G_n}(i)>t_n} d_{G_n}(i)^2.
\end{align*}
Hence, 
\begin{align}\label{eq:tn}
\E{\left|\sum_{i=1}^n  d_{G_n}(i)^2(X_i-p_n)\right|} &\le \left(p_n\sum_{i:  d_{G_n}(i)\le t_n} d_{G_n}(i)^4\right)^{\frac{1}{2}}+2p_n\sum_{i:  d_{G_n}(i)>t_n} d_{G_n}(i)^2 \nonumber \\
        &\le \left(p_nt_n\sum_{i=1}^n  d_{G_n}(i)^3\right)^{\frac{1}{2}} +\frac{2p_n}{t_n}\sum_{i=1}^n  d_{G_n}(i)^3 \nonumber \\
        &=3\left(p_n\sum_{i=1}^n  d_{G_n}(i)^3\right)^{\frac{2}{3}}.
    \end{align}
    Combining \eqref{eq:sparseA2}, \eqref{eq:momentVn}, and \eqref{eq:tn}, with \eqref{eq:var_sparse} and
    $\sum_{i=1}^n  d_{G_n}(i)^2=2(\NH{\Ktwo}+\NH{\Pthree})$, for some constant $C > 0$, 
    \begin{align*}
        |A_2|  &\le \frac{C}{\sign^2} \left[ \left(p_n\sum_{i=1}^n  d_{G_n}(i)^3\right)^{\frac{2}{3}} +p_n^2\sum_{i=1}^n  d_{G_n}(i)^2 \right] \\
        &\le C\left[\left(\frac{\sum_{i=1}^n  d_{G_n}(i)^3} {p_n^{\frac{1}{2}} (\NH{\Ktwo}+\NH{\Pthree})^{\frac{3}{2}}}  \right)^{\frac{2}{3}}  +p_n \right]
        =o(1),
    \end{align*}
    where the last step follows from \cref{lem:threestarbd_sparse} and $p_n\to 0$. This completes the proof of \cref{thm:clt} $(ii)$.
  \end{proof}

\section{Proof of \cref{thm:Un}, \cref{ppn:clthalf_sparse} and \cref{ppn:thirdfourthsparse}}
\label{sec:Unpf}

We will prove the results in reverse order. The Proof of
\cref{ppn:thirdfourthsparse} follows readily from
\cref{lem:fourthmomentsparse} and \cref{lem:thirdmomentsparse}. We
provide the details below.

\begin{proof}[Proof of {\rm \cref{ppn:thirdfourthsparse}:}]
Suppose $\E{U_n^4}\to3$. By \cref{lem:fourthmomentsparse},
$$
p_n(\NH{\Ktwo}+\NH{\Pthree})\to\infty
\quad\text{and}\quad
\NH{\Konefour}
=o(p_n(\NH{\Ktwo}+\NH{\Pthree})^2).
$$
Hence, by  \eqref{eq:4star3starbound}, 
$$\frac{\NH{\Konethree}} {p_n^{\frac12}(\NH{\Ktwo}+\NH{\Pthree})^{\frac32}} \rightarrow 0.$$
The conclusion now follows from \cref{lem:thirdmomentsparse}.
\end{proof}

The proof of \cref{ppn:clthalf_sparse} also follows immediately.

\begin{proof}[Proof of {\rm \cref{ppn:clthalf_sparse}:}]
  The result in  \cref{ppn:clthalf_sparse} is an immediate consequence of  \cref{thm:clt}(ii)   and \cref{lem:thirdmomentsparse}. 
  \end{proof} 

The proof of \cref{thm:Un} requires one more lemma which ensures the uniform integrability of $U_n^4$. We state the lemma next.

\begin{lemma}\label{lem:momentUn}  
Fix $p\in(0,1)$ and let $U_n$ be as defined in \eqref{eq:Un}. Then there exists a constant $C_p<\infty$, depending only on $p$, such that
\begin{align}\label{eq:momentUn}
\sup_{n \geq 1} \E{|U_n|^6}  <  C_p. 
\end{align}
\end{lemma}

Note that the bound is for fixed $p$. Uniform integrability of $U_n^4$ can fail when $p_n\to0$, which is why asymptotic normality does not force the moment conditions in the sparse regime (see \cref{example:thirdsparse}). We shall defer the proof of this lemma to the end of this section and prove \cref{thm:Un} assuming the above result.  

\begin{proof}[Proof of  {\rm \cref{thm:Un}:}]
$(i)$ Fix $p \in (0, 1) \backslash \{\frac 12, \frac{1\pm 1/\sqrt{3}}{2}\}$. Then $(b)$ implies $(a)$ follows from \cref{thm:clt}$(i)$. The equivalence of $(b)$ and $(c)$ follows from \eqref{eq:fmc_cond1} in \cref{lem:fm_pneqhalf}(a). Further, using the bound  \eqref{eq:4star3starbound} from \cref{lem:4star3starbound} and \eqref{eq:thirdmc_cond1}from \cref{lem:thirdm_pneqhalf} it follows that $(b)$ implies $(d)$. To prove (d) implies (b), we use the bound $$( d_{G_n}(i)-3)^3\le  d_{G_n}(i)( d_{G_n}(i)-1)( d_{G_n}(i)-2)=6\binom{ d_{G_n}(i)}{3},$$ for any $ d_{G_n}(i) \geq 1$. This yields, 
\begin{align*}
\NH{\Konefour} =\sum_{i=1}^n \binom{ d_{G_n}(i)}{4} \le \frac{6^{\frac{1}{3}}}{4}
       \sum_{i=1}^n \binom{ d_{G_n}(i)}{3}^{\frac{4}{3}} & \le \frac{6^{\frac{1}{3}}}{4}
       \left(\sum_{i=1}^n \binom{ d_{G_n}(i)}{3}\right)^{\frac{4}{3}}
=\frac{6^{\frac{1}{3}}}{4}\left(\NH{\Konethree}\right)^{\frac{4}{3}}.
\end{align*}
Hence, 

\begin{align}\label{eq:starcount}
\frac{\NH{\Konefour}} {\left(\NH{\Ktwo}+\NH{\Pthree}\right)^2} \le \frac{6^{\frac{1}{3}}}{4} \left( \frac{\NH{\Konethree}} {\left(\NH{\Ktwo}+\NH{\Pthree}\right)^{\frac{3}{2}}} \right)^{\frac{4}{3}} = o(1). 
\end{align}

Again employing \eqref{eq:thirdmc_cond1} yields the implication. Finally, since $\{U_n^4\}_{n \geq 1}$ is uniformly integrable by \cref{lem:momentUn}, we have $(a)$ implies both $(c)$ and $(d)$. This proves all the implications in  \cref{thm:Un} $(i)$.  

\medskip

$(ii)$ Fix $p = \frac{1\pm 1/\sqrt{3}}{2}$. Then $(b)$ implies $(a)$ follows from \cref{thm:clt}(i). The equivalence of $(b)$ and $(c)$ follows from \eqref{eq:thirdmc_cond1}, \cref{lem:4star3starbound}, and \eqref{eq:starcount}. Finally, since $\{U_n^4\}_{n \geq 1}$ is uniformly integrable by \cref{lem:momentUn}, we have $(a)$ implies $(c)$. This proves all the implications in  \cref{thm:Un} $(ii)$.

\medskip 

$(iii)$ Fix $p = \frac{1}{2}$. Then $(b)$ implies $(a)$ follows from \cref{ppn:clthalf}. The equivalence of $(b)$ and $(c)$ follows from \eqref{eq:fmc_phalf} in \cref{lem:fm_pneqhalf} (c). Finally, since $\{U_n^4\}_{n \geq 1}$ is uniformly integrable by \cref{lem:momentUn}, we have $(a)$ implies $(c)$. This proves all the implications in  \cref{thm:Un} $(iii)$.  \hfill  
\end{proof}

To complete the proof of \cref{thm:Un}  we now prove \cref{lem:momentUn}.

\begin{proof}[Proof of {\rm \cref{lem:momentUn}:}]
Let $S_n$ be as in \eqref{eq:Sn}. For $1 \leq i \leq n$, define 
$$Z_i:=\frac{X_i-p}{\sqrt{p(1-p)}}.$$
Then $S_n$ can be written as
\begin{align*}
S_n & =(2p-1)\sqrt{p(1-p)} \sum_{i=1}^n  d_{G_n}(i) Z_i +2 p (1-p) \sum_{\{i,j\}\in E(G_n) } Z_iZ_j  \nonumber \\ 
& = \sum_{i=1}^n c_i Z_i + \sum_{1 \leq i < j \leq n} b_{ij} Z_iZ_j  
\end{align*}  
where $b_{ij} = 2 p (1-p) \bm 1\{\{i,j\}\in E(G_n) \}$ and $c_i=(2p-1)\sqrt{p(1-p)}  d_{G_n}(i) $. 
Note that the $\{Z_i\}_{1 \leq i \leq n}$ and $\{Z_i Z_j\}_{1 \leq i <j \leq n}$, are mutually uncorrelated, and each has second moment one. Hence, 
$$\sigma_n^2 = \Var{S_n} =\sum_{i=1}^n c_i^2+\sum_{1 \leq i<j \leq n}b_{ij}^2.$$
Since $U_n = \frac{S_n}{\sigma_n}$, the result in \eqref{eq:momentUn} then follows if we show the following: 
\begin{align}\label{eq:Z}
 \E{\left|\sum_{i=1}^n c_i Z_i \right|^6 } \leq C_p' \left( \sum_{i=1}^n c_i^2\right)^3 \text{ and } \E{\left| \sum_{1 \leq i < j \leq n} b_{ij} Z_iZ_j \right|^6 } \leq C_p'' \left( \sum_{1 \leq i < j \leq n} b_{ij}^2 \right)^3  , 
\end{align}
for constants $C_p', C_p'' >0$, depending only on $p$. 

To establish the first inequality in \eqref{eq:Z}, note that every term in the sixth moment expansion in which an index appears exactly once has expectation zero. Hence, 
\begin{align*}
& \mathbb{E}\left|\sum_{i=1}^n c_i Z_i\right|^6  \nonumber  \\  
& \le 90 M_p^6\left( \sum_{i=1}^n |c_i|^6 + \sum_{1 \leq i\ne j \leq n}|c_i|^4|c_j|^2 + \sum_{1 \leq i<j \leq n}|c_i|^3|c_j|^3 + \sum_{1 \leq i<j<k \leq n}|c_i|^2|c_j|^2|c_k|^2 \right)  ,  
\end{align*}
where we use the bound $|Z_i|\le M_p := \frac{\max\{p, 1-p\}}{\sqrt{p (1-p)}}$. Using the bound $\sum_{i=1}^n |c_i|^r \le (\sum_{i=1}^n c_i^2)^{\frac{r}{2}}$, for  $r\ge2$, the first inequality in \eqref{eq:Z} follows.

For the second inequality, define $\beta:=2p(1-p)$ and 
$$Q_n:=\sum_{1 \leq i<j \leq n}b_{ij}Z_iZ_j = \beta\sum_{\{i,j\}\in E(G_n)}Z_iZ_j.$$
Expanding the sixth power gives
$$\mathbb{E}|Q_n|^6 =\beta^6 \sum_{\{i_1, j_1\}, \{i_2, j_2\}, \ldots, \{i_6, j_6\} \in E(G_n)}
\mathbb{E}\left[ \prod_{r=1}^6 Z_{i_r}Z_{j_r} \right] . $$
Associate with each ordered tuple of edges $\{i_1, j_1\}, \{i_2, j_2\}, \ldots, \{i_6, j_6\}$ the
loop-less multigraph $H$ formed by these six edges, retaining
their multiplicities. Then 
$$E_H:= \mathbb{E}\left[\prod_{r=1}^6 Z_{i_r}Z_{j_r} \right] =\prod_{v\in V(H)}\mathbb{E}[Z_v^{d_H(v)}],$$
where $d_H(v)$ is the degree of the vertex $v$ in $H$. Note that $E_H = 0$ whenever $H$ has a vertex of degree one and $|E_H| \leq M_p^{12}$. Denote by $\mathcal H_6$ the collection of all multigraphs with six edges up to isomorphism such that $\min_{v \in V(H)} d_H(v) \geq 2$. Then, by the above discussion, 
\begin{align}\label{eq:Qn}
\mathbb{E}|Q_n|^6 \leq A_p \sum_{H \in \mathcal H_6} N(H, G_n)  ,  
\end{align}
for some constant $A_p > 0$. We claim that there exists a constant $C > 0$ such that   
\begin{align}\label{eq:NHGnmoment}
 N(H, G_n) \leq C \NH{\Ktwo}^3, \quad \text{ for every } H \in \mathcal H_6  .  
\end{align}
Note that the claim follows immediately if every $H \in \mathcal H_6$ has an edge-cover consisting of at most three edges. (Recall that an {\it edge cover} of $H$ is a collection of edges of $H$ incident to every vertex). Through the following cases we show that such an edge cover exists, with one exception.
\begin{itemize}

\item $|V(H)|\le4$: In this case a spanning forest of its underlying simple graph is an edge cover with at most three edges.

\item $|V(H)|=5$. In this case, the underlying simple graph contains a matching of size two. If not, the edges of the underlying simple graph would pairwise intersect. A family of pairwise intersecting edges either has a common endpoint or is contained in a triangle. Since all five vertices are incident to an edge, the underlying graph would therefore be the 4-star $K_{1,4}$. Each of its four leaves would require an incident edge of multiplicity at least two, contradicting the fact that $H$ has only six edges. Hence, the two edges forming the  matching together with any edge incident to the remaining vertex, give the required three-edge cover.

\item If $|V(H)|=6$, then every vertex of $H$ has degree exactly two, because the sum of the degrees is twelve. Hence, $H$ is a disjoint union of cycles, with a doubled edge regarded as a cycle of length two. Therefore, the possibilities are a 6-cycle, a disjoint union of a 4-cycle and a 2-cycle, a disjoint union of two 3-cycles ($\Kthree \sqcup \Kthree$), and a disjoint union of three 2-cycles. All of these except $\Kthree \sqcup \Kthree$ have an edge cover consisting of three edges.
\end{itemize}

To deal with the case $H= \Kthree \sqcup \Kthree$, first recall that by the Kruskal-Katona theorem (see, for example, \cite[Equation~(2.3)]{lovasz_book}), 
$$\NH{\Kthree} \leq C' \NH{\Ktwo}^\frac{3}{2},$$
for some constant $C' > 0$. Hence, 
$$\NH{\Kthree \sqcup \Kthree} \leq  2 \NH{\Kthree}^2 \leq C \NH{\Ktwo}^3.$$
The above discussion shows that \eqref{eq:NHGnmoment} holds for $H \in \mathcal H_6$. Therefore, using \eqref{eq:Qn},  
$$\mathbb{E}|Q_n|^6 \le C A_p \NH{\Ktwo}^3 =C_p'' \left(\sum_{1 \leq i<j \leq n}  b_{ij}^2\right)^3,$$
which proves the second inequality in \eqref{eq:Z}.  
\end{proof}

\section{Proof of \cref{thm:cltH0} }
\label{sec:proofs-knn}

Recall that given $L_n \sim \mathrm{Poi}(n_1+n_2)$,  $\mathcal Z_{n}=\{Z_{1}, Z_{2}, \ldots, Z_{L_{n}}\}$ is an i.i.d. collection from the distribution with density $\frac{n_{1}}{n} f+ \frac{n_{2}}{n} g$. Then $\mathcal Z_{n}$ is a nonhomogeneous Poisson process in $\R^{d}$ with rate function $n \phi_n=n_{1} f+n_{2} g$.  Denote by $\mathcal G_K(\mathcal Z_{n})$ the $K$-NN graph on the pooled sample $\mathcal Z_{n}$. The following establishes bounds on the relevant motif counts in $\mathcal G_K(\mathcal Z_{n})$.

\begin{lemma}\label{lem:NHK}
For an integer $K\ge1$, let $\tilde K:=\min\{K,\max\{L_n-1,0\}\}$. There exists a constant $B_d>0$, depending only on $d$, such that the following hold almost surely, 
\begin{itemize}

\item[$(i)$] $\frac12 L_n\tilde K \le \NHK{\Ktwo}\le L_n\tilde K$,

\item[$(ii)$] $L_n\binom{\tilde K}{2} \le \NHK{\Pthree} \le \frac{1}{2} B_d^2 L_n\tilde K^2$,

\item[$(iii)$] $\NHK{\Konefour} \le \frac{1}{24} B_d^4  L_n\tilde K^4$.  
\end{itemize}
Further, for every $K\ge1$, 
\begin{align}\label{eq:NHK}
\frac12L_n\tilde K^2 \le \NHK{\Ktwo}+\NHK{\Pthree} \le \frac{1}{2} B_d^2 L_n\tilde K^2.
\end{align}   
\end{lemma}

\begin{proof}
Note that the bounds are immediate if $L_n\le1$. Therefore, we assume $L_n\ge2$. We will use the following standard degree-bound: For $x \in\mathcal Z_n $, the degree $d_x$ of the vertex $x$ in the graph $\mathcal G_K(\mathcal Z_{n})$ satisfies: 
\begin{align}\label{eq:vertexdegree}
\tilde K\le d_x\le B_d \tilde K, 
\end{align}
for some constant $B_d > 0 $, depending only on $d$. (Note that the lower bound is immediate from the definition and for the upper bound see \cite[Proposition 12, Section 11]{Stone1977} and also \citet[Chapter~6]{Gyorfi2002}).

Using the degree bound the proof can be completed as follows:

\begin{itemize}

\item There are exactly $L_n\tilde K$ directed nearest-neighbor
relations. Each undirected edge corresponds to either one or
two such relations, giving $\frac12L_n\tilde K \le \NHK{\Ktwo}\le L_n\tilde K$.

\item For $\Pthree$ the  degree bound gives, 
$$L_n\binom{\tilde K}{2} \le \NHK{\Pthree} =\sum_{x\in\mathcal Z_n}\binom{d_x }{2} \le \frac{B_d^2}{2}L_n\tilde K^2.$$

\item For $\Konefour$ the  degree bound gives, 
$$\NHK{\Konefour}=\sum_{x\in\mathcal Z_n}\binom{d_x}{4} \le \frac{B_d^4}{24}L_n\tilde K^4.$$  
\end{itemize}
Also, \eqref{eq:NHK} follows from the degree bound and the observation: 
$$\NHK{\Ktwo}+\NHK{\Pthree} =\frac12\sum_{x\in\mathcal Z_n} d_x^2.$$
This completes the proof. 
\end{proof}

To prove \cref{thm:cltH0} it is more convenient to work with $T^\circ(\mathcal G_K(\mathcal Z_{n}))$, the number of monochromatic edges (recall \eqref{eq:Tpoissonm}). Note that 
\begin{equation*}  
U_n^\times(\mathcal G_K(\mathcal Z_{n}))  = \frac{T^{\times}(\mathcal G_K(\mathcal Z_{n})) - \mu_{\mathscr F}^\times }{\sigma_{\mathscr F} }  = - \frac{T^{\circ}(\mathcal G_K(\mathcal Z_{n})) - \mu_{\mathscr F}^\circ }{\sigma_{\mathscr F} }  ,  
\end{equation*}
where 
$$\mu_{\mathscr F}^\circ := \E{T^{\circ}(\mathcal G_K(\mathcal Z_{n}))| \mathscr F} = (\lambda_n^2 + (1-\lambda_n)^2) \NHK{\Ktwo} = \NHK{\Ktwo}  -  \mu_{\mathscr F}^\times.$$ We will now apply the results in \cref{sec:wass} by conditioning on sigma-algebra $\mathscr F= \sigma(L_{n}, \mathcal Z_n)$.  For this, we will require the following notion of the conditional Wasserstein distance.

\begin{definition}
Given two real-valued random variables $X$ and $Z$ with finite first moments, and a $\sigma$-algebra $\mathcal{F}$, the Wasserstein distance between $X \mid \mathcal F$ (the conditional law of $X$ given the $\sigma$-algebra $\mathcal{F}$) and $Z$ is defined as
$$\mathrm{Wass}(X\mid\mathcal{F}, Z):=\sup_{g \in\mathcal{L}} \left|\E{g(X)\mid\mathcal{F}}  - \E{g(Z)} \right|,$$
where $\mathcal L$ is the collection of 1-Lipschitz functions from $\R \rightarrow \R$.  
\label{definition:cond_wass}
\end{definition}

Recall that $\lambda_n = \frac{n_1}{n}$. First we consider the case $\lambda_n  \rightarrow \lambda \in (0, 1) \backslash \{\frac{1}{2}\}$. Note that conditional on $\mathcal Z_n$, $\mathcal G_K(\mathcal Z_n)$ is a fixed graph with an i.i.d.\ $\Ber(\lambda_n)$ vertex coloring. Hence, from \eqref{eq:WNUndense}, whenever $L_{n} \geq 2$, 
\begin{align*}
& \Wass(U_n^\times(\mathcal G_K(\mathcal Z_{n}))|\mathscr F, N(0,1))  \nonumber  \\   
& \le C_{\lambda} \left( \frac{1}{\sqrt{\NHK{\Ktwo}+\NHK{\Pthree}}} +  \frac{\sqrt{\NHK{\Konefour}}}{\NHK{\Ktwo}+\NHK{\Pthree}} \right) ,
\end{align*}
where $C_{\lambda} > 0$ is a constant depending only on $\lambda$. Applying the bounds in \cref{lem:NHK} then gives, 
\begin{align}\label{eq:Wp}
 \Wass(U_n^\times(\mathcal G_K(\mathcal Z_{n}))|\mathscr F, N(0,1)) & \le \frac{C_{\lambda, d}}{\sqrt{L_n}}  ,
\end{align}
almost surely, for some constant $C_{\lambda, d} > 0$, that depends only on $d$ and $\lambda$.

Next, we consider the case where $n_1 = n_2$. This means, $\lambda_n = \frac{1}{2}$. In this case, by \eqref{eq:WUnZhalf}, whenever $L_{n} \geq 2$, 
\begin{align*}
\Wass(U_n^\times(\mathcal G_K(\mathcal Z_{n}))|\mathscr F, N(0,1)) \le 4\sqrt{C} \left( \NHK{\Ktwo}^{-\frac{1}{2}}+\frac{\NHK{\Cfour}}{\NHK{\Ktwo}^2} \right)^{\frac{1}{10}}.
\end{align*}
 Applying the bounds in \cref{lem:NHK} and $\NHK{\Cfour} = O(\NHK{\Ktwo} \tilde K^2) = O(L_n \tilde K^3)$ gives, 
\begin{align}\label{eq:Wphalf}
\Wass(U_n^\times(\mathcal G_K(\mathcal Z_{n}))|\mathscr F, N(0,1)) \le \overline{C}_d \left( \frac{1}{\sqrt{L_n\tilde K}} +\frac{ \tilde K}{L_n}  \right)^{\frac{1}{10}}  ,  
\end{align}
where $\overline{C}_d> 0 $ is a constant depending on $d$.

To remove the conditioning, define $\mathcal E_n:=\{L_n\geq n/2\}$. Then,  
$$ \operatorname{Wass}(U_n^\times(\mathcal G_K(\mathcal Z_{n})),N(0,1)) \leq \E{\operatorname{Wass}(U_n^\times(\mathcal G_K(\mathcal Z_{n})) \mid\mathscr F, N(0,1)) \bm 1\{ \mathcal E_n \}} +2\mathbb P(\mathcal E_n^c) , $$
where for the second term we use the bound $\operatorname{Wass}(U_n^\times(\mathcal G_K(\mathcal Z_{n})) | \mathscr F, N(0,1)) \leq 2$. Since $\mathbb E[L_n]=  \Var{L_n} = n$, Chebyshev's inequality gives $\mathbb P(\mathcal E_n^c) \leq \mathbb P(|L_n-n|\geq n/2) \leq \frac{4}{n}$. Combining the above with \eqref{eq:Wp} shows that $\operatorname{Wass}(U_n^\times(\mathcal G_K(\mathcal Z_{n})),N(0,1)) \rightarrow 0$, for $\lambda \ne \frac{1}{2}$ and irrespective of the choice of $K$. For $n_1=n_2$, the same argument with \eqref{eq:Wphalf} shows that $\operatorname{Wass}(U_n^\times(\mathcal G_K(\mathcal Z_{n})),N(0,1)) \rightarrow 0$, whenever $K= o(n)$.  \hfill  $\Box$

\section{Examples}  
\label{sec:examples}

In this section, we present various examples illustrating the necessity of the conditions in the preceding results. We begin with an example showing that the fourth moment condition is not enough for the asymptotic normality of $U_n$ when $p = \frac{1 \pm 1/\sqrt{3}}{2}$. Recall from \cref{thm:Un} $(ii)$, that for these two values of $p$, the third moment condition $\E{U_n^3} \rightarrow 0$ is necessary and sufficient for the asymptotic normality of $U_n$. Departing slightly from the notation in the preceding sections, we use $n$ here to index the graph sequence rather than to denote the number of vertices.

\begin{example}[The fourth moment condition does not suffice for $p = \frac{1 \pm 1/\sqrt{3}}{2}$]\label{example:fourth}
    For $n\ge 1$, let $G_{n}=K_{1,n}$ on the vertex set $\{0,1,\ldots,n\}$, with the central vertex labeled $0$. Assume $p = \frac{1 + 1/\sqrt{3}}{2}$. The case $p=\frac{1 - 1/\sqrt{3}}{2}$ is similar. Then $\E{U_n^4} \rightarrow 3$, from  \cref{lem:fm_pneqhalf} (b). However, in this case $U_n$ does not have a limiting normal distribution. In particular, we claim that, as $n\to \infty$,    
    \begin{equation} \label{eq:fmp_counter}
        U_n \xrightarrow{D} p  \delta_a + (1-p)  \delta_b,
        \quad \text{ where } a = \frac{\sqrt{3}-1}{\sqrt 2} \text{ and } b = -\frac{\sqrt{3}+1}{\sqrt 2}.
    \end{equation}
    To see this, define $R_i = 2X_i - 1$, for $0\le i\le n$. Note that $\{R_i\}_{0 \leq i \leq n}$ are i.i.d.\ $\pm1$-valued with mean $1/\sqrt3$ and variance $\frac{2}{3}$. Since $\bm{1}\{X_0=X_i\} = (1+R_0R_i)/2$ and $\E{R_0R_i} = \frac{1}{3}$,
    \begin{align*}
    U_n &= \frac{1}{2\sigma_n} \sum_{i=1}^n \left(R_0 R_i - \frac13\right)
    =\frac{R_0}{2\sigma_n}\left(\sum_{i=1}^n R_i - \frac{n}{\sqrt3}\right)
      + \frac{\left(n R_0/\sqrt{3} - n/3\right)}{ 2 \sigma_n} . 
    \end{align*}
    Note that 
    $$\frac{\left(n R_0/\sqrt{3} - n/3\right)}{ 2 \sigma_n} \xrightarrow{D} p  \delta_a + (1-p)  \delta_b  , $$ 
    for $a$ and $b$ as defined in \eqref{eq:fmp_counter}. Further, by the classical CLT,   
    $$\frac{\sum_{i=1}^n R_i - n/\sqrt{3}}{\sqrt{2n/3}} \xrightarrow{D} N(0,1).$$ Since, by  \eqref{eq:var_term}, $\sign = \frac{1}{3\sqrt{2}} n + o(n)$, this implies, 
    $$\frac{R_0}{2\sigma_n}\left(\sum_{i=1}^n R_i - \frac{n}{\sqrt3}\right) \xrightarrow{P} 0  .  $$ Combining the above proves the result in \eqref{eq:fmp_counter}. 
    \end{example}

Next, we show that the third moment condition is not enough for the asymptotic normality of $U_n$ when $p = \frac{1}{2}$. Recall from \cref{thm:Un} (iii) that in this case, the fourth moment condition $\E{U_n^4} \rightarrow 3$ is necessary and sufficient for the asymptotic normality of $U_n$.

\begin{example}[The third moment condition does not suffice for $p= \frac{1}{2}$]
\label{example:third}
For $n\ge 1$, let $G_{n}=K_{2, n}$, the complete bipartite graph with two vertices on one side and $n$ vertices on the other side. Label the vertices in the two sides by $\{1', 2'\}$ and $\{1, \ldots,n\}$, respectively. Assume $p = \frac{1}{2}$.  Then, since $K_{2, n}$ is triangle-free, by  \cref{lem:thirdm_pneqhalf} (b), $\E{U_n^3} \rightarrow 0$. 
However, in this case $U_n$ does not have a limiting normal distribution. To see this, note that there are two possibilities: 
\begin{itemize}

\item If the vertices $\{1', 2'\}$ receive different colors (which happens with probability $\frac{1}{2}$), then the number of monochromatic edges $T_n = n$.  

\item If the vertices $\{1', 2'\}$ receive the same color (which also happens with probability $\frac{1}{2}$), then the number of monochromatic edges $T_n \stackrel{D} = 2 M_n$, where $M_n\sim \mathrm{Bin}(n, \frac{1}{2})$. 
\end{itemize} 
Hence, $\E{T_n} = n$, $\sigma_n^2 = \Var{T_n} = \frac{n}{2}$ (recall \eqref{eq:var_term}), and 
$$U_n = \frac{T_n - \E{T_n}}{\sigma_n} \stackrel{D}{\rightarrow}\frac{1}{2}N(0,2)+\frac{1}{2}\delta_0,$$
which is a mixture of a normal distribution and a point mass at 0. 
\end{example}

The next example shows that in the sparse regime $p_n \rightarrow 0$, as $n \rightarrow \infty$, the third moment condition, although sufficient, is not necessary for the asymptotic normality of $U_n$.

\begin{example}[Third moment condition not necessary in the sparse regime]
Suppose $G_n$ consists of an $n^2$-star $K_{1, n^2}$ and $n^5$ disjoint edges. Choose $p_n = \frac{1}{n^3}$. In this case, it is easy to check that condition (b) of \cref{lem:thirdmomentsparse} does not hold. Hence, $\E{U_n^3} \nrightarrow 0$. In particular, by direct calculation using \cref{tab:third} it follows that $\E{U_n^3} \rightarrow -\frac{1}{2\sqrt2}$. However, $U_n$ is asymptotically normal. To see this, recall the definition of $T_n$ from \eqref{eq:Tn} and denote 
$$T_n=A_n+B_n,$$ 
where $A_n$ and $B_n$ count the number of monochromatic edges in the star and the disjoint edges, respectively. Let $u_n=(2p_n-1)^2$. Note that $B_n\sim\operatorname{Bin}(n^5,p_n^2+(1-p_n)^2)$ and $\Var{B_n}\sim2n^2$. Also,   from \eqref{eq:varianceUn} we have,  
$$\Var{A_n}=p_n(1-p_n)(n^2+u_nn^4) \sim n.$$
Since $A_n$ and $B_n$ are independent, $\sign^2 = \Var{T_n} \sim2n^2$. Hence,
$$\E{\left(\frac{A_n-\E{A_n}}{\sign}\right)^2} =\frac{\Var{A_n}}{\sign^2} \rightarrow0.$$
Therefore, by the CLT and Slutsky's theorem, 
$$U_n =\frac{A_n-\E{A_n}}{\sign} +\frac{B_n-\E{B_n}}{\sign} \xrightarrow{D} N(0,1).$$ 
Hence, $U_n \xrightarrow{D} N(0,1)$, even though $\E{U_n^3} \nrightarrow 0$.   \label{example:thirdsparse}
\end{example}

Finally, we give an example in the sparse regime where the third moment condition holds, hence, $U_n$ is asymptotically normal, but the fourth moment condition does not. We show in \cref{ppn:thirdfourthsparse} that the fourth moment condition always implies the third moment condition in the sparse case. This together with the following example implies that the third moment condition is strictly weaker than the fourth moment condition in the sparse regime.

\begin{example}[Third moment condition does not imply the fourth moment condition]
Suppose $G_n$ consists of an $n$-star $K_{1, n}$ and $n^3$ disjoint edges. Choose $p_n = \frac{1}{n^2}$. In this case, it is easy to check that condition (b) of \cref{lem:thirdmomentsparse} holds, which implies $$\E{U_n^3} \rightarrow 0.$$ Hence, by \cref{thm:clt} $U_n \stackrel{D} \rightarrow N(0, 1)$. However, condition (b) of \cref{lem:fourthmomentsparse} does not hold. Hence, $\E{U_n^4} \nrightarrow 3$. In particular, by direct calculation using \cref{tab:fourth} it follows that $\E{U_n^4} \rightarrow\frac{13}{4}$. This shows that the third moment condition does not imply the fourth moment condition in the sparse regime.  
\label{example:sparsenormal}
\end{example}

\section{Simulations}
\label{sec:experiments}

In this section, we investigate the effect of $K$, the number of nearest neighbors, on the empirical performance of the $K$-NN-based two-sample test discussed in \cref{sec:knn}. We consider the following two examples. In both cases, the sample sizes are chosen to be independent Poisson random variables, each with mean $100$, the nominal level $\alpha=0.05$ and $K\in\{5,10,15\}$.  We construct a two-sample test with precise control on the asymptotic level as follows: Fix $\alpha \in (0, 1)$ and consider the test function that rejects the null hypothesis $H_0$ as follows:  
\begin{align}  
\label{eq:Knnthreshold}
\psi_n := \bm 1\{ T^{\times}(\mathcal G_K(\mathcal Z_{n})) < \mu_{\mathscr F}^\times - z_{\alpha}\sigma_{\mathscr F} \},
\end{align}
where $z_\alpha$ is the $(1-\alpha)$-th quantile of the standard normal distribution.

We estimate the empirical power of the test $\psi_n$ using $2000$ repetitions for each configuration.

\begin{enumerate}[label=(E\arabic*)]

\item \label{location} \textit{Location shifts}: Here, 
$$P=N(\bm{0},I_{10}) \quad \text{and} \quad Q=N(\mu\bm{u},I_{10}),$$
where $\bm{u}=\frac{1}{\sqrt{10}}(1,\ldots,1)^\top$.
This normalization ensures that $\bm{u}$ is a unit vector, so that
$\mu\geq 0$ is the Euclidean distance between the two means.
\Cref{fig:power}(a) shows the empirical power for a grid of values
of $\mu$ between $0$ and $1.5$, for each choice of $K$.
 
\item \label{mixture} \textit{Heavy-tailed mixture alternatives:} Here, 
$$P=N(\bm{0},I_2) \quad \text{and} \quad Q=(1-\varepsilon)N(\bm{0},I_2)+\varepsilon R,$$
where $R$ is the distribution of $\frac{1}{\sqrt 3} (T_1,T_2)^\top$, and $T_1$ and $T_2$ are independent $t$-distributed random variables with $3$ degrees of freedom. Since each $T_i$ has variance $3$, scaling by $\frac{1}{\sqrt 3}$ gives each coordinate unit variance. Thus, $P$ and $Q$ both have mean $\bm{0}$ and covariance matrix $I_2$. \Cref{fig:power}(b) shows the empirical power for a grid of values of $\varepsilon$ between $0$ and $1$, for each choice of $K$.
\end{enumerate}

In both examples, the empirical power increases as the alternative moves farther from the null, as expected. It also generally increases with $K$, illustrating the benefit of using more neighbors in these settings. More generally, the dependence of power on $K$ is subtle and can vary with the underlying distributions, leading to different detection rates. This dependence was investigated by \citet{Bhattacharya2020} for fixed $K$, and recently extended by \citet{Kanekar2025} to the regime $K=o(n^{\frac{1}{4}})$. A more complete understanding of how $K$ affects power and detection rates for specific alternatives remains an interesting direction for future research.

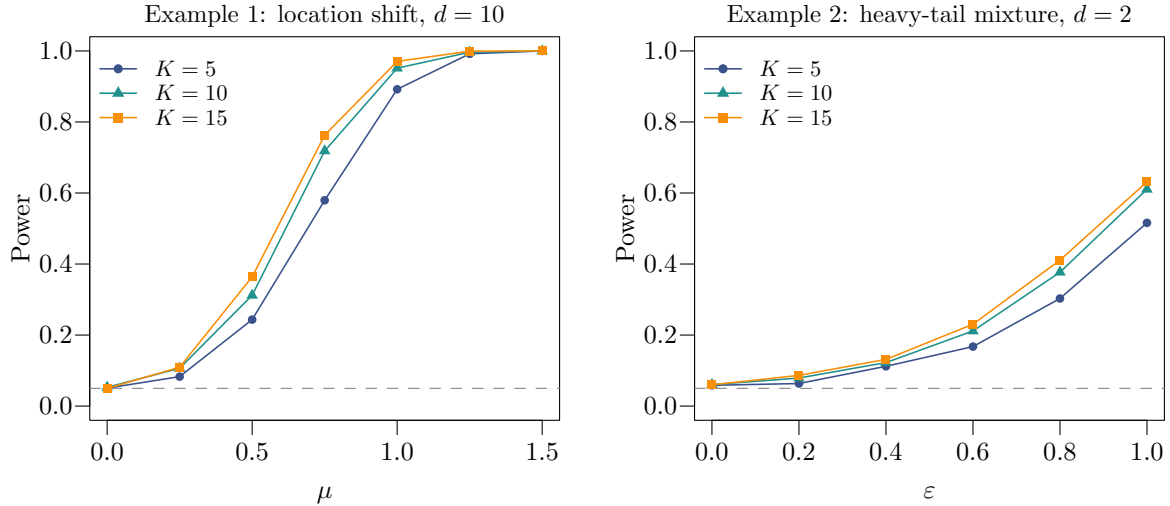
\begin{figure}[h]
\centering
\begin{tikzpicture}[x=1pt,y=1pt]
\definecolor{fillColor}{RGB}{255,255,255}
\path[use as bounding box,fill=fillColor,fill opacity=0.00] (0,0) rectangle (216.81,195.13);
\begin{scope}
\path[clip] (  0.00,  0.00) rectangle (216.81,195.13);
\definecolor{drawColor}{RGB}{0,0,0}

\path[draw=drawColor,line width= 0.4pt,line join=round,line cap=round] ( 40.03, 33.48) -- (203.78, 33.48);

\path[draw=drawColor,line width= 0.4pt,line join=round,line cap=round] ( 40.03, 33.48) -- ( 40.03, 28.08);

\path[draw=drawColor,line width= 0.4pt,line join=round,line cap=round] ( 94.61, 33.48) -- ( 94.61, 28.08);

\path[draw=drawColor,line width= 0.4pt,line join=round,line cap=round] (149.20, 33.48) -- (149.20, 28.08);

\path[draw=drawColor,line width= 0.4pt,line join=round,line cap=round] (203.78, 33.48) -- (203.78, 28.08);

\node[text=drawColor,anchor=base,inner sep=0pt, outer sep=0pt, scale=  0.90] at ( 40.03, 18.36) {0.0};

\node[text=drawColor,anchor=base,inner sep=0pt, outer sep=0pt, scale=  0.90] at ( 94.61, 18.36) {0.5};

\node[text=drawColor,anchor=base,inner sep=0pt, outer sep=0pt, scale=  0.90] at (149.20, 18.36) {1.0};

\node[text=drawColor,anchor=base,inner sep=0pt, outer sep=0pt, scale=  0.90] at (203.78, 18.36) {1.5};

\path[draw=drawColor,line width= 0.4pt,line join=round,line cap=round] ( 33.48, 38.83) -- ( 33.48,172.50);

\path[draw=drawColor,line width= 0.4pt,line join=round,line cap=round] ( 33.48, 38.83) -- ( 28.08, 38.83);

\path[draw=drawColor,line width= 0.4pt,line join=round,line cap=round] ( 33.48, 65.56) -- ( 28.08, 65.56);

\path[draw=drawColor,line width= 0.4pt,line join=round,line cap=round] ( 33.48, 92.30) -- ( 28.08, 92.30);

\path[draw=drawColor,line width= 0.4pt,line join=round,line cap=round] ( 33.48,119.03) -- ( 28.08,119.03);

\path[draw=drawColor,line width= 0.4pt,line join=round,line cap=round] ( 33.48,145.77) -- ( 28.08,145.77);

\path[draw=drawColor,line width= 0.4pt,line join=round,line cap=round] ( 33.48,172.50) -- ( 28.08,172.50);

\node[text=drawColor,anchor=base east,inner sep=0pt, outer sep=0pt, scale=  0.90] at ( 27.00, 35.73) {0.0};

\node[text=drawColor,anchor=base east,inner sep=0pt, outer sep=0pt, scale=  0.90] at ( 27.00, 62.46) {0.2};

\node[text=drawColor,anchor=base east,inner sep=0pt, outer sep=0pt, scale=  0.90] at ( 27.00, 89.20) {0.4};

\node[text=drawColor,anchor=base east,inner sep=0pt, outer sep=0pt, scale=  0.90] at ( 27.00,115.93) {0.6};

\node[text=drawColor,anchor=base east,inner sep=0pt, outer sep=0pt, scale=  0.90] at ( 27.00,142.67) {0.8};

\node[text=drawColor,anchor=base east,inner sep=0pt, outer sep=0pt, scale=  0.90] at ( 27.00,169.40) {1.0};

\path[draw=drawColor,line width= 0.4pt,line join=round,line cap=round] ( 33.48, 33.48) --
	(210.33, 33.48) --
	(210.33,177.85) --
	( 33.48,177.85) --
	cycle;
\end{scope}
\begin{scope}
\path[clip] (  0.00,  0.00) rectangle (216.81,195.13);
\definecolor{drawColor}{RGB}{0,0,0}

\node[text=drawColor,anchor=base,inner sep=0pt, outer sep=0pt, scale=  0.85] at (121.91,183.54) {Example 1: location shift, $d=10$};

\node[text=drawColor,anchor=base,inner sep=0pt, outer sep=0pt, scale=  0.90] at (121.91,  4.32) {$\mu$};

\node[text=drawColor,rotate= 90.00,anchor=base,inner sep=0pt, outer sep=0pt, scale=  0.90] at ( 10.80,105.66) {Power};
\end{scope}
\begin{scope}
\path[clip] ( 33.48, 33.48) rectangle (210.33,177.85);
\definecolor{drawColor}{gray}{0.50}

\path[draw=drawColor,line width= 0.4pt,dash pattern=on 4pt off 4pt ,line join=round,line cap=round] ( 33.48, 45.51) -- (210.33, 45.51);
\definecolor{drawColor}{RGB}{59,82,139}

\path[draw=drawColor,line width= 0.6pt,line join=round,line cap=round] ( 40.03, 45.58) --
	( 67.32, 49.92) --
	( 94.61, 71.38) --
	(121.91,116.29) --
	(149.20,158.07) --
	(176.49,171.43) --
	(203.78,172.50);
\definecolor{fillColor}{RGB}{59,82,139}

\path[fill=fillColor] ( 40.03, 45.58) circle (  1.62);

\path[fill=fillColor] ( 67.32, 49.92) circle (  1.62);

\path[fill=fillColor] ( 94.61, 71.38) circle (  1.62);

\path[fill=fillColor] (121.91,116.29) circle (  1.62);

\path[fill=fillColor] (149.20,158.07) circle (  1.62);

\path[fill=fillColor] (176.49,171.43) circle (  1.62);

\path[fill=fillColor] (203.78,172.50) circle (  1.62);
\definecolor{drawColor}{RGB}{33,145,140}

\path[draw=drawColor,line width= 0.6pt,line join=round,line cap=round] ( 40.03, 45.91) --
	( 67.32, 53.06) --
	( 94.61, 80.53) --
	(121.91,134.87) --
	(149.20,166.02) --
	(176.49,172.03) --
	(203.78,172.50);
\definecolor{fillColor}{RGB}{33,145,140}

\path[fill=fillColor] ( 40.03, 48.43) --
	( 42.21, 44.65) --
	( 37.85, 44.65) --
	cycle;

\path[fill=fillColor] ( 67.32, 55.58) --
	( 69.50, 51.80) --
	( 65.14, 51.80) --
	cycle;

\path[fill=fillColor] ( 94.61, 83.05) --
	( 96.80, 79.27) --
	( 92.43, 79.27) --
	cycle;

\path[fill=fillColor] (121.91,137.39) --
	(124.09,133.61) --
	(119.72,133.61) --
	cycle;

\path[fill=fillColor] (149.20,168.54) --
	(151.38,164.76) --
	(147.01,164.76) --
	cycle;

\path[fill=fillColor] (176.49,174.55) --
	(178.67,170.77) --
	(174.31,170.77) --
	cycle;

\path[fill=fillColor] (203.78,175.02) --
	(205.96,171.24) --
	(201.60,171.24) --
	cycle;
\definecolor{drawColor}{RGB}{249,142,9}

\path[draw=drawColor,line width= 0.6pt,line join=round,line cap=round] ( 40.03, 45.51) --
	( 67.32, 53.46) --
	( 94.61, 87.48) --
	(121.91,140.82) --
	(149.20,168.56) --
	(176.49,172.44) --
	(203.78,172.50);
\definecolor{fillColor}{RGB}{249,142,9}

\path[fill=fillColor] ( 38.41, 43.89) --
	( 41.65, 43.89) --
	( 41.65, 47.13) --
	( 38.41, 47.13) --
	cycle;

\path[fill=fillColor] ( 65.70, 51.84) --
	( 68.94, 51.84) --
	( 68.94, 55.08) --
	( 65.70, 55.08) --
	cycle;

\path[fill=fillColor] ( 92.99, 85.86) --
	( 96.23, 85.86) --
	( 96.23, 89.10) --
	( 92.99, 89.10) --
	cycle;

\path[fill=fillColor] (120.28,139.20) --
	(123.53,139.20) --
	(123.53,142.44) --
	(120.28,142.44) --
	cycle;

\path[fill=fillColor] (147.58,166.94) --
	(150.82,166.94) --
	(150.82,170.18) --
	(147.58,170.18) --
	cycle;

\path[fill=fillColor] (174.87,170.82) --
	(178.11,170.82) --
	(178.11,174.06) --
	(174.87,174.06) --
	cycle;

\path[fill=fillColor] (202.16,170.88) --
	(205.40,170.88) --
	(205.40,174.12) --
	(202.16,174.12) --
	cycle;
\definecolor{drawColor}{RGB}{59,82,139}

\path[draw=drawColor,line width= 0.6pt,line join=round,line cap=round] ( 37.31,165.78) -- ( 51.08,165.78);
\definecolor{drawColor}{RGB}{33,145,140}

\path[draw=drawColor,line width= 0.6pt,line join=round,line cap=round] ( 37.31,156.60) -- ( 51.08,156.60);
\definecolor{drawColor}{RGB}{249,142,9}

\path[draw=drawColor,line width= 0.6pt,line join=round,line cap=round] ( 37.31,147.42) -- ( 51.08,147.42);
\definecolor{fillColor}{RGB}{59,82,139}

\path[fill=fillColor] ( 44.20,165.78) circle (  1.72);
\definecolor{fillColor}{RGB}{33,145,140}

\path[fill=fillColor] ( 44.20,159.28) --
	( 46.52,155.26) --
	( 41.88,155.26) --
	cycle;
\definecolor{fillColor}{RGB}{249,142,9}

\path[fill=fillColor] ( 42.48,145.70) --
	( 45.92,145.70) --
	( 45.92,149.14) --
	( 42.48,149.14) --
	cycle;
\definecolor{drawColor}{RGB}{0,0,0}

\node[text=drawColor,anchor=base west,inner sep=0pt, outer sep=0pt, scale=  0.77] at ( 57.97,163.15) {$K=5$};

\node[text=drawColor,anchor=base west,inner sep=0pt, outer sep=0pt, scale=  0.77] at ( 57.97,153.97) {$K=10$};

\node[text=drawColor,anchor=base west,inner sep=0pt, outer sep=0pt, scale=  0.77] at ( 57.97,144.79) {$K=15$};
\end{scope}
\end{tikzpicture}
\hfill
\begin{tikzpicture}[x=1pt,y=1pt]
\definecolor{fillColor}{RGB}{255,255,255}
\path[use as bounding box,fill=fillColor,fill opacity=0.00] (0,0) rectangle (216.81,195.13);
\begin{scope}
\path[clip] (  0.00,  0.00) rectangle (216.81,195.13);
\definecolor{drawColor}{RGB}{0,0,0}

\path[draw=drawColor,line width= 0.4pt,line join=round,line cap=round] ( 40.03, 33.48) -- (203.78, 33.48);

\path[draw=drawColor,line width= 0.4pt,line join=round,line cap=round] ( 40.03, 33.48) -- ( 40.03, 28.08);

\path[draw=drawColor,line width= 0.4pt,line join=round,line cap=round] ( 72.78, 33.48) -- ( 72.78, 28.08);

\path[draw=drawColor,line width= 0.4pt,line join=round,line cap=round] (105.53, 33.48) -- (105.53, 28.08);

\path[draw=drawColor,line width= 0.4pt,line join=round,line cap=round] (138.28, 33.48) -- (138.28, 28.08);

\path[draw=drawColor,line width= 0.4pt,line join=round,line cap=round] (171.03, 33.48) -- (171.03, 28.08);

\path[draw=drawColor,line width= 0.4pt,line join=round,line cap=round] (203.78, 33.48) -- (203.78, 28.08);

\node[text=drawColor,anchor=base,inner sep=0pt, outer sep=0pt, scale=  0.90] at ( 40.03, 18.36) {0.0};

\node[text=drawColor,anchor=base,inner sep=0pt, outer sep=0pt, scale=  0.90] at ( 72.78, 18.36) {0.2};

\node[text=drawColor,anchor=base,inner sep=0pt, outer sep=0pt, scale=  0.90] at (105.53, 18.36) {0.4};

\node[text=drawColor,anchor=base,inner sep=0pt, outer sep=0pt, scale=  0.90] at (138.28, 18.36) {0.6};

\node[text=drawColor,anchor=base,inner sep=0pt, outer sep=0pt, scale=  0.90] at (171.03, 18.36) {0.8};

\node[text=drawColor,anchor=base,inner sep=0pt, outer sep=0pt, scale=  0.90] at (203.78, 18.36) {1.0};

\path[draw=drawColor,line width= 0.4pt,line join=round,line cap=round] ( 33.48, 38.83) -- ( 33.48,172.50);

\path[draw=drawColor,line width= 0.4pt,line join=round,line cap=round] ( 33.48, 38.83) -- ( 28.08, 38.83);

\path[draw=drawColor,line width= 0.4pt,line join=round,line cap=round] ( 33.48, 65.56) -- ( 28.08, 65.56);

\path[draw=drawColor,line width= 0.4pt,line join=round,line cap=round] ( 33.48, 92.30) -- ( 28.08, 92.30);

\path[draw=drawColor,line width= 0.4pt,line join=round,line cap=round] ( 33.48,119.03) -- ( 28.08,119.03);

\path[draw=drawColor,line width= 0.4pt,line join=round,line cap=round] ( 33.48,145.77) -- ( 28.08,145.77);

\path[draw=drawColor,line width= 0.4pt,line join=round,line cap=round] ( 33.48,172.50) -- ( 28.08,172.50);

\node[text=drawColor,anchor=base east,inner sep=0pt, outer sep=0pt, scale=  0.90] at ( 27.00, 35.73) {0.0};

\node[text=drawColor,anchor=base east,inner sep=0pt, outer sep=0pt, scale=  0.90] at ( 27.00, 62.46) {0.2};

\node[text=drawColor,anchor=base east,inner sep=0pt, outer sep=0pt, scale=  0.90] at ( 27.00, 89.20) {0.4};

\node[text=drawColor,anchor=base east,inner sep=0pt, outer sep=0pt, scale=  0.90] at ( 27.00,115.93) {0.6};

\node[text=drawColor,anchor=base east,inner sep=0pt, outer sep=0pt, scale=  0.90] at ( 27.00,142.67) {0.8};

\node[text=drawColor,anchor=base east,inner sep=0pt, outer sep=0pt, scale=  0.90] at ( 27.00,169.40) {1.0};

\path[draw=drawColor,line width= 0.4pt,line join=round,line cap=round] ( 33.48, 33.48) --
	(210.33, 33.48) --
	(210.33,177.85) --
	( 33.48,177.85) --
	cycle;
\end{scope}
\begin{scope}
\path[clip] (  0.00,  0.00) rectangle (216.81,195.13);
\definecolor{drawColor}{RGB}{0,0,0}

\node[text=drawColor,anchor=base,inner sep=0pt, outer sep=0pt, scale=  0.85] at (121.91,183.54) {Example 2: heavy-tail mixture, $d=2$};

\node[text=drawColor,anchor=base,inner sep=0pt, outer sep=0pt, scale=  0.90] at (121.91,  4.32) {$\varepsilon$};

\node[text=drawColor,rotate= 90.00,anchor=base,inner sep=0pt, outer sep=0pt, scale=  0.90] at ( 10.80,105.66) {Power};
\end{scope}
\begin{scope}
\path[clip] ( 33.48, 33.48) rectangle (210.33,177.85);
\definecolor{drawColor}{gray}{0.50}

\path[draw=drawColor,line width= 0.4pt,dash pattern=on 4pt off 4pt ,line join=round,line cap=round] ( 33.48, 45.51) -- (210.33, 45.51);
\definecolor{drawColor}{RGB}{59,82,139}

\path[draw=drawColor,line width= 0.6pt,line join=round,line cap=round] ( 40.03, 46.65) --
	( 72.78, 47.32) --
	(105.53, 53.80) --
	(138.28, 61.22) --
	(171.03, 79.33) --
	(203.78,107.80);
\definecolor{fillColor}{RGB}{59,82,139}

\path[fill=fillColor] ( 40.03, 46.65) circle (  1.62);

\path[fill=fillColor] ( 72.78, 47.32) circle (  1.62);

\path[fill=fillColor] (105.53, 53.80) circle (  1.62);

\path[fill=fillColor] (138.28, 61.22) circle (  1.62);

\path[fill=fillColor] (171.03, 79.33) circle (  1.62);

\path[fill=fillColor] (203.78,107.80) circle (  1.62);
\definecolor{drawColor}{RGB}{33,145,140}

\path[draw=drawColor,line width= 0.6pt,line join=round,line cap=round] ( 40.03, 46.91) --
	( 72.78, 49.39) --
	(105.53, 55.20) --
	(138.28, 67.10) --
	(171.03, 89.22) --
	(203.78,120.44);
\definecolor{fillColor}{RGB}{33,145,140}

\path[fill=fillColor] ( 40.03, 49.43) --
	( 42.21, 45.65) --
	( 37.85, 45.65) --
	cycle;

\path[fill=fillColor] ( 72.78, 51.91) --
	( 74.96, 48.13) --
	( 70.60, 48.13) --
	cycle;

\path[fill=fillColor] (105.53, 57.72) --
	(107.71, 53.94) --
	(103.35, 53.94) --
	cycle;

\path[fill=fillColor] (138.28, 69.62) --
	(140.46, 65.84) --
	(136.10, 65.84) --
	cycle;

\path[fill=fillColor] (171.03, 91.74) --
	(173.21, 87.96) --
	(168.85, 87.96) --
	cycle;

\path[fill=fillColor] (203.78,122.95) --
	(205.96,119.18) --
	(201.60,119.18) --
	cycle;
\definecolor{drawColor}{RGB}{249,142,9}

\path[draw=drawColor,line width= 0.6pt,line join=round,line cap=round] ( 40.03, 46.85) --
	( 72.78, 50.39) --
	(105.53, 56.41) --
	(138.28, 69.64) --
	(171.03, 93.77) --
	(203.78,123.24);
\definecolor{fillColor}{RGB}{249,142,9}

\path[fill=fillColor] ( 38.41, 45.23) --
	( 41.65, 45.23) --
	( 41.65, 48.47) --
	( 38.41, 48.47) --
	cycle;

\path[fill=fillColor] ( 71.16, 48.77) --
	( 74.40, 48.77) --
	( 74.40, 52.01) --
	( 71.16, 52.01) --
	cycle;

\path[fill=fillColor] (103.91, 54.79) --
	(107.15, 54.79) --
	(107.15, 58.03) --
	(103.91, 58.03) --
	cycle;

\path[fill=fillColor] (136.66, 68.02) --
	(139.90, 68.02) --
	(139.90, 71.26) --
	(136.66, 71.26) --
	cycle;

\path[fill=fillColor] (169.41, 92.15) --
	(172.65, 92.15) --
	(172.65, 95.39) --
	(169.41, 95.39) --
	cycle;

\path[fill=fillColor] (202.16,121.62) --
	(205.40,121.62) --
	(205.40,124.86) --
	(202.16,124.86) --
	cycle;
\definecolor{drawColor}{RGB}{59,82,139}

\path[draw=drawColor,line width= 0.6pt,line join=round,line cap=round] ( 37.31,165.78) -- ( 51.08,165.78);
\definecolor{drawColor}{RGB}{33,145,140}

\path[draw=drawColor,line width= 0.6pt,line join=round,line cap=round] ( 37.31,156.60) -- ( 51.08,156.60);
\definecolor{drawColor}{RGB}{249,142,9}

\path[draw=drawColor,line width= 0.6pt,line join=round,line cap=round] ( 37.31,147.42) -- ( 51.08,147.42);
\definecolor{fillColor}{RGB}{59,82,139}

\path[fill=fillColor] ( 44.20,165.78) circle (  1.72);
\definecolor{fillColor}{RGB}{33,145,140}

\path[fill=fillColor] ( 44.20,159.28) --
	( 46.52,155.26) --
	( 41.88,155.26) --
	cycle;
\definecolor{fillColor}{RGB}{249,142,9}

\path[fill=fillColor] ( 42.48,145.70) --
	( 45.92,145.70) --
	( 45.92,149.14) --
	( 42.48,149.14) --
	cycle;
\definecolor{drawColor}{RGB}{0,0,0}

\node[text=drawColor,anchor=base west,inner sep=0pt, outer sep=0pt, scale=  0.77] at ( 57.97,163.15) {$K=5$};

\node[text=drawColor,anchor=base west,inner sep=0pt, outer sep=0pt, scale=  0.77] at ( 57.97,153.97) {$K=10$};

\node[text=drawColor,anchor=base west,inner sep=0pt, outer sep=0pt, scale=  0.77] at ( 57.97,144.79) {$K=15$};
\end{scope}
\end{tikzpicture}

\caption{ \small{ Empirical power of the test $\psi_n$ for Example \ref{location} (left panel) and Example \ref{mixture} (right panel). The dashed lines show the nominal level $\alpha=0.05$. } }
\label{fig:power}
\end{figure}

\normalsize

\small

\bibliographystyle{abbrvnat}
\bibliography{bibliography}

\appendix

\section{Motif Tables}
\label{app:tables}

In this section, we collect the coefficients appearing in the motif expansions used in the proofs, together with elementary bounds on the corresponding motif counts.

\begin{itemize}

\item[$(i)$] \cref{tab:varV} lists the coefficients $c_H(u_n)$ in the motif expansion \eqref{eq:VarVmotif1} of $\Var{V_n}$, used in the proof of \cref{lemma:Varlemma}.

\item[$(ii)$] \cref{tab:third} lists the coefficients $c_H^{(3)}(u_n)$ in the motif expansion
\eqref{eq:H_moment_expansion} of $\E{S_n^3}$.

\item[$(iii)$] \cref{tab:fourth} lists the coefficients $c_H^{(4)}(u_n)$ in the motif expansion \eqref{eq:fourth_moment_expansion} of $\E{S_n^4}$.

\end{itemize}

In each table, the last column gives an asymptotic bound on the corresponding motif count, obtained using the identities $\sum_{i=1}^n  d_{G_n}(i) = 2\NH{\Ktwo}$ and $\sum_{i=1}^n \binom{ d_{G_n}(i)}{2} = \NH{\Pthree}$ and elementary counting arguments.

\begin{center}
\small
    \renewcommand{\arraystretch}{2}
    \begin{longtable}{|l|l|l|}
        \caption{Coefficients $c_H(u_n)$ in \eqref{eq:VarVmotif1} and the order of the motif counts.  }
        \label{tab:varV}\\
        \hline
        \textbf{Motif $H$} & \textbf{Coefficient $c_H(u_n)$} & \textbf{Order of $\NH{H}$}  \\
        \hline
        \endfirsthead
        \hline
        \textbf{Motif $H$} & \textbf{Coefficient $c_H(u_n)$} & \textbf{Order of $\NH{H}$}  \\
        \hline
        \endhead
        \hline
        \endfoot
        \hline
        \endlastfoot
            $K_2$ (\Ktwo) & $\displaystyle \frac{u_n^2(1-u_n^2)}{4}$ & $O(\NH{\Ktwo})$ \\ \hline
            $P_3$ (\Pthree) & $\displaystyle \frac{(1-u_n)(14u_n^3 - 2u_n^2 + u_n + 1)}{4}$ & $O(\NH{\Pthree})$ \\ \hline
            $K_3$ (\Kthree) & $- 3 u_n(1-u_n)^2(3 u_n + 1)$ & $O(\NH{\Pthree})$ \\ \hline
            $P_4$ (\Pfour) & $\displaystyle -\frac{5u_n^2(1-u_n)^2}{2}$ & $O(\NH{\Ktwo} \NH{\Pthree})$ \\ \hline
            $K_{1,3}$ (\Konethree) & $\displaystyle \frac{3 u_n(1-u_n)(6 u_n^2 - u_n + 1)}{2}$ & $O(\NH{\Ktwo} \NH{\Pthree})$ \\ \hline
            $P_5$ (\Pfive) & $\displaystyle \frac{u_n(1-u_n)^3}{2}$ & $O(\NH{\Ktwo} \NH{\Pthree})$ \\ \hline
            $K_{1,4}$ (\Konefour) & $\displaystyle 6u_n^3(1-u_n)$ & $O(\NH{\Pthree}^2)$ \\ \hline
            $T_{3,2}$ (\Tthreetwo) & $\displaystyle -u_n^2(1-u_n)^2$ & $O(\NH{\Ktwo}  \NH{\Pthree})$ \\ \hline
            $C_4$ (\Cfour) & $\displaystyle (1-u_n)^2(1 + 2 u_n - 2u_n^2)$ & $O(\NH{\Ktwo} \NH{\Pthree})$ \\ \hline
            $\mathrm{paw}$ (\Paw) & $\displaystyle -3 u_n^2(1-u_n)^2$ & $O(\NH{\Ktwo} \NH{\Pthree})$\\ \hline
    \end{longtable}
\end{center}

\begin{center}
\small
    \renewcommand{\arraystretch}{2}
    \begin{longtable}{|l|l|l|}
        \caption{Coefficients $c_H^{(3)}(u_n)$ in the expansion of $\E{S_n^3}$ and the order of the motif counts.}
        \label{tab:third}\\
        \hline
        \textbf{Motif $H$} & \textbf{Coefficient $c_H^{(3)}(u_n)$} & \textbf{Order of $\NH{H}$}  \\
        \hline
        \endfirsthead
        \hline
        \textbf{Motif $H$} & \textbf{Coefficient $c_H(u_n)$} & \textbf{Order of $\NH{H}$}  \\
        \hline
        \endhead
        \hline
        \endfoot
        \hline
        \endlastfoot
            $K_2$ (\Ktwo) & $\displaystyle -\frac{u_n(1-u_n^2)}{4}$ & $O(\NH{\Ktwo})$ \\ \hline
            $P_3$ (\Pthree) & $\displaystyle -\frac{3u_n^2(1-u_n)}{2}$ & $O(\NH{\Pthree})$ \\ \hline
            $K_3$ (\Kthree) & $\displaystyle \frac{3(1-u_n)^2(1+2 u_n)}{4}$ & $O(\NH{\Pthree})$ \\ \hline
            $P_4$ (\Pfour) & $\displaystyle \frac{3 u_n (1- u_n)^2}{4}$ & $O(\NH{\Ktwo} \NH{\Pthree})$ \\ \hline
            $K_{1,3}$ (\Konethree) & $\displaystyle -\frac{3u_n^2(1-u_n)}{2}$ & $O(\NH{\Ktwo} \NH{\Pthree})$ \\ \hline  
    \end{longtable}
\end{center}

\begin{center}
\small
\renewcommand{\arraystretch}{2}
    \begin{longtable}{|l|l|l|}
        \caption{Coefficients $c_H^{(4)}(u_n)$ in the expansion of $\E{S_n^4}$ and the order of the motif counts.}
        \label{tab:fourth}\\
        \hline
        \textbf{Motif $H$} & \textbf{Coefficient $c_H^{(4)}(u_n)$} & \textbf{Order of $\NH{H}$}  \\
        \hline
        \endfirsthead
        \hline
        \textbf{Motif $H$} & \textbf{Coefficient $c_H(u_n)$} & \textbf{Order of $\NH{H}$}  \\
        \hline
        \endhead
        \hline
        \endfoot
        \hline
        \endlastfoot
        $K_2$ (\Ktwo) & $\displaystyle \frac{(1-u_n^2)(1+3u_n^2)}{16}$ & $O(\NH{\Ktwo})$ \\ \hline
        $P_3$ (\Pthree) & $\displaystyle \frac{(1-u_n)(21u_n^3 - 3u_n^2 + 7u_n + 3)}{8}$ & $O(\NH{\Pthree})$ \\ \hline
        $2K_2$ (\Ktwo$\sqcup$\Ktwo) & $\displaystyle \frac{3(1-u_n^2)^2}{8}$ & $O(\NH{\Ktwo}^2)$ \\ \hline
        $K_3$ (\Kthree) & $\displaystyle -\frac{9u_n(1-u_n)^2(3u_n+1)}{4}$ & $O(\NH{\Pthree})$ \\ \hline
        $P_4$ (\Pfour) & $\displaystyle -\frac{3u_n(1-u_n)^2(2u_n-1)}{2}$ & $O(\NH{\Ktwo} \NH{\Pthree})$ \\ \hline
        $K_{1,3}$ (\Konethree) & $\displaystyle \frac{9u_n(1-u_n)(3u_n^2+1)}{4}$ & $O(\NH{\Ktwo} \NH{\Pthree})$ \\ \hline
        $K_2\sqcup P_3$ (\Ktwo$\sqcup$\Pthree) & $\displaystyle \frac{3u_n(1-u_n^2)(1-u_n)}{4}$ & $O(\NH{\Ktwo} \NH{\Pthree})$ \\ \hline
        $P_5$ (\Pfive) & $\displaystyle \frac{3u_n(1-u_n)^2}{2}$ & $O(\NH{\Ktwo} \NH{\Pthree})$ \\ \hline
        $K_{1,4}$ (\Konefour) & $\displaystyle \frac{3u_n^2(1-u_n)(3u_n+1)}{2}$ & $O(\NH{\Pthree}^2)$ \\ \hline
        $T_{3,2}$ (\Tthreetwo) & $\displaystyle -\frac{3u_n^2(1-u_n)^2}{2}$ & $O(\NH{\Ktwo}  \NH{\Pthree})$ \\ \hline
        $C_4$ (\Cfour) & $\displaystyle \frac{3(1-u_n)^2(1+2u_n-u_n^2)}{2}$ & $O(\NH{\Ktwo} \NH{\Pthree})$ \\ \hline
        $\mathrm{paw}$ (\Paw) & $\displaystyle -3u_n^2(1-u_n)^2$ & $O(\NH{\Ktwo} \NH{\Pthree})$ \\ \hline
        $2P_3$ (\Pthree$\sqcup$\Pthree) & $\displaystyle \frac{3u_n^2(1-u_n)^2}{2}$ & $O(\NH{\Pthree}^2)$\\ \hline
    \end{longtable}
\end{center}

\normalsize

\section{Universal Consistency}

In this section we establish universal consistency of the undirected $K$-NN test. 

\begin{prop}\label{ppn:KNNconsistency}
Let $T^{\times}(\mathcal G_K(\mathcal Z_{n}))$ be as defined in \eqref{eq:Tpoissonc}. Suppose $\lambda_n := \frac{n_1}{n}  \rightarrow \lambda  \in  (0, 1)$ and either $K \geq 1$ fixed or $K_n \rightarrow \infty$ with $K_n = o(n).$
Then,  as $n\to\infty$,
\begin{equation}\label{eq:Tfg}
\frac{T^{\times}(\mathcal G_K(\mathcal Z_{n}))}{ \NHK{\Ktwo} }  \stackrel{P} \rightarrow \delta_\lambda(f, g) :=  \int \frac{ 2 \lambda (1-\lambda)  f(x) g(x)}{ \lambda f(x) + (1-\lambda) g(x) } \mathrm {d} x  .  
\end{equation}
Further, $\delta_\lambda(f, g) \leq \delta_\lambda(f, f)= 2 \lambda (1-\lambda)$ and equality holds if and only if $f=g$ almost everywhere.    
\end{prop}

The proof, given in \cref{sec:consistencypf}, works directly with Lebesgue points of $f$ and $g$ and, unlike in \cite{Kanekar2025}, requires no additional continuity, smoothness, or support assumptions on the densities.  

Recall that  we can construct a universally consistent two-sample test with precise control on the asymptotic level as follows: Fix $\alpha \in (0, 1)$ and consider the test function that rejects the null hypothesis $H_0$ as follows:  
\begin{align}  
\label{eq:Knnthreshold1}
\psi_n := \bm 1\{ T^{\times}(\mathcal G_K(\mathcal Z_{n})) < \mu_{\mathscr F}^\times - z_{\alpha}\sigma_{\mathscr F} \},
\end{align}
where $z_\alpha$ is the $(1-\alpha)$-th quantile of the standard normal distribution. The following result is a direct corollary of the above results.

\begin{corollary}\label{cor:levelconsistency}
Fix $\alpha \in (0, 1)$ and let $\psi_n$ be as defined above. Suppose either $K \geq 1$ is fixed or $K_n \rightarrow \infty$ with $K_n = o(n).$ Then the following hold, whenever 
$\lambda_n= \frac{n_1}{n}  \rightarrow \lambda  \in  (0, 1) \backslash \{ \frac{1}{2}\}$ or $n_1 = n_2$: 
\begin{itemize}
\item[$(i)$] (Asymptotic Type I error control) $\lim_{n \rightarrow \infty}\mathbb P_{H_0}(\psi_n=1) = \alpha$
\item[$(ii)$] (Universal consistency) For any two distributions $f \ne g$ in $H_1$, $\lim_{n \rightarrow \infty}\mathbb P_{H_1}(\psi_n=1) = 1$. 
\end{itemize}
\end{corollary}

The proof of \cref{cor:levelconsistency} is given in \cref{sec:levelconsistencypf}.

\subsection{Proof of \cref{ppn:KNNconsistency} }
\label{sec:consistencypf}

Let $\mathcal P_1$ denote a homogeneous Poisson process of unit intensity in $\R^d$, and let $\mathcal P_{n\phi_n}$ denote an inhomogeneous Poisson process in $\R^d$ with intensity function
$n\phi_n$, where $\phi_n(z)=\lambda_n f(z)+(1-\lambda_n)g(z)$. For a finite set $S\subset\R^d$, we denote by $\mathcal P_1^S$ and $\mathcal P_{n\phi_n}^S$ the respective point processes obtained by adding to them the points of $S$. The proof of \cref{ppn:KNNconsistency} will require the following result about the limiting value of the expected degree of a vertex in a $K$-NN graph, in the diverging $K$ regime.

\begin{lemma}
\label{lem:degree} 
Suppose $K = K_n \to\infty$ such that $K=o(n)$. Then, for almost every $z \in \R^d$ with $\phi (z) := \lambda f(z) + (1-\lambda) g(z) > 0$,   \begin{equation}\label{eq:degreeG}
\frac{1}{K} \E{d_{\mathcal G_K(\mathcal P_{n\phi_n}^{z})}(z)} \rightarrow 1. 
\end{equation}
\end{lemma}

The proof of \cref{lem:degree} is given later in \cref{sec:nearestneigborpf}. We first proceed with the proof of \cref{ppn:KNNconsistency}, assuming this result. We begin by computing the expected value of $T^{\times}(\mathcal G_K(\mathcal Z_{n}))$. To this end, define,  
$$h_n(x, y)=\frac{n_1n_2 f(x)g(y)}{(n_1f(x)+n_2g(x))(n_1f(y)+n_2g(y))}$$
and 
\begin{align}\label{eq:h}
h(x, y)=\frac{ \lambda (1-\lambda) f(x)g(y)}{( \lambda f(x)+ (1- \lambda) g(x))( \lambda f(y)+ (1- \lambda) g(y))} . 
\end{align} 
Note that $\|h_n-h\|_\infty \rightarrow 0$.  Conditionally on $\mathscr F$, the labels are independent, and
$h_n(Z_i,Z_j)$ is the probability that $Z_i$ has label $1$
and $Z_j$ has label $2$. Therefore,
\begin{align}
\E{T^{\times}(\mathcal G_K(\mathcal Z_n)) | \mathscr F} = \sum_{1\le i\ne j\le L_n} h_n(Z_i,Z_j) \bm 1 \{\{ Z_i,Z_j\}\in E(\mathcal G_K(\mathcal Z_n))\}.
\label{eq:ETn}
\end{align}
Note that 
\begin{align}
& \frac{1}{nK} \left| \sum_{1\le i\ne j\le L_n} \{h_n(Z_i,Z_j)-h(Z_i,Z_j)\}
\mathbf 1 \{\{Z_i,Z_j\}\in E(\mathcal G_K(\mathcal Z_n)) \}
\right|  \nonumber  \\  
& \leq 2\|h_n-h\|_\infty \frac{|E(\mathcal G_K(\mathcal Z_n))|}{nK}  \nonumber \\ 
& \leq 2\|h_n-h\|_\infty\frac{L_n}{n} =o_{L^1}(1)  ,  
\label{eq:ETnF}
\end{align}
where we used $|E(\mathcal G_K(\mathcal Z_n))|\le KL_n$ and $\mathbb E L_n=n$. Hence, by \eqref{eq:ETn}, 
\begin{align*}
\frac{1}{nK} \E{T^{\times}(\mathcal G_K(\mathcal Z_n)) | \mathscr F } & = \frac{1}{nK} \sum_{1\le i\ne j\le L_n} h(Z_i,Z_j) \bm 1 \{\{Z_i,Z_j\}\in E(\mathcal G_K(\mathcal Z_n)) \} +o_{L^1}(1).
\end{align*}
The following lemma computes the limit of the expected value of the sum above.

\begin{lemma}\label{lem:expectationh}
Suppose either $K$ is fixed or $K=K_n\to\infty$ with $K_n=o(n)$. Let $H : \mathbb R^d\times\mathbb R^d\to\mathbb R$ be a jointly measurable and bounded function, such that for almost every $z \in \R^d$, $H(z, \cdot)$ is measurable and $z$ is a Lebesgue point of the function $\phi(\cdot) H(z, \cdot)$. Then 
\begin{align*}
&\frac{1}{nK} \E{ \sum_{1\le i\ne j\le L_n} H(Z_i,Z_j) \mathbf 1 \{\{Z_i,Z_j\}\in E(\mathcal G_K(\mathcal Z_n)) \} }  \rightarrow \gamma\int_{\mathbb R^d}H(z,z)\phi(z) \mathrm{d} z . 
\end{align*} 
\begin{align*}
\gamma= \begin{cases}
\frac{1}{K} \E{d_{\mathcal G_K(\mathcal P_1^0)}(0) } , &\text{if $K$ is fixed},\\ 
1,
&\text{if $K=K_n\to\infty$}.
\end{cases}
\end{align*}  
\label{eq:expectationconstant} 
\end{lemma}

\begin{proof} 
By the Palm theory for Poisson processes \cite[Theorem 1.6]{geometricgraph}, 
\begin{equation*}
\E{\sum_{1\le i\ne j\le L_n} H(Z_i,Z_j) \mathbf 1 \{\{Z_i,Z_j\}\in E(\mathcal G_K(\mathcal Z_n)) \} } = n\int_{\mathbb R^d}w_n(z)\phi_n(z) \mathrm{d} z , 
\label{eq:poissonexpectation}
\end{equation*}
where, for $z\in\mathbb R^d$, 
$$w_n(z) := \E{ \sum_{z'\in\mathcal P_{n\phi_n}} H(z,z') \mathbf 1 \{\{z,z'\}\in
E(\mathcal G_K(\mathcal P_{n\phi_n}^{z})) \}}. $$
Note that the geometric degree bound gives $|w_n(z)|\le C_dK\|H\|_\infty$. Also, for all sufficiently large $n$, $\phi_n\le C_\lambda\phi$. Therefore, by dominated convergence, it suffices to prove
\begin{equation}
\frac{w_n(z)}{K} \rightarrow \gamma H(z,z) , 
\label{eq:Kh}
\end{equation}  
for almost every $z \in \R^d$ with $\phi (z) > 0$. Fix a common Lebesgue point $z$ of $f$ and $g$ such that $\phi(z)>0$. Fix a sufficiently large constant $A$, and choose $r_n$ so that 
$$n\phi_n(z)|B(z,r_n)|=AK , $$ 
where $B(z,r_n)$ is the $d$-dimensional Euclidean ball with center at $z$ and radius $r_n$. Since $K=o(n)$, we have $r_n\to0$. Cover the directions from $z$
by a finite collection of cones $\mathscr C$ each of angular diameter strictly less than
$\frac{\pi}{3}$. Fix one cone $\mathcal C_j \in \mathscr C$. By invariance under dilation,
$$|\mathcal C_j \cap B(z,r)|=c_j |B(z,r)|,$$
where $c_j:=\frac{| \mathcal C_j \cap B(z,1)|}{|B(z,1)|}>0$. Define 
$$N_j := | \mathcal P_{n\phi_n} \cap \mathcal C_j \cap B(z,r_n)| \sim \mathrm{Poi}\left( n\int_{\mathcal C_j \cap B(z,r_n)}\phi_n(y) \mathrm{d} y \right).$$ Hence, by the Lebesgue differentiation theorem, 
\begin{align*}
\left| \int_{\mathcal C_j \cap B(z,r_n)}  (\phi_n(y)-\phi_n(z) ) \mathrm{d} y \right| \le \int_{B(z,r_n)}
|\phi_n(y)-\phi_n(z)| \mathrm{d} y =o(r_n^d).
\end{align*}
Therefore, using the choice
$n\phi_n(z)|B(z,r_n)|=AK$, we obtain
\begin{align*}
\E{N_j} = n\int_{\mathcal C_j \cap B(z,r_n)}\phi_n(y) \mathrm{d} y
&=n\phi_n(z)| \mathcal C_j \cap B(z,r_n)|+n o(r_n^d)\\
&=c_j n\phi_n(z)|B(z,r_n)|+n o(r_n^d)\\
&=(c_jA+o(1))K  ,  
\end{align*}
using $\phi_n(z)\to\phi(z)>0$ and $nr_n^d=O(K)$, which justifies $n o(r_n^d)=o(K)$. Note that the constants $c_j$ may differ between cones. However, there are only finitely many cones, so the minimum of these constants is positive. Hence, by choosing $A$ sufficiently large, we can make $\E{N_j} \geq \frac{1}{2} c_j A K \geq 2K$, for all $\mathcal C_j \in \mathscr C$, for sufficiently large $n$. Hence, by Chebyshev's inequality and a union bound, for sufficiently large $n$, 
\begin{align}
\mathbb P( \exists ~ \mathcal C_j \in \mathscr C: N_j < K) \leq \sum_{\mathcal C_j \in \mathscr C} \mathbb P(N_j < K) \leq \sum_{\mathcal C_j \in \mathscr C} \frac{\E{N_j}}{(\E{N_j}-K)^2}  & \leq \sum_{\mathcal C_j \in \mathscr C} \frac{4}{\E{N_j}} \nonumber \\ 
 & \leq \frac{C_d}{A K }  ,  
\label{eq:C} 
\end{align}
with $C_d = 8 \sum_{\mathcal C_j \in \mathscr C} \frac{1}{c_j }$. Let $\mathcal E_n = \{N_j \geq K \text{ for all } \mathcal C_j \in \mathscr C\}$ be the event that every cone contains at least $K$ points in $B(z, r_n)$. Note that 
\begin{align*}
& \left| w_n(z) -H(z,z) \E{d_{\mathcal G_K(\mathcal P_{n\phi_n}^{z})}(z)} \right|  \nonumber \\  
& \le \E{\sum_{z'\in\mathcal P_{n\phi_n}} |H(z,z')-H(z,z)| \bm 1\{ \{z, z'\} \in E(\mathcal G_K (\mathcal P_{n \phi_n}^{z})) \}}  .  
\end{align*}
Note that on $\mathcal E_n$, every undirected neighbor of $z$ lies in $B(z,r_n)$. Also, on $\mathcal E_n^c$, each summand above is at most $2\|H\|_\infty$ and the degree bound gives at most $B_dK$ terms. Hence, 
\begin{align*}
& \left| w_n(z) -H(z,z) \E{d_{\mathcal G_K(\mathcal P_{n\phi_n}^{z})}(z)} \right|  \nonumber \\ 
&\le \E{ \sum_{z'\in\mathcal P_{n\phi_n}\cap B(z,r_n)} |H(z,z')-H(z,z)| } +  2B_dK\|H\|_\infty  \mathbb P(\mathcal E_n^c)  .  
\end{align*}
Dividing by $K$, we conclude that
\begin{align}
& \frac{1}{K} \left| w_n(z) -H(z,z)  \E{d_{\mathcal G_K(\mathcal P_{n\phi_n}^{z})}(z)} \right|  \nonumber  \\  
&  \le \frac{n}{K} \int_{B(z,r_n)} |H(z,y)-H(z,z)|\phi_n(y) \mathrm{d} y + 2C_d\|H\|_\infty \mathbb P(\mathcal E_n^c).
\label{eq:Hdegree}
\end{align}
Using $\phi_n\le C_\lambda\phi$ and the Lebesgue differentiation theorem 
$$\int_{B(z,r_n)} |H(z,y)-H(z,z)|\phi_n(y) \mathrm{d} y = o(r_n^d).$$
Hence, for fixed $A$, the first term in \eqref{eq:Hdegree} goes to zero as $n \rightarrow \infty$, since $n o(r_n^d)=o(K)$. The second term in \eqref{eq:Hdegree} also goes to zero as $n\to\infty$ followed by $A \rightarrow \infty$, by using \eqref{eq:C}. Hence, 
\begin{equation*}
\frac{1}{K} w_n(z) - H(z,z)  \frac{1}{K} \E{d_{\mathcal G_K(\mathcal P_{n\phi_n}^{z})}(z)} \rightarrow 0.
\end{equation*}
Now, for $K$ fixed the result in \eqref{eq:Kh} follows by invoking standard results of stabilization of the $K$-NN graph (see, for example, \cite[Lemma 3.2]{geometricprobability}). For $K \rightarrow \infty$, the result in \eqref{eq:Kh} follows by invoking \cref{lem:degree}. This completes the proof of \cref{lem:expectationh}. 
\end{proof}

By \eqref{eq:ETnF} and \cref{lem:expectationh} with $H=h$ as in \eqref{eq:h}, 
\begin{align}\label{eq:ETnfg}
\frac{1}{nK} \E{T^{\times}(\mathcal G_K(\mathcal Z_n)) } & \rightarrow \gamma \int \frac{ \lambda (1-\lambda) f(x)g(x)}{\lambda f(x)+ (1- \lambda) g(x) } \mathrm d x .
\end{align}
The variance bound in the next lemma shows that this convergence holds in $L^2$.

\begin{lemma}\label{lem:varianceU}
Suppose either $K$ is fixed or $K=K_n \to \infty$. Then 
$$\Var{T^\times(\mathcal G_K(\mathcal Z_n))} \le C_d nK^2 ,$$ 
for some constant $C_d >0$, depending on $d$.   
\end{lemma}

\begin{proof} 
Let $B_dK$ be the geometric upper bound on every vertex degree as in \eqref{eq:vertexdegree}. Note that adding one marked point introduces at most $B_dK$ edges incident to the new point. Further, an existing vertex changes its nearest-neighbor list only if it selects the new point, and each
such vertex loses at most one previous neighbor. There are at most $B_dK$ such vertices, so at most $B_dK$ old edges disappear. Hence, adding or deleting one marked point changes the cross-edge count by at most $2B_dK$. Consequently, replacing one marked observation, including its location and label, changes the cross-edge count by at most $4B_dK$.

Condition on $L_n=\ell$, denote $T=T^{\times}(\mathcal G_K(\mathcal Z_n))$ and  $T^{(i)}=T^{\times}(\mathcal G_K(\mathcal Z_n^{(i)}))$, where $\mathcal Z_n^{(i)}$ is obtained from 
$\mathcal Z_n$ by replacing the $i$-th marked observation with an independent copy and rebuilding the graph. Then by the Efron--Stein inequality \cite{es}, 
\begin{align}\label{eq:varTn}
\Var{ T | L_n=\ell } \le \frac{1}{2} \sum_{i=1}^{\ell} \E{(T-T^{(i)})^2 | L_n=\ell }  .  
\end{align}
Replacing one marked observation can be performed by first deleting the original point and then inserting its replacement. Each operation changes the cross-edge count by at most $2B_dK$, as discussed above. Hence, by triangle inequality, $|T-T^{(i)}|\le 4B_dK$. Therefore, 
$$\E{(T-T^{(i)})^2 | L_n=\ell} \le (4B_dK)^2.$$
Applying this bound in \eqref{eq:varTn} gives, 
\begin{align*}
\Var{T^\times(\mathcal G_K(\mathcal Z_n)) | L_n=\ell} &\le \frac{1}{2} \sum_{i=1}^{\ell}(4B_dK)^2 =8B_d^2\ell K^2.
\end{align*}
Since $\mathbb E L_n=n$, it follows that
\begin{equation}
\E{\Var{T^\times(\mathcal G_K(\mathcal Z_n))\mid L_n}} \le 8B_d^2nK^2.
\label{eq:KNNvariancemean}
\end{equation}

To control the variance of the conditional mean, define
$$a_\ell:= \E{T^\times(\mathcal G_K(\mathcal Z_n))\mid L_n=\ell}  ,  $$
for $\ell\ge0$. By the same argument as above, $|a_{\ell+1}-a_\ell|\le 2B_dK$.
Therefore, $|a_{\ell} - a_{\ell'}|\le 2B_dK|\ell-\ell'|$, for $\ell, \ell' \ge0$. 
Let $L_n'$ be an independent copy of $L_n$. Since $L_n\sim\operatorname{Poi}(n)$,
\begin{align}
\Var{a_{L_n}} =\frac{1}{2} \E{(a_{L_n}-a_{L_n'})^2} \le \frac{(2B_dK)^2}{2} \E{(L_n-L_n')^2} 
&=4B_d^2K^2\Var{L_n}  \nonumber \\  
& =4B_d^2nK^2.
\label{eq:KNNmeanvariance}
\end{align}
Combining \eqref{eq:KNNvariancemean} and \eqref{eq:KNNmeanvariance}, the result in \cref{lem:varianceU} follows.  
\end{proof}

Combining \eqref{eq:ETnfg} with \cref{lem:varianceU} gives, 
\begin{align*}
\frac{1}{nK} T^{\times}(\mathcal G_K(\mathcal Z_n)) & \stackrel{P} \rightarrow \gamma \int \frac{ \lambda (1-\lambda) f(x)g(x)}{ \lambda f(x)+ (1- \lambda) g(x) } \mathrm d x .
\end{align*}
Also, applying \cref{lem:expectationh} with $H=1$ and the same variance argument as in \cref{lem:varianceU} gives, 
$$\frac{\NHK{\Ktwo}}{nK} \stackrel{P} \rightarrow \frac{\gamma}{2}.$$
Combining the above proves the result in \eqref{eq:Tfg}.

Finally, the fact that $\delta_\lambda(f, g) \leq \delta_\lambda(f, f)= 2 \lambda (1-\lambda)$ and equality holds if and only if $f=g$, follows from \cite[Theorem 1 and Corollary 1]{GyorfiNemetz1977}. This completes the proof of \cref{ppn:KNNconsistency}.  \hfill  $\Box$

\subsection{Proof of \cref{lem:degree}}
\label{sec:nearestneigborpf}

The lower bound in \eqref{eq:degreeG} is immediate: Since the degree of any vertex is at least the outgoing degree, 
$$\E{d_{\mathcal G_K(\mathcal P_{n\phi_n}^{z})}(z)} \geq \E{\min\{K, |\mathcal P_{n\phi_n}| \}}  = K-o(K),$$ 
since $|\mathcal P_{n\phi_n}|\sim\operatorname{Pois}(n)$ and $K=o(n)$.  

For the upper bound, fix $\varepsilon>0$, and now choose $r_n$ such that 
\begin{align}\label{eq:degreeradius}
n\phi_n(z)|B(z,r_n)|=(1+\varepsilon)K.
\end{align}
 Choose $\delta>0$ sufficiently small that $(1+\varepsilon)(1-2\delta)^d > 1+ \frac{\varepsilon}{2}$, and choose a finite $\delta$-net $\{u_s\}_{s\in \mathcal S}$ of the unit sphere. Fix a common Lebesgue point $z$ of $f$ and $g$ (so that $\int_{B(z,r)}|\phi_n(y)-\phi_n(z)|\,\mathrm dy=o(r^d)$ uniformly in $n$) with $\phi(z) > 0$ and define
$$N_s:=|B (z+r_nu_s, (1-2\delta)r_n ) \cap \mathcal P_{n\phi_n}^{z} |, $$
the number of points of the Poisson process $\mathcal P_{n\phi_n}^{z}$ in the ball $B (z+r_nu_s, (1-2\delta)r_n )$. (Note that the inserted point $z$ lies outside this ball $B (z+r_nu_s, (1-2\delta)r_n )$). By the Lebesgue differentiation theorem and \eqref{eq:degreeradius},
\begin{align*}
\E{N_s} = n \int_{B (z+r_nu_s,(1-2\delta)r_n )} \phi_n(y) \mathrm d y =  ((1+\varepsilon)(1-2\delta)^d+o(1) )K \geq   \left( 1+ \frac{\varepsilon}{2}  \right) K,  
\end{align*}
for sufficiently large $n$ and all $s \in \mathcal S$. Then Chebyshev's
inequality gives
\begin{align*}
\mathbb P(N_s<K+1) =\mathbb P(N_s\le K) \le \mathbb P (|N_s- \E{N_s} |\ge \E{N_s} -K ) &\le \frac{\E{N_s}}{(\E{N_s}-K)^2}  \nonumber  \\  
&\le \frac{4\E{N_s}}{\varepsilon^2K^2} = O\left(  \frac{1}{K}  \right)  .  
\end{align*}
Also, define $N_0:=|B(z,r_n)\cap\mathcal P_{n\phi_n}|$. Note that the inserted point $z$ is excluded from this count. Applying Lebesgue differentiation theorem and \eqref{eq:degreeradius} again gives
\begin{align}\label{eq:ENK}
\mathbb E N_0 =n\int_{B(z,r_n)}\phi_n(y)  \mathrm{d} y =(1+\varepsilon+o(1))K. 
\end{align}
Again, by Chebyshev's inequality,  $\mathbb P(N_0<K) \le \frac{\mathbb E N_0}{(\mathbb E N_0-K)^2} = O\left(  \frac{1}{K}  \right) $. 
Define 
$$\mathcal E_n:= \{N_0 \geq K \} \bigcap \left( \bigcap_{s\in\mathcal S}\{N_s \geq K+1\} \right)  .  $$
Since $\mathcal S$ is finite and $K = K_n \to\infty$, a union bound gives, 
\begin{align*}
\mathbb P\left( \mathcal E_n^c \right) \le \mathbb P(N_0<K)  
+\sum_{s\in\mathcal S}\mathbb P(N_s<K+1) = O\left(  \frac{1}{K}  \right) \rightarrow 0.
\end{align*}  
Note that on the event $\mathcal E_n$ every undirected neighbor of $z$ lies in $B(z, r_n)$, hence the degree in this case is at most $N_0$. Then using the degree bound \eqref{eq:vertexdegree} on the event $\mathcal E_n^c$ gives, 
\begin{align*}
\E{ d_{\mathcal G_K(\mathcal P_{n\phi_n}^{z})} (z) } \leq \E{N_0} + B_d K \mathbb P(\mathcal E_n^c) \le n\int_{B(z,r_n)}\phi_n(y) \mathrm{d} y+O(1) =(1+\varepsilon+o(1))K  ,  
\end{align*}
where the last step uses \eqref{eq:ENK}. Then dividing both sides by $K$ and taking limits, the result follows.  \hfill $\Box$

\subsection{Proof of \cref{cor:levelconsistency} }
\label{sec:levelconsistencypf}

The result in (1) is immediate from \cref{thm:cltH0}. 

For the consistency result, we first show that
\begin{equation}
\frac{\sigma_{\mathscr F}}{\NHK{\Ktwo}} = O_{P}\left(\frac{1}{\sqrt{n}} \right).
\label{eq:meanvariance}
\end{equation}
For this, recalling \eqref{eq:varianceT} and denoting $d_z=d_{\mathcal G_K(\mathcal Z_n)}(z)$ gives, 
\begin{align*}
\sigma_{\mathscr F}^2 \le \frac{1}{4}\left(\NHK{\Ktwo}+\sum_{z\in\mathcal Z_n}\binom{d_z}{2} \right) =\frac18\sum_{z\in\mathcal Z_n}d_z^2.
\end{align*}
Let $\tilde K=\min\{K,\max\{L_n-1,0\}\}$. Then by \cref{lem:NHK}, 
$$\frac{\sigma_{\mathscr F}^2}{\NHK{\Ktwo}^2}  \le  \frac{B_d^2}{2L_n} = O_{P}\left(\frac{1}{n} \right)  ,  $$
since $\frac{L_n}{n}  \to  1$ in probability. This proves \eqref{eq:meanvariance}.  Now fix an alternative with $f$ and $g$ differing on a set of
positive Lebesgue measure. By \cref{ppn:KNNconsistency}, 
$$\frac{T^\times(\mathcal G_K(\mathcal Z_n))}{\NHK{\Ktwo}} \xrightarrow{P}
\delta_\lambda(f,g) < 2\lambda(1-\lambda).$$
Also, $\frac{\mu_{\mathscr F}^{\times}}{\NHK{\Ktwo}} \rightarrow 2\lambda(1-\lambda)$. Combining the above with \eqref{eq:meanvariance} gives, 
\begin{align*}
\frac{\mu_{\mathscr F}^{\times}-z_\alpha\sigma_{\mathscr F}-T^\times(\mathcal G_K(\mathcal Z_n))}{\NHK{\Ktwo}} \xrightarrow{P} 2\lambda(1-\lambda)-\delta_\lambda(f,g)>0.
\end{align*}
Therefore,
$$\mathbb P_{H_1}(\psi_n=1) = \mathbb P_{H_1} \left(T^\times(\mathcal G_K(\mathcal Z_n)) < \mu_{\mathscr F}^{\times}-z_\alpha\sigma_{\mathscr F}
\right) \rightarrow 1.$$
This proves the result in (2).  \hfill  $\Box$

\bigskip

\noindent \textsc{Siva Athreya}, International Centre for Theoretical Sciences, Tata Institute of Fundamental Research, Survey No.~151, Shivakote,
Hesaraghatta Hobli, Bengaluru--560089, India. \\  
\texttt{athreya@icts.res.in}\\[0.4em]

\noindent\textsc{Bhaswar B.~Bhattacharya}, Department of Statistics and Data Science, The Wharton School, University of Pennsylvania, Academic Research Building, 265 South 37th Street, 3rd and 4th Floors, Philadelphia, PA 19104--1686, USA, and Department of Statistics and Data Science, Faculty of Science,
National University of Singapore, Block S16, Level 7, 6 Science Drive 2,
Singapore 117546.  \\  
\texttt{bhaswar.bhattacharya@gmail.com}\\[0.4em]

\noindent\textsc{Yogesh Dahiya}, Department of Mathematical Sciences, Indian Institute of Science Education and Research Mohali, Knowledge City, Sector 81, SAS Nagar, Manauli PO, Punjab 140306, India.  \\  
\texttt{ph20002@iisermohali.ac.in}\\[0.4em]

\noindent\textsc{Karambir Das}, Department of Mathematics,
Indian Institute of Science, Bengaluru 560012, India.  \\  
\texttt{karambirdas@iisc.ac.in}\\[0.4em]

\noindent\textsc{Spandan Ghoshal}, Indian Statistical Institute,
203 B. T. Road, Kolkata 700108, India. \\ 
\texttt{spandan.ghoshal\_r@isical.ac.in}\\[0.4em]

\noindent\textsc{Sourish Maniyar}, Department of Mathematics,
Indian Institute of Science, Bengaluru 560012, India.\\
\texttt{sourishparag@iisc.ac.in}\\[0.4em]

\noindent\textsc{Elizabeth Sara Roy}, International Centre for Theoretical Sciences, Tata Institute of Fundamental Research, Survey No.~151, Shivakote,
Hesaraghatta Hobli, Bengaluru--560089, India.  \\  
\texttt{elizabeth.sararoy@icts.res.in}

\end{document}